\documentclass[letterpaper,11pt]{article}

\usepackage[margin=1in]{geometry}  
\usepackage{xspace,exscale,relsize}
\usepackage{fancybox,shadow}
\usepackage{graphicx}
\usepackage{color}
\usepackage{amsthm,amsfonts}
\usepackage{amssymb}
\usepackage{authblk}
\usepackage[title]{appendix}

\usepackage{extarrows}
\usepackage{comment}

\usepackage[noadjust]{cite}

\usepackage{amsmath}
	\makeatletter
	\let\over=\@@over \let\overwithdelims=\@@overwithdelims
	\let\atop=\@@atop \let\atopwithdelims=\@@atopwithdelims
  	\let\above=\@@above \let\abovewithdelims=\@@abovewithdelims
  	\makeatother
\usepackage{rotating}

\usepackage{ifpdf}

\usepackage{subfigure}
\usepackage{psfrag}

\usepackage{prettyref}
\usepackage{enumitem}

\usepackage{tikz}
\usetikzlibrary{arrows}
\tikzstyle{int}=[draw, fill=blue!20, minimum size=2em]
\tikzstyle{dot}=[circle, draw, fill=blue!20, minimum size=2em]
\tikzstyle{init} = [pin edge={to-,thin,black}]

\usepackage[
            CJKbookmarks=true,
            bookmarksnumbered=true,
            bookmarksopen=true,
            colorlinks=true,
            citecolor=red,
            linkcolor=blue,
            anchorcolor=red,
            urlcolor=blue,
            ]{hyperref}

\usepackage[all]{xy}

\usepackage{mathtools}

\usepackage{ifthen}
\newboolean{aos}
\setboolean{aos}{FALSE}

\ifx\eqref\undefined
	\newcommand{\eqref}[1]{~(\ref{#1})}
\fi
\ifx\mod\undefined
	\def\mod{\mathop{\rm mod}}
\fi

\usepackage{bm}

\newcommand{\norm}[1]{{\left\Vert #1 \right\Vert}}

\def\dim{\mathop{\rm dim}}

\def\exp{\mathop{\rm exp}}

\def\tr{\mathop{\rm tr}}

\def\Var{\mathrm{Var}}
\def\Cov{\mathrm{Cov}}

\def\eqdef{\triangleq}

\newcommand{\Unif}{\mathrm{Unif}}

\newcommand{\reals}{\mathbb{R}}

\newcommand{\Expect}{\mathbb{E}}
\newcommand{\expect}[1]{\mathbb{E}\left[#1\right]}

\newcommand{\prob}[1]{\mathbb{P}\left[#1\right]}
\newcommand{\pprob}[1]{\mathbb{P}[#1]}

\newcommand{\TV}{{\rm TV}}

\newcommand{\expects}[2]{\mathbb{E}_{#2}\left[ #1 \right]}
\newcommand{\diff}{{\rm d}}

\newcommand{\iid}{iid\xspace}
\newcommand{\Law}{\text{Law}}

\newcommand{\Fnorm}[1]{\|#1\|_{\rm F}}
\newcommand{\fnorm}[1]{\|#1\|_{\rm F}}
\newcommand{\pth}[1]{\left( #1 \right)}
\newcommand{\qth}[1]{\left[ #1 \right]}
\newcommand{\sth}[1]{\left\{ #1 \right\}}

\newcommand{\iiddistr}{{\stackrel{\text{\iid}}{\sim}}}

\newcommand{\var}{\Var}
\newcommand\indep{\protect\mathpalette{\protect\independenT}{\perp}}
\def\independenT#1#2{\mathrel{\rlap{$#1#2$}\mkern2mu{#1#2}}}
\newcommand{\Bern}{\text{Bern}}

\newcommand{\iprod}[2]{\langle #1, #2 \rangle}
\newcommand{\Iprod}[2]{\langle #1, #2 \rangle}
\newcommand{\indc}[1]{{\mathbf{1}\left\{{#1}\right\}}}

\definecolor{myblue}{rgb}{.8, .8, 1}
\definecolor{mathblue}{rgb}{0.2472, 0.24, 0.6} 
\definecolor{mathred}{rgb}{0.6, 0.24, 0.442893}
\definecolor{mathyellow}{rgb}{0.6, 0.547014, 0.24}

\newcommand{\blue}{\color{blue}}
\newcommand{\nb}[1]{{\sf\blue[#1]}}

\newcommand{\nbb}[1]{#1}

\newcommand{\E}{\Expect}

\newcommand{\calE}{{\mathcal{E}}}
\newcommand{\calF}{{\mathcal{F}}}

\newcommand{\calH}{{\mathcal{H}}}

\newcommand{\calN}{{\mathcal{N}}}

\newcommand{\diverge}{\to \infty}
\newcommand{\mmse}{\mathsf{mmse}}

\newrefformat{eq}{(\ref{#1})}
\newrefformat{thm}{Theorem~\ref{#1}}
\newrefformat{th}{Theorem~\ref{#1}}
\newrefformat{chap}{Chapter~\ref{#1}}
\newrefformat{sec}{Section~\ref{#1}}
\newrefformat{seca}{Section~\ref{#1}}
\newrefformat{algo}{Algorithm~\ref{#1}}
\newrefformat{fig}{Fig.~\ref{#1}}
\newrefformat{tab}{Table~\ref{#1}}
\newrefformat{rmk}{Remark~\ref{#1}}
\newrefformat{clm}{Claim~\ref{#1}}
\newrefformat{def}{Definition~\ref{#1}}
\newrefformat{cor}{Corollary~\ref{#1}}
\newrefformat{lmm}{Lemma~\ref{#1}}
\newrefformat{prop}{Proposition~\ref{#1}}
\newrefformat{pr}{Proposition~\ref{#1}}
\newrefformat{app}{Appendix~\ref{#1}}
\newrefformat{apx}{Appendix~\ref{#1}}
\newrefformat{ex}{Example~\ref{#1}}
\newrefformat{conj}{Conjecture~\ref{#1}}

\def\unifto{\mathop{{\mskip 3mu plus 2mu minus 1mu%
	\setbox0=\hbox{$\mathchar"3221$}%
	\raise.6ex\copy0\kern-\wd0%
	\lower0.5ex\hbox{$\mathchar"3221$}}\mskip 3mu plus 2mu minus 1mu}}

\ifx\lesssim\undefined
\def\simleq{{{\mskip 3mu plus 2mu minus 1mu%
	\setbox0=\hbox{$\mathchar"013C$}%
	\raise.2ex\copy0\kern-\wd0%
	\lower0.9ex\hbox{$\mathchar"0218$}}\mskip 3mu plus 2mu minus 1mu}}
\else
\def\simleq{\lesssim}
\fi

\ifx\gtrsim\undefined
\def\simgeq{{{\mskip 3mu plus 2mu minus 1mu%
	\setbox0=\hbox{$\mathchar"013E$}%
	\raise.2ex\copy0\kern-\wd0%
	\lower0.9ex\hbox{$\mathchar"0218$}}\mskip 3mu plus 2mu minus 1mu}}
\else
\def\simgeq{\gtrsim}
\fi

\newtheorem{theorem}{Theorem}
\newtheorem{lemma}{Lemma}
\newtheorem{corollary}{Corollary}

\newtheorem{proposition}{Proposition}

\newtheorem{conjecture}{Conjecture}

\theoremstyle{definition}

\newtheorem{remark}{Remark}

\newif\ifmapx
{\catcode`/=0 \catcode`\\=12/gdef/mkillslash\#1{#1}}
\edef\jobnametmp{\expandafter\string\csname embayes2_apx\endcsname}
\edef\jobnameapx{\expandafter\mkillslash\jobnametmp}
\edef\jobnameexpand{\jobname}
\ifx\jobnameexpand\jobnameapx
\mapxtrue
\else
\mapxfalse
\fi

\newcommand{\ER}{Erd\H{o}s--R\'enyi\xspace}

\newcommand{\polylog}{\mathsf{polylog}}

\newcommand{\RGG}{\textsf{RGG}}

\renewcommand{\hat}{\widehat}
\renewcommand{\tilde}{\widetilde}

\def\mmse{\mathrm{mmse}}

\newcommand{\dtest}{d_{\sf test}^*}
\newcommand{\dest}{d_{\sf est}^*}
\newcommand{\KL}{\mathrm{KL}}
\newcommand{\dKS}{d_\mathrm{KS}}

\begin{document}
\ifpdf
\DeclareGraphicsExtensions{.pgf}
\graphicspath{{figures/}{plots/}}
\fi

\title{Random geometric graphs with smooth kernels: sharp detection threshold and a spectral conjecture}
 \author{Cheng Mao, Yihong Wu,  and Jiaming Xu\thanks{
 C.\ Mao is with the School of Mathematics, Georgia Institute of Technology, Atlanta, Georgia, USA
\texttt{cheng.mao@math.gatech.edu}.
 Y.\ Wu is  with the Department of Statistics and Data Science, Yale University, New Haven CT, USA, 
 \texttt{yihong.wu@yale.edu}.
J.\ Xu is with The Fuqua School of Business, Duke University, Durham NC, USA, \texttt{jx77@duke.edu}.
 }}
\maketitle

\begin{abstract}

A random geometric graph (RGG) with kernel $K$ is constructed by first sampling latent points $x_1,\ldots,x_n$ independently and uniformly from the $d$-dimensional unit sphere, then connecting each pair $(i,j)$ with probability $K(\langle x_i,x_j\rangle)$.
We study the sharp detection threshold, namely the highest dimension at which an RGG can be distinguished from its Erd\H{o}s--R\'enyi counterpart with the same edge density.

For dense graphs, we show that for smooth kernels the critical scaling is  $d = n^{3/4}$, substantially lower than the threshold $d = n^3$ known for the hard RGG with step-function kernels \cite{bubeck2016testing}. We further extend our results to kernels whose signal-to-noise ratio scales with $n$, and formulate a unifying conjecture that 
the critical dimension is determined by  $n^3 \mathop{\rm tr}^2(\kappa^3) = 1$, where $\kappa$ is the standardized kernel operator on the sphere.

Departing from the prevailing approach of bounding the Kullback-Leibler divergence by successively exposing latent points, which breaks down in the sublinear regime of $d=o(n)$, our key technical contribution is a careful analysis of the posterior distribution of the latent points given the observed graph, in particular, the overlap between two independent posterior samples. As a by-product, we establish that $d=\sqrt{n}$
 is the critical dimension for non-trivial estimation of the 
 latent vectors up to a global rotation.

\end{abstract}

\tableofcontents

\section{Introduction}
\subsection{Problem setup}

\paragraph{Random geometric graph}
Given a function $K: [-1,1]\to [0,1]$, the $d$-dimensional \emph{random geometric graph} (RGG) on $n$ vertices with kernel $K$ refers to the following random graph ensemble, denoted by $\RGG(n,d,K)$.
Let $x_1,\ldots,x_n$ be independently and uniformly distributed on the unit sphere $S^{d-1}$ in $\reals^d$. Conditioned on these latent points, a graph $G$ on the set of vertices $[n]$ is generated such that each pair $(i,j)$ is connected independently with probability $K(\iprod{x_i}{x_j})$.

In the statistics literature, this model is also known as the \textit{latent space model} with link function $K$ \cite{hoff2002latent,KaurRastelliFrielRaftery2023LPM}, a widely used model for analyzing social networks and relational data. Other variants postulate that the connection probability is a function of the Euclidean distance between points as opposed to their inner product. These are equivalent in the present paper since we assume the latent space is the unit sphere.

\paragraph{Detection problem}
The primary feature of the RGG model is that, unlike the \ER graph, 
the edges of RGG are dependent. 
Since this dependency weakens as the latent dimension $d$ increases, it motivates the question of when it is possible to distinguish an RGG from its \ER counterpart $G(n,p)$, where 
\begin{equation}
p=\expect{K(\Iprod{x_1}{x_2})}
    \label{eq:edgedensity}
\end{equation}
is the edge density. The key quantity governing this hypothesis testing problem is the total variation
\[
\TV \equiv \TV(\RGG(n,d,K),G(n,p)).
\]
We are interested in the \textit{critical dimension} $d_*=d_*(n,K)$ such that 
this total variation converges to one if $d\ll d_*$ (detectable) and zero if $d\gg d_*$ (undetectable), as  $n \to \infty$.
In other words, $d_*$ represents the highest dimension at which the geometry in the RGG can be detected.

The prior literature (e.g.~\cite{bubeck2016testing,brennan2020phase,liu2022testing,bangachev2024fourier,bangachev2025sandwiching}) focuses on the so-called \textit{hard RGG} given rise by the step-function kernel $K(t) = \indc{t \geq \tau}$. Originally introduced to study wireless communication networks (known as the Gilbert disc model \cite{Gilbert61}), this model postulates that two points are connected if and only if their distance is below a given threshold. In high dimensions, the seminal work  of \cite{bubeck2016testing} shows that in the dense regime of $p=\Theta(1)$,
the critical dimension for detection is $d_*=n^3$, 
achieved by the test statistic of signed triangle count. Subsequent works~\cite{brennan2020phase,liu2022testing} derived sharper bounds for sparser graphs with $p=o(1)$. In particular, the  results of~\cite{liu2022testing} establish the negative result that  $d_*\le n^3p^2\polylog(n)$ for general $p$ and that $d_* \le \polylog(n)$ for
$p=\Theta(1/n)$. 
Moreover, it is shown that $d_*\ge (n p \log \frac{1}{p})^3$ for general $p$ again by counting signed triangles.
Despite this progress, a tight characterization of the critical dimension $d_*$ remains open in the intermediate regime $1/n \ll p \ll 1$ (see \cite[Sec.~3]{duchemin2023random} for a detailed account of this literature).


In comparison, the statistics literature typically works with \textit{soft RGG} models with smooth kernels, such as the logistic link function in the original paper \cite{hoff2002latent}.
More general kernels, including Gaussian kernels, are considered in \cite{dettmann2016random,ma2020universal,duchemin2023random}. A particularly important special case is given by linear kernels, which give rise to the \textit{Random Dot Product Graph} (RDPG) model \cite{young2007random,athreya2018statistical}. Despite the widespread use of soft RGGs, their detection problem was only recently initiated by Liu and R\'acz \cite{liu2023probabilistic}, who established upper and lower bounds on the detection threshold.


In this paper, we determine the critical dimension of detection for soft RGG with smooth kernels. Our main contributions are as follows.
\begin{itemize}
    \item 
    For dense graphs with a fixed kernel $K$ bounded away from 0 and 1 and satisfying $K'(0) \neq 0$, we show that the critical dimension for detection is $d_*(n,K)= n^{3/4}$.
    \item For scaled kernels of the form $K_r(\cdot)= K(r \cdot)$, we extend the analysis to both low-SNR ($r\to 0$) and high-SNR ($r\to\infty$) regimes, where the critical dimension is shown to be $d_* = n^{3/4} r^{3/2}$ under appropriate conditions.


    \item 
    To reconcile these results with existing and conjectured thresholds for the hard RGG, we formulate a general spectral conjecture that the optimal detection threshold is determined by 
    $n^3 \tr^2(\kappa^3) = 1$, where $\kappa$ is the standardized kernel operator $\kappa$ defined in~\prettyref{eq:operator}.   
\end{itemize}

As an interesting by-product of our proof, we show that estimating the latent points $(x_1, \ldots, x_n)$ with non-trivial correlation is impossible when $d \gg \sqrt{n}$. Conversely, when $d \ll \sqrt{n}$, consistent estimation of the latent points (up to a common rotation) can be achieved by a simple spectral method. These results establish that the critical dimension for estimation is $\sqrt{n}$, substantially lower than that for detection.


\subsection{Notation}
\label{sec:noation}
Throughout the paper, we denote by $\kappa$ the standardized version of the kernel $K$:
\begin{equation}
\kappa(t)
\eqdef
\frac{K(t)-p}{\sqrt{p(1-p)}},
\quad t \in [-1,1],
    \label{eq:kappa} 
\end{equation}
where $p$ is the average edge density in \prettyref{eq:edgedensity}.
That is, for $y\sim \Unif(S^{d-1})$ and any fixed $x \in S^{d-1}$, we have $p = \Expect[K(\Iprod{x}{y}))]$, so
$\Expect[\kappa(\Iprod{x}{y})]=0$.

For any $L_2$-function $\kappa$ satisfying $\Expect[\kappa(\Iprod{x}{y})^2]<\infty$, we may identify it as an operator 
on the sphere by
\begin{equation}
(\kappa f)(x) = \Expect[\kappa(\Iprod{x}{y}) f(y)],
\quad x\in S^{d-1}.
    \label{eq:operator}
\end{equation}
It is well-known (cf.~e.g.~\cite{dai2013approximation}) that any such operator is diagonalized by the spherical harmonics. For $k\geq 0$, let $\lambda_k$ denote the $k$th eigenvalue of $\kappa$.
Then the corresponding eigenfunctions are homogeneous harmonic polynomials of degree $k$ in $\reals^d$, restricted to $S^{d-1}$. The multiplicity of $\lambda_k$ is given by the dimension of the space of all such polynomials.
Furthermore, the eigenvalues can be found by expanding $\kappa$ under the Gegenbauer polynomial basis. See \prettyref{app:Gegenbauer} for details.

We denote as a shorthand $\kappa^m \equiv \kappa \circ \cdots \circ \kappa$ for $\kappa$ applied $m$ times, whose trace is given by
\begin{equation}
\tr(\kappa^m) 
= \Expect[\kappa(\Iprod{y_1}{y_2}) \cdots \kappa(\Iprod{y_m}{y_1})]
    \label{eq:tr}
\end{equation}
where $y_1,\ldots,y_m \iiddistr \Unif(S^{d-1})$.

\section{Main results}
\label{sec:main}

\subsection{Fixed kernel}
\label{sec:fixed}

We start by considering a fixed kernel $K$.
In this case, it turns out that strong detection is possible in any constant dimension $d$.\footnote{
By assumption the standardized kernel $\kappa$ is not zero and $\tr(\kappa^4)$ is a positive constant. Analogous to the signed triangle count, 
consider the signed 4-cycle count $S = \tr((A-p)^4)$.
Then $|\Expect_P[S]-\Expect_Q[S]|=\Theta(n^4)$, far exceeding the standard deviation under either $P$ or $Q$ which is at most $O(n^{7/2})$.}
Therefore, throughout the rest of this section we shall assume that $d\diverge$.
Our first main result shows that for smooth kernels the critical dimension is always $n^{3/4}$.

\begin{theorem}
Let $K: [-1,1]\to [0,1]$ be a fixed $C^\infty$ function that is bounded away from 0 and 1 and satisfies $K'(0) \neq 0$.
Then $d_*(n,K)=n^{3/4}$. In particular, when $d \ll d_*$,
the signed triangle count $\tr((A-p)^3)$
achieves detection. 
\label{thm:fixed}
\end{theorem}

To the best of our knowledge, \prettyref{thm:fixed} is the first work that pinpoints a sharp detection threshold beyond the dense hard RGG studied in~\cite{bubeck2016testing}. Interestingly, in this case the critical dimension scales as $n^{3/4}$, far below the $n^{3}$ threshold for dense hard RGG, indicating that detection is substantially harder in soft RGGs with a fixed kernel.

A key assumption is that $K'(0)\neq 0$, which ensures that the first non-trivial eigenvalue of the associated kernel operator is nonzero (see~\prettyref{lmm:rodrigues} for details). This condition is crucial for the $n^{3/4}$ scaling. When $K'(0)=0$, the detection threshold is conjectured to follow a different scaling, as we discuss in~\prettyref{sec:spectral-conjecture}.

\begin{remark}\label{rmk:remark_comparison}
The related work~\cite{liu2023probabilistic} also considers soft RGGs with smooth kernels but assumes Gaussian latent points.
In their model, the latent points $x_i$ are i.i.d.\ $\calN(0,I_d)$ and the kernel $K$ is the CDF of a random variable with mean $\mu$ and variance $r^2 d$ (with $r\ge 1$), satisfying $K'(t)>0$. Under additional technical conditions, they show that the total variation distance tends to one if $d \ll n^3/r^6$ and tends to zero if $d \gg n^3/r^4$. Taking $r=\sqrt d$ is analogous to our model with a fixed kernel and spherical latent points; under this correspondence, their bounds translate to $n^{3/4} \le d_* \le n$. 
In contrast, we show that the critical dimension is exactly $d_* =n^{3/4}$.
\end{remark}

\subsection{Low-SNR and high-SNR regimes}
\label{sec:lowhighsnr}
Given a fixed $K$, let us consider an RGG model with a scaled kernel 
\[
K_r(t) \equiv K(rt), \quad r>0. 
\]
Thus, the connection probability between nodes $i$ and $j$ equals $K(r \Iprod{x_i}{x_j})$.
As $r$ increases, the latent geometry becomes more pronounced and easier to detect.
Accordingly, we refer to the case of $r \ll 1 $ and $r\gg 1$ as the \emph{low-SNR} and the 
\emph{high-SNR} regime, respectively.

\paragraph{Low-SNR regime}

The extension of \prettyref{thm:fixed} to the low-SNR regime turns out to be relatively straightforward. First of all, it is easy to see that if $r \ll n^{-1/2}$, it is impossible to distinguish the RGG from \ER even in one dimension.\footnote{For $d=1$, we have $x_i \in \{\pm 1\}$, so $i$ and $j$
are connected with probability $K(r)$ if $x_ix_j=1$ and $K(-r)$ if $x_ix_j=-1$. This is a stochastic block model (SBM) with two communities and the conclusion then follows from existing results on the dense SBM \cite{banerjee2018contiguity}. }
Assuming a slightly stronger condition of $r \gg n^{-1/2+ \epsilon}$, 
the next theorem determines the critical dimension for detection.

\begin{theorem}
Let $n^{-1/2+ \epsilon} \ll r \ll 1$ for some constant $\epsilon > 0$.
Under the same condition on the kernel $K$ as in \prettyref{thm:fixed},
we have $d_*(n,K_r)=n^{3/4} r^{3/2}$.
\label{thm:lowsnr}
\end{theorem}

Next, for the special case of linear kernels (RDPG), the following results determine the sharp threshold even for \textit{sparse graphs} with average degree $np \gtrsim (\log n)^2$.

\begin{theorem}[Linear kernel]
\label{thm:linear}
    Consider $K(t) = p + r t$, where $p \in (0,\frac{1}{2})$ and $0<r\leq p$.
    Assume that 
    $r \gtrsim \sqrt{\frac{p}{n} }\log n$.
    Then $d_*(n,K) = (\frac{nr^2}{p})^{3/4}$.
    
\end{theorem}

\paragraph{High-SNR regime}

Following \cite{liu2023probabilistic}, we consider a kernel $K$ given by the CDF of a fixed probability distribution. As such, the connection probability $K(r\Iprod{x_i}{x_j})$ is monotonically increasing in terms of the overlap $\Iprod{x_i}{x_j}$. 
This effect is further amplified by the scaling factor $r\to\infty$ in the high-SNR regime.

As noted in \cite{liu2023probabilistic}, the interesting regime for soft RGG is $r \lesssim \sqrt{d}$, because if $r \gg \sqrt{d}$, 
$K(r\Iprod{x_i}{x_j})$ tends to $\indc{ \Iprod{x_i}{x_j} >0}$, which is the hard RGG kernel, and the detection threshold is always given by $d \gg n^3$ independent of $r$ and the kernel. In other words, the soft RGG model is only meaningful for $r \ll \sqrt{d}$. 
The next theorem determines the detection threshold under the stronger assumption of $r \lesssim d^{1/12-\epsilon}$
and other conditions on the kernel.

\begin{theorem}
Let $K(x) = \int_{-\infty}^x dt f(t)$ where $f$ is a probability density function.
Suppose that there are constants $c,C>0$ such that
(a) $\min\{K(-x),1-K(x)\} \geq c \exp(-C x^2)$  for all $x>0$;
(b) the characteristic function $\phi(t) \equiv \int dx f(x) e^{i xt}$ satisfies
$|\phi(\omega)| \leq C \exp(-c \omega^2)$  for all $\omega$.
Assume that $r \leq d^{1/12 - \epsilon}$ for some constant $\epsilon > 0$. 
Then $d_*(n,K_r)=n^{3/4} r^{3/2}$.
    \label{thm:highsnr}
\end{theorem}

We now comment on the assumptions imposed on the kernel $K$ in the preceding theorem. It is straightforward to verify that these conditions are satisfied by the CDF of any Gaussian or any Gaussian convolution (with an arbitrary distribution). 
These conditions arise in the polynomial approximation argument used in our proof of the negative result:
\begin{itemize}
    \item We first approximate the kernel by a polynomial up to statistically negligible error, then study the RGG with a polynomial kernel. Unlike Theorems \ref{thm:fixed} and \ref{thm:lowsnr}, where a constant-degree polynomial suffices, the high-SNR regime requires the degree to grow to infinity at an appropriate speed. The  condition (a) in \prettyref{thm:highsnr} ensures the Taylor polynomial approximant is itself a valid probability kernel (bounded between 0 and 1). 
\item The condition (b) on the characteristic function ensures\footnote{This is essentially an equivalent condition; see \cite[Problem 5.1]{stein2010complex}.} the $\ell$th Taylor coefficient decays superpolynomially as $(O(\ell))^{-\ell/2}$.
This is crucially required for the current result for polynomial kernels (see \prettyref{thm:poly}) to control various combinatorial quantities and higher-order terms. Relaxing this condition may require substantially new ideas and is beyond the scope of the current program.
\end{itemize}

\subsection{Recovery of latent points}
\label{sec:recovery}

Next, we turn to the problem of recovering the latent points given the random geometric graph. 
Since the latent points $x_1, \dots, x_n$ are only identifiable up to a global orthogonal transformation, it is equivalent to estimating the inner products $\langle x_i, x_j \rangle$. 
Define $X \in \reals^{n \times n}$ by $X_{ij} = \langle x_i, x_j \rangle$ for $i \ne j$ and $X_{ii} = 0$. 
The minimum mean squared error (MMSE) of estimating $X$ is
$$
\mmse \eqdef \min_{\hat X} \Expect \|\hat X - X\|_F^2 ,
$$
where the infimum is over all estimators $\hat X$ measurable with respect to the observed graph $A$. As a benchmark, the trivial estimator $\hat X = 0$ achieves a mean squared error $\Expect \|X\|_F^2 = n(n-1)/d$. We refer to $\Expect \|\hat X - X\|_F^2 \ll \Expect \|X\|_F^2$ as \textit{consistent} estimation and 
$\Expect \|\hat X - X\|_F^2 \leq (1-\Omega(1)) \Expect \|X\|_F^2$ as 
\textit{non-trivial} estimation (i.e.~weak recovery), respectively.

\begin{theorem} \label{thm:fixeded-kernel-spectral-recovery}
Let $K$ be a fixed smooth kernel satisfying the assumption of \prettyref{thm:fixed}.
If $1 \ll d \ll \sqrt{n}$, then $\mmse = O\left( \frac{1}{d^2} + \frac{d^2}{n} \right) \Expect \|X\|_F^2$, 
achieved by an efficient spectral method; if $d \gg \sqrt{n}$, then $\mmse = \left( 1 - O\left( \frac{1}{d} + \frac{\sqrt{\log n}}{n} \right) \right) \Expect \|X\|_F^2$.
\end{theorem}

The theorem shows that the critical dimension for weak recovery is given by $d = n^{1/2}$, lower than the detection threshold $d_* = n^{3/4}$. This phenomenon is analogous to known results in dense hard RGG, where the recovery threshold $d = \tilde \Theta(n)$ \cite{li2023spectral,mao2024impossibility} as compared to the detection threshold $d = n^3$ \cite{bubeck2016testing}. 
More precisely, assuming a hard RGG with Gaussian latent points, \cite{li2023spectral} shows that $\mmse = o(\Expect \|X\|_F^2)$ if $\polylog(n) \le d \le \frac{n}{\polylog(n)}$; for either spherical or Gaussian latent points, \cite{mao2024impossibility} proves that $\mmse = \Omega(\Expect \|X\|_F^2)$ if $d = \Omega(n)$. 
In comparison, our recovery threshold $d = n^{1/2}$ for soft RGG 
is tight up to constant factors. 


Moreover, our upper bound for the recovery of latent inner products is achieved via a simple spectral method that computes the leading eigenspace of the centered adjacency matrix, as in the hard RGG setting~\cite{li2023spectral}. 
For RGGs with general kernels, prior work shows that the spectral method achieves consistency when $d$ is a constant, under additional regularity conditions \cite{araya2019latent} or eigenvalue conditions \cite{eldan2022community} on the kernel. 
Our spectral analysis yields the tight condition $1 \ll d \ll n$ without logarithmic factors or extra eigenvalue conditions, which is new for recovery in any RGG model.



\subsection{Discussions}
\label{sec:discussion}

\paragraph{Proof techniques}
For the upper bound, we consider the signed triangle count introduced in \cite{bubeck2016testing}. Its analysis is by now standard and has appeared in several works (see, e.g.,~\cite{bubeck2016testing,liu2022testing,liu2023probabilistic,liu2023phase}). Here, leveraging the operator $\kappa$, we give a simple and unified argument that bounds the mean and variance under RGG in terms of $\tr(\kappa^3)$ and $\tr(\kappa^4)$ (cf.~\prettyref{sec:analysis_triangle}).

For the lower bound, the proof in \cite{bubeck2016testing} relies on the indistinguishability between the Wishart and Wigner matrices when $d\gg n^3$, a fact also proved by  \cite{jiang2015approximation}. 
By the data processing inequality, this implies the indistinguishability of the  hard RGG and the \ER, which can be obtained from the Wishart and Wigner by applying the step-function kernel entrywise.
This argument, by design, cannot yield a result that depends on the kernel, so news ideas are required in our setting. 

A common strategy for proving the impossibility of detection is to bound the second moment of the likelihood ratio. However, as we explain in~\prettyref{sec:outline_proof}, this second moment diverges when $d \ll n$, and therefore cannot capture the sublinear critical threshold $d \asymp n^{3/4}$.

Instead, our lower bound proof proceeds by bounding the KL divergence and expanding it using the chain rule. This technique was  introduced in~\cite{brennan2020phase} for RGGs and later refined in~\cite{liu2022testing,liu2023phase,liu2023probabilistic}. Roughly speaking, it decomposes the RGG by sequentially revealing the edges between the vertex $t$ and the previous vertices $1, \ldots, t-1$, thereby reducing the problem to bounding the incremental KL divergence at each step $t$. Obtaining a tight bound, however, requires understanding the posterior distribution of these edges given all edges exposed so far (i.e., the induced graph on vertices $1,\ldots,t-1$.)  

Nearly all existing works \cite{liu2022testing,liu2023phase,liu2023probabilistic,bangachev2024detection,bangachev2025random} avoids this difficulty by further exposing the latent points  $x_1, \ldots, x_{t-1}$.\footnote{The only  exception is \cite{liu2022testing} for sparse hard RGG, which analyzes the posterior distribution via belief
propagation and exploits the locally tree-like structure of neighborhoods. This approach, however, does not appear
to extend to our setting in the dense regime.}
While convenient, this strategy turns out to reveal too much information, causing the KL bound to blow up whenever $d \ll n$. This limitation has been noted in several recent works and identified as a common central obstacle to obtain sharp detection thresholds for various RGG models
(see, e.g.,~\cite[Section 8.1]{bangachev2025random} for a detailed discussion.)

To break the barrier at $d \asymp n$, we directly analyze the posterior. The core technical step requires proving a statement of the following type: For estimating $\iprod{x_1}{x_2}$, the information contained in the entire graph $A$ beyond the single edge $A_{12}$ is negligible.
As a by-product, we also obtain an impossibility result for non-trivial estimation of the latent points. We defer a more detailed discussion of these ideas to~\prettyref{sec:outline_proof}. 


\paragraph{Universality and non-universality} 
There are a variety of questions on universality arising in the RGG and related models. Of particular interest are \textit{channel universality} (with respect to the observation model) and \textit{source universality} (with respect to the latent point distribution.)

For the first question, let us consider a parametric model 
$P(\cdot | \theta)$. This output channel defines the following  hypothesis testing problem
$$
\calH_1: (A_{ij})_{i<j} \iiddistr P\left(\cdot | \iprod{x_i}{x_j} \right) \quad \text{v.s.} \quad \calH_0: (A_{ij})_{i<j} \iiddistr P\left(\cdot | \theta_0 \right),
$$
where $\theta_0$ is chosen to match the first moment of $A_{ij}$ under $\calH_0$ and $\calH_1$. Note that, under $\calH_1$, the observation $A$ is a noisy observation of the matrix $X=(\Iprod{x_i}{x_j})$ with rank $\min\{d,n\}$. 
A canonical model for low-rank matrix estimation and PCA assumes additive Gaussian noise, with $P(\cdot |\theta)=\calN(\theta,1)$ and $\theta_0=0$. 
The RGG model corresponds to the Bernoulli channel with kernel $K$, so that 
$P(\cdot |\theta) = \Bern(K(\theta))$ and $K(\theta_0) =p= \Expect[K(\Iprod{x_1}{x_2})]$.

For fixed $d$, the threshold for detection and non-trivial estimation, in terms of the SNR parameter, has been shown to be \textit{channel-universal} and only depend on the Fisher information at $\theta_0$ \cite{lesieur2015mmse,krzakala2016mutual,lelarge2017fundamental}.
For large $d$, channel universality may not hold.
To see this, consider the following example:
Under the Gaussian channel, i.e., $A_{ij} = \Iprod{x_i}{x_j} + z_{ij}$ with $z_{ij}\iiddistr N(0,1)$, it is straightforward to show the critical dimension of detection is $d_*=n$.
\footnote{The mean of $T=\sum_{i<j}A_{ij}^2$ is $\binom{n}{2}$ under $\calH_0$ and $\binom{n}{2}(1+1/d)$ under $\calH_1$; in both cases the variance is $O(n^2)$. Thus this test statistic succeeds when $d\ll n$. Conversely, when $d\gg n$, a direct calculation yields that the second moment of the likelihood ratio is $\Expect_{x \indep x^*}[ \exp ( \sum_{i<j} \iprod{x_i}{x_j} \iprod{x^*_i}{x^*_j}) ] \le \exp ( O(n^2/d^2) ) = 1+o(1)$.} 
Now, if we take the sign 
$A_{ij} = \indc{\Iprod{x_i}{x_j} + z_{ij}>0}$, we obtain a soft RGG with kernel given by the Gaussian CDF, and
\prettyref{thm:fixed} shows that 
the critical dimension lowers to $d_* = n^{3/4}$.
In other words, 
applying a one-bit quantization to the Gaussian observation makes the detection problem strictly harder.
This phenomenon stands in sharp contrast to the recovery threshold: For both Gaussian and its quantized version (RGG),  the  threshold for weak recovery is $d= \sqrt{n}$. (See \prettyref{thm:fixeded-kernel-spectral-recovery} and  \prettyref{rmk:recovery-gaussian}.)


Turning to source universality, for dense hard RGG, the threshold $d=n^3$ is known to hold for latent points that are spherical, Gaussian, or, more recently, 
uniform on the Hamming cube (\cite[Corollary 4.8]{bangachev2025random}.)
Extending these results, the following theorem (proved in \prettyref{app:dn3}) establishes non-detection at $d\gg n^3$ for a wide range of kernels and latent point distributions. In view of the  results in the existing literature and the present paper, we expect this to be sharp for the hard-RGG kernel and loose for smooth kernels.

\begin{theorem}
    \label{thm:dn3-universal}
    Let the kernel $K$ be a right-continuous function on $\reals$ with bounded variation.    
    Suppose that the latent points $x_1,\ldots,x_n$ are independently drawn from a product distribution $\pi^{\otimes d}$, where $\pi$ is a fixed subgaussian distribution.
    Then $\KL(P_A\|Q_A) = O(\frac{n^3}{d})$, 
provided that $d = \Omega(n^3)$.
\end{theorem}


Note that the above result assumes an inner product kernel. Otherwise, the detection threshold may \textit{not} be universal with respect to the latent point distribution. In \prettyref{app:non-universality}, we give an example of a distance kernel under which the detection thresholds for spherical and Gaussian latent points are different.


\subsection{A spectral conjecture}
\label{sec:spectral-conjecture}

Here, we reconcile our results with the existing literature on the hard RGG and formulate a unifying conjecture.
Recall that $\kappa$ denotes the normalized kernel identified as an operator \prettyref{eq:operator} on the sphere. 
We conjecture that the critical threshold for detection is determined by $\tr(\kappa^3)$.
\begin{conjecture}\label{conj:trace}
The sharp threshold for distinguishing RGG from \ER
is given by 
$n^3 \tr^2(\kappa^3) \gg 1$ (detectable)
and $n^3 \tr^2(\kappa^3) \ll 1$ (undetectable).
\end{conjecture}
This conjecture is supported by evidence from both the analysis of the signed triangle statistic and the impossibility proof. While $\tr(\kappa^3)$ naturally arises in the analyzing the signed triangle count, it also appears, somewhat surprisingly, in the KL expansion. See~\prettyref{app:heuristic} for more details.



In what follows, we denote the detection threshold $d_*$ as $\dtest$ to distinguish from the estimation threshold denoted by $\dest$. 
If $\tr(\kappa^3)$ is dominated by the contribution from $\lambda_1$ (with multiplicity $\Theta(d)$),  that is, $\tr(\kappa^3) \asymp \lambda_1^3 d$, then the conjectured critical threshold simplifies to 
\begin{equation}
\dtest = b_{1}^{3/2} n^{3/4},
    \label{eq:dtest}
\end{equation}
where $b_1 \equiv d \lambda_1$ is the scaled second eigenvalue. 
This setting encompasses the following examples:
\begin{itemize}
\item Hard RGG: $b_1 \asymp \sqrt{d p\log(1/p)}$, so  the conjectured threshold reduces to $d\asymp (np\log(1/p))^3$, matching the existing conjecture in~\cite{liu2022testing} with the exact log factors.
\item Scaled smooth kernel: $b_1 \asymp r$, so 
the conjectured threshold reduces to $d\asymp n^{3/4} r^{3/2}$, which is confirmed by Theorems~\ref{thm:fixed}, \ref{thm:lowsnr}, and~\ref{thm:highsnr} for $n^{-1/2+\epsilon} \ll r \ll d^{1/12-\epsilon}$.
\item Sparse linear kernel: $b_1 \asymp r/\sqrt{p}$, so the conjectured threshold reduces to
$d \asymp (nr^2/p)^{3/4}$, which is confirmed by~\prettyref{thm:linear}.
\end{itemize}
It is interesting to compare this with the phase transition threshold for estimation. For estimation, the critical dimension is conjectured to be 
\begin{equation}
\dest = b_{1} n^{1/2},
    \label{eq:dest}
\end{equation}
in the sense that if $d \ll \dest$, then consistent estimation is possible;
if $d \gg \dest$, then estimation better than chance is impossible.

More generally, 
if $\tr(\kappa^3)$ is dominated by the contribution from the $k_0^{\rm th}$ eigenvalue, namely, 
$\tr(\kappa^3) \asymp \lambda^3_{k_0} d^{k_0}$ (e.g., all eigenvalues before $k_0$ vanish), 
then the conjectured critical dimension for detection becomes
\begin{equation}
\dtest = b_{k_0}^{3/(2k_0)} n^{3/(4k_0)},
    \label{eq:dtest-general}
\end{equation}
where $b_{k} \equiv d^k \lambda_k$ is the scaled eigenvalue. For example, for the quadratic kernel $K(t)=\frac{1}{2} (t^2-\frac{1}{d})+\frac{1}{2}$, the conjectured critical dimension for detection is $\dtest = n^{3/8}$.

The quantity $k_0$ is reminiscent of the notion of information exponent in the generalized linear model 
\cite{benarous2021online,damian2024computational}, defined as the degree $k^\star$ of the first nonzero Hermite coefficient of the link function. In that setting, the role of the information exponent is computational, as it dictates that the sample complexity of stochastic gradient descent and low-degree polynomials scales as $\tilde\Theta(d^{\max(1,k^\star-1)})$. In contrast, here the role of $k_0$ is information-theoretic  as it determines the detection threshold.


\section{Proof outline}
\label{sec:outline_proof}
In this section, we outline the main proof ideas for detection lower bounds and highlight the key new ingredients. For clarity, we focus primarily on the \textit{linear} kernel. We first explain why the naive second-moment method fails for sublinear $d\ll n$, which motivates the use of the KL divergence. We then review a common simplification in prior work by revealing the latent points in the KL expansion, which fails to improve over the second-moment method. Next, we show how a careful posterior analysis allows us to break the sublinearity barrier and obtain the sharp threshold $d \asymp n^{3/4}$.
Finally, we discuss extensions to polynomial kernels and, more generally, to smooth kernels via polynomial approximation.

 \paragraph{Failure of the naive second-moment bound.} 
Let $P_A$ and $Q_A$ denote the distribution of the adjacency matrix $A$ under the RGG model with kernel $K$ and the \ER model, respectively.
A standard approach to proving the impossibility of detection is to bound the second moment of the likelihood ratio. Specifically, for linear kernels $\kappa(t)=t$, we have 
\begin{align}
\chi^2(P_A \| Q_A) +1  = \expects{\left( \frac{P_A}{Q_A} \right)^2}{Q_A} 
& =\expects{ \prod_{i<j} \left( 1+ \iprod{x_i}{x_j} \iprod{x^*_i}{x^*_j} \right)}{x \indep x^*} \label{eq:second_moment_expression} \\
& \le \expects{ \exp \left(  \sum_{i<j} \iprod{x_i}{x_j} \iprod{x^*_i}{x^*_j}\right)}{x \indep x^*} \nonumber \\
& \le \exp \left( O(n^2/d^2) \right) = 1+o(1), \quad \text{ when } d \gg n. \nonumber 
\end{align}
Unfortunately, when $d \ll n$ the second-moment method is derailed by certain rare events, for example,
$\calE=\{\|x_1-x_i\|_2 \le 0.1,~i=2,\ldots,n\}$, which satisfies $\prob{\calE} \ge \exp(-cnd)$ for some constant $c>0$. But, on the event $\calE$,
we have $\iprod{x_i}{x_j} \ge 0.5$. It then follows from~\prettyref{eq:second_moment_expression} that 
$$
\expects{\left( \frac{P_A}{Q_A} \right)^2}{Q_A}
\ge  \prob{x \in \calE} \prob{x^* \in \calE} \times \left( 1+ 0.5^2 \right)^{\binom{n}{2}}
\ge e^{-2 cnd} \times (1.25)^{\binom{n}{2}}
\to \infty, \quad \text{whenever } d \ll n.
$$


\paragraph{KL expansion.} 
To curb the influence of rare events, a tighter approach is to bound the KL divergence $\KL(P_A \| Q_A) = \Expect_{P_A}[\log (P_A/Q_A)]$.
Let $a^t=(A_{t,1},\ldots,A_{t,t-1})$ denote edges between node $t$ and  $1,\ldots,t-1$, and let $A^t$ be the subgraph induced by nodes $1, \ldots, t$.
Applying the chain rule for the KL divergence according to the decomposition $A = (a^2,\ldots,a^n)$ and the fact that $A$ consists of i.i.d.\  $\Bern(p)$ under $Q$, we obtain
\begin{align}
     \KL(P_A\|Q_A) 
= \sum_{t=2}^n 
\Expect_{A^{t-1}}[\KL(P_{a^t|A^{t-1}}\|\Bern(p)^{\otimes (t-1)})] .
\label{eq:KL_chainrule}
\end{align}
The law of $a^t$
conditional on $A^{t-1}$ is a mixture of products of Bernoulli distributions:
\begin{align}
P_{a^t|A^{t-1}}= 
\Expect_{x_1,\ldots,x_t|A^{t-1}}
\qth{\prod_{i=1}^{t-1} \Bern(K(\Iprod{x_i}{x_t}) }, \label{eq:posterior_distribution}
\end{align}
where the mixing distribution is given by the \textit{posterior} distribution of the latent points $(x_1,\ldots,x_t)$ given $A^{t-1}$. 
Note that  
$(x_1,\ldots,x_{t-1}, A^{t-1})$ defines an RGG with the same kernel with $t-1$ nodes and 
that $x_t \sim \Unif(S^{d-1})$ is independent of $(x_1,\ldots,x_{t-1}, A^{t-1})$. Nevertheless, the posterior distribution of $(x_1, \ldots, x_{t-1})$ given $A^{t-1}$ is typically complicated and challenging to analyze.

\paragraph{Revealing latent points leads to loose bounds.} 
In view of the conditional independence of $a^t$ and $A^{t-1}$ given the latent positions $x_1,\ldots, x_t$, to avoid dealing with the complicated posterior distribution in \prettyref{eq:posterior_distribution}, a simplification used in the prior work \cite{liu2022testing,liu2023probabilistic} is to further condition on the latent points, which, thanks to the convexity of the KL divergence, yields
\begin{align}
     \KL(P_A\|Q_A) 
&\le \sum_{t=2}^n 
\Expect_{x_1, \dots, x_{t-1}, A^{t-1}}[\KL(P_{a^t|x_1, \dots, x_{t-1}, A^{t-1}}\|\Bern(p)^{\otimes (t-1)})] \nonumber \\
&= \sum_{t=2}^n 
\Expect_{x_1, \dots, x_{t-1}}[\KL(P_{a^t|x_1, \dots, x_{t-1}}\|\Bern(p)^{\otimes (t-1)})] \nonumber \\
&\le \sum_{t=2}^n 
\Expect_{x_1, \dots, x_{t-1}}[\chi^2(P_{a^t|x_1, \dots, x_{t-1}}\|\Bern(p)^{\otimes (t-1)})], 
\label{eq:KL_loose}
\end{align}
where the last step follows from the general fact that 
$\KL \leq \chi^2$. Furthermore, introducing an independent replica $x_t^*$ of $x_t$, this $\chi^2$-divergence can be computed as
\begin{align*}
 \Expect_{x_1, \dots, x_{t-1}}[\chi^2(P_{a^t|x_1, \dots, x_{t-1}}\|\Bern(p)^{\otimes (t-1)})]
 & =\Expect_{x^{t-1} }  \Expect_{ x_t \indep x_t^* } \qth{ \prod_{i=1}^{t-1} \left( 1+ \iprod{x_i}{x_t} \iprod{x_i}{x^*_t} \right)} -1 \\
& = \Expect_{ x_t \indep x_t^* }  [\left(1+ \iprod{x_t}{x^*_t}/d \right)^t] -1 = \Theta (t^2/d^3). 
\end{align*}
Plugging this back into~\prettyref{eq:KL_loose}, we obtain $\KL=o(1)$, provided that $\sum_{t=2}^n t^2/d^3=o(1)$, that is, $d \gg n$. Thus,
while the KL expansion is in principle tighter than the second-moment method, once the latent points are revealed, this approach cannot succeed in the sublinear regime of $d \ll n$. 
In fact, this is also the primary source of looseness in the current impossibility conditions for the soft RGG (see~\cite[Theorem~1.1]{liu2023probabilistic}), the hard RGG 
(see~\cite[Theorem 1.2]{liu2022testing}) and other variants (see~\cite[Theorem 1.3]{bangachev2024detection} and~\cite[Section 8.1]{bangachev2025random}).


\paragraph{Analyzing the posterior.} 
To make progress in the sublinear regime,
we therefore need to analyze the posterior distribution in~\eqref{eq:posterior_distribution} in earnest.
Specifically, upper bounding each KL term in the expansion \prettyref{eq:KL_chainrule} by $\chi^2$, we obtain
\begin{align}
     \KL(P_A\|Q_A) 
= \sum_{t=2}^n 
\Expect_{A^{t-1}}[\KL(P_{a^t|A^{t-1}}\|\Bern(p)^{\otimes (t-1)})]
\leq  \sum_{t=2}^n 
\Expect_{A^{t-1}}[\chi^2(P_{a^t|A^{t-1}}\|\Bern(p)^{\otimes (t-1)})] . \label{eq:kl-chain-rule-chi-sq}
\end{align}
A direct computation of the above $\chi^2$-divergence (see Lemma~\ref{lem:kl-expansion}) yields that
$$
\KL(P_A\|Q_A)
\le \sum_{k=2}^{n-1} \binom{n}{k+1} 
\underbrace{\Expect_{A}\left[\pth{\Expect_{x_1,\ldots,x_k, x_{n+1}|A}
\qth{\prod_{i=1}^k
\Iprod{x_i}{x_{n+1}}}}^2\right]}_{\triangleq g(k)},
$$
where $x_{n+1}$ is an independent copy of $x_i$'s. 
By symmetry, $g(k)=0$ for all odd $k$. For even $k$, we may average over $x_{n+1}$ using 
Wick's formula (Isserlis' theorem) to obtain:
$$
g(k) = \left( \frac{1}{d(d+2)\cdots (d+k-2)}\right)^2 
\expects{ \left(\sum_{\pi \in \Pi([k])} 
\expects{\prod_{(i,j)\in \pi} \iprod{x_i}{x_j}}{x_1,\ldots,x_k | A }\right)^2 }{A},
$$
where $\Pi([k])$ denotes the set of all pairings of $[k]$. 
Thus, the problem reduces to controlling these cross moments of posterior correlations. 
Of particular importance is the case of $k=2$, where 
\[
g(2) = d^{-2} \Expect[(\Expect[\iprod{x_1}{x_2} \mid A])^2].
\]
Note that $\Expect[\iprod{x_1}{x_2}^2]=1/d$, and 
$\Expect[(\Expect[\iprod{x_1}{x_2}|A])^2]$ is precisely the variance reduction thanks to observing the graph $A$. It is straightforward to verify that 
observing a single edge $A_{12}$ reduces the variance by
$\Expect[\left(\expect{\iprod{x_1}{x_2} \mid A_{12} }\right)^2] = \Theta(d^{-2})$. 
Crucially, we show that the additional variance reduction due to \textit{all} remaining edges is negligible,  that is, 
\begin{equation}
\Expect[(\Expect[\iprod{x_1}{x_2} |A])^2]  =O(d^{-2})
\label{eq:varred-A}
\end{equation}
so that $g(2)=O(d^{-4})$. Moreover, this argument extends to show that $g(k)=O(d^{-2k})$ for all $k \le k_0$, where $k_0$ is some  large constant. 
For $k \ge k_0$, we instead establish a coarser bound $g(k)=O(k^k d^{-3k/2})$ which is looser in $d$ but has a controlled dependency on $k$. Using these estimates, we show that the sum $\sum_{k=2}^{n-1} \binom{n}{k+1} g(k)$ is dominated by the leading term of $k=2$,
and therefore vanishes whenever $n^3 g(2)=o(1)$, yielding the sharp condition $d \gg n^{3/4}$.


As discussed above, a \emph{sine qua non} of our proof is a careful analysis of the posterior distribution and, in particular, of the following quantity which we call the \textit{posterior overlap}. We 
show that
when $d \gg\sqrt{n}$, 
for $\delta =\Theta(1/d)$,
\begin{align}
\prob{\sum_{i<j} \iprod{\tilde{x}_i}{\tilde{x}_j}
\iprod{x_i}{x_j} \ge \delta \frac{n^2}{d}} \le n^{-\Omega(1)} , \label{eq:posterior_probability_0}
\end{align}
where $\tilde{x}$ is  a fresh  draw from the \emph{posterior} distribution of $x$ given $A$. By symmetry, this immediately yields the desired bound 
\prettyref{eq:varred-A}.\footnote{
We remark that results analogous to \prettyref{eq:posterior_probability_0}, but with $\delta=o(1)$, have been obtained in the related problem of low-rank matrix estimation aiming at sharp signal-to-noise threshold; however,  to the best of our knowledge, existing rigorous results \cite{lelarge2017fundamental} are largely limited to fixed dimension $d$, while the high-dimensional case remains open~\cite{pourkamali2024matrix,barbier2024information}. Crucially, in the current paper, obtaining the sharp detection threshold requires establishing~\prettyref{eq:posterior_probability_0} with $\delta=\Theta(1/d)$ in high dimensions.}

To prove \prettyref{eq:posterior_probability_0}, we apply a change-of-measure argument to upper-bound the probability by 
a truncated exponential moment of the form $$
\Expect\left[\exp \left(\sum_{i<j} \langle x_i,x_j\rangle \langle x^*_i,x^*_j\rangle \right) \indc{ \sum_{i<j} \langle x_i,x_j\rangle \langle x^*_i,x^*_j\rangle \ge \delta n^2/d} \right],
$$
where 
$x$ and $x^*$
 are two independent draws from the \emph{prior}. We then show that this truncated exponential moment is $n^{-\Omega(1)} $ by a suitable application of the Hanson-Wright inequality.
We remark that this change-of-measure argument  is  an instance of the so-called ``planting trick'', a technique first introduced
in the study of random constraint satisfaction problems~\cite{achlioptas2008algorithmic}, and more recently used to establish impossibility results for non-trivial estimation in statistical inference problems, including  group testing~\cite{coja2022statistical} and planted subgraph recovery~\cite{Mossel2025sharp,gaudio2025all}.

\paragraph{Extensions to polynomial and general kernels.} We next extend the analysis from the linear kernel to polynomial kernels. For degree-$L$ polynomial kernels, one can still derive an explicit expression for $g(k)$ by averaging over $x_{n+1}$ using Wick's formula. The resulting expression consists of higher-order moments of posterior correlations of the form 
\begin{equation*}
M_{\boldsymbol{\ell}}
\triangleq \Expect_{A}
\pth{\Expect_{x_1,\ldots,x_k|A}
\qth{ \prod_{1 \le i<j \le k} 
\iprod{x_i}{x_j}^{\ell_{ij}}}
}^2,
\end{equation*}
where the multi-index
$\boldsymbol{\ell}=(\ell_{ij})$ defines a multi-graph on $[n]$ with maximal degree $L$. Generalizing the previous posterior  analysis, we show that when $d \gtrsim \sqrt{n}$, 
$M_{\boldsymbol{\ell}}=O(n^{2\ell-v} d^{-2\ell})$, where $v$ is the number of non-isolated vertices and $\ell$ is the total number of edges, provided $\ell \le L_0$ for some large but fixed  constant $L_0$. For $\ell \ge L_0$, we instead rely on a coarser bound $M_{\boldsymbol{\ell}} \le  d^{-\ell} \prod_{i=1}^k (2\ell_i-1)!!$, where $\ell_i$ is the degree of vertex $i$.
Summing contributions over all multi-graphs 
$\boldsymbol{\ell}$ yields $g(k)=O(d^{-2k})$ for $k \le k_0$ and $g(k)=O(k^k d^{-3k/2})$ for $k>k_0$, leading again to the threshold $d \gg n^{3/4}$.

Finally, to handle general smooth kernels, we approximate them by polynomials. For a fixed or low SNR, the degree only needs a sufficiently large constant, while for a high SNR, the degree must grow to infinity at an appropriate speed. 
We show that this approximation induces a vanishing total variation distance between the resulting random geometric graph models, allowing the above analysis to carry over to general kernels.


\section{Counting signed triangles}
\label{sec:analysis_triangle}
In this section, we analyze the performance of the signed triangle count for detection. 
Denote by $\bar A = (\bar A_{ij})$ the standardized adjacency matrix, with 
\begin{equation}
\bar A_{ij} \eqdef \frac{A_{ij} - p}{\sqrt{p (1-p)}}.
\label{eq:def-A-bar}
\end{equation}
Define the signed triangle count as 
\begin{equation}
T(A) \eqdef \sum_{(i,j,k) \in \binom{[n]}{3}} \bar A_{ij} \bar A_{jk} \bar A_{ki} .
    \label{eq:signed-triangle}
\end{equation}
One can verify that under the \ER model, $\Expect_Q[T(A)] = 0$ and 
$\Var_Q[T(A)] = \binom{n}{3}$ (see \cite[Sec.~3.1]{bubeck2016testing}). The next theorem bounds the mean and variance under the RGG.

\begin{theorem}
\label{thm:ub}
Under the RGG model,
$$
\Expect_P[T(A)] = \binom{n}{3} \tr(\kappa^3) 
$$
and
$$
\Var_P(T(A)) \le 6 \binom{n}{4} \tr(\kappa^4) \pth{1 + \frac{|1-2p|}{1-p} \lor \frac{|1-2p|}{p}} + \binom{n}{3} \left(1 + \frac{(1-2p)^3}{(p(1-p))^{3/2}} \tr(\kappa^3) \right)
$$
where $\kappa$ denotes the standardized kernel in \prettyref{eq:kappa}.

\end{theorem}



\begin{proof}
It is straightforward to compute the mean. 
For the variance, we have
$$
\Var_P(T(A)) = \sum_{(i,j,k),(i',j',k')} \Cov_P(\bar A_{ij} \bar A_{jk} \bar A_{ki}, \bar A_{i'j'} \bar A_{j'k'} \bar A_{k'i'}) .
$$
Consider the following cases.
\begin{itemize}
\item
If $|\{i,j,k\} \cap \{i',j',k'\}| = 0$, then $\bar A_{ij} \bar A_{jk} \bar A_{ki}$ and $\bar A_{i'j'} \bar A_{j'k'} \bar A_{k'i'}$ are independent, so their covariance is zero. 

\item
If $|\{i,j,k\} \cap \{i',j',k'\}| = 1$, then the covariance is also zero. 
To see this, suppose without loss of generality that $i = i'$.
Let $Q \in \reals^{d \times d}$ be a uniformly random orthogonal matrix. Then $x_i$ and $Q x_i$ are independent. 
Therefore, the tuple of random variables
$$
(\iprod{x_i}{x_j}, \iprod{x_i}{x_k}, \iprod{x_j}{x_k}) = (\iprod{Q x_i}{Q x_j}, \iprod{Q x_i}{Q x_k}, \iprod{Q x_j}{Q x_k})
$$
is independent from $(\iprod{x_i}{x_{j'}}, \iprod{x_i}{x_{k'}}, \iprod{x_{j'}}{x_{k'}})$. 
Hence $\bar A_{ij} \bar A_{jk} \bar A_{ki}$ and $\bar A_{i'j'} \bar A_{j'k'} \bar A_{k'i'}$ are independent, and the claim follows. 

\item 
If $|\{i,j,k\} \cap \{i',j',k'\}| = 2$, suppose $i=i'$ and $j=j'$ without loss of generality. 
Write $X = (x_1, \dots, x_n)$. 
Then we have
\begin{align*}
\Expect[\bar A_{ij}^2 \mid X]
&= K(\iprod{x_i}{x_j}) \frac{(1-p)^2}{p(1-p)} + (1-K(\iprod{x_i}{x_j})) \frac{p^2}{p(1-p)} \\
&= \left( p + \kappa(\iprod{x_i}{x_j}) \sqrt{p(1-p)} \right) \frac{(1-p)^2}{p(1-p)} + \left( 1 - p - \kappa(\iprod{x_i}{x_j} ) \sqrt{p(1-p)} \right) \frac{p^2}{p(1-p)} \\
&= 1 + \kappa(\iprod{x_i}{x_j}) \frac{1-2p}{\sqrt{p(1-p)}} .
\end{align*}
By the conditional independence of the edges given $X$, we obtain 
\begin{align*}
&\Expect_P[\bar A_{ij}^2 \bar A_{ik} \bar A_{jk} \bar A_{ik'} \bar A_{jk'}] \\
&= \Expect_X\left[ \left( 1 + \kappa(\iprod{x_i}{x_j}) \frac{1-2p}{\sqrt{p(1-p)}} \right) \kappa(\iprod{x_i}{x_k}) \kappa(\iprod{x_j}{x_k}) \kappa(\iprod{x_i}{x_{k'}}) \kappa(\iprod{x_j}{x_{k'}}) \right] \\
&= \tr(\kappa^4) + \frac{1-2p}{\sqrt{p(1-p)}} \Expect_X\left[ \kappa(\iprod{x_i}{x_j}) \kappa(\iprod{x_i}{x_k}) \kappa(\iprod{x_j}{x_k}) \kappa(\iprod{x_i}{x_{k'}}) \kappa(\iprod{x_j}{x_{k'}}) \right] .
\end{align*}
Since $|\kappa(t)| = \frac{|K(t)-p|}{\sqrt{p(1-p)}} \le \sqrt{\frac{p}{1-p}} \lor \sqrt{\frac{1-p}{p}}$, we obtain
\begin{align*}
&\left| \Expect_X\left[ \kappa(\iprod{x_i}{x_j}) \kappa(\iprod{x_i}{x_k}) \kappa(\iprod{x_j}{x_k}) \kappa(\iprod{x_i}{x_{k'}}) \kappa(\iprod{x_j}{x_{k'}}) \right] \right| \\
&= \left| \Expect_{x_i,x_j} \left[ \kappa(\iprod{x_i}{x_j}) ( \Expect_{x_k} [ \kappa(\iprod{x_i}{x_k}) \kappa(\iprod{x_j}{x_k}) ] )^2 \right] \right| \\
&\le \pth{\sqrt{\frac{p}{1-p}} \lor \sqrt{\frac{1-p}{p}}} \Expect_{x_i,x_j} \left[ ( \Expect_{x_k} [ \kappa(\iprod{x_i}{x_k}) \kappa(\iprod{x_j}{x_k}) ] )^2 \right] \\
&= \pth{\sqrt{\frac{p}{1-p}} \lor \sqrt{\frac{1-p}{p}}} \Expect_X\left[ \kappa(\iprod{x_i}{x_k}) \kappa(\iprod{x_j}{x_k}) \kappa(\iprod{x_i}{x_{k'}}) \kappa(\iprod{x_j}{x_{k'}}) \right] \\
&= \pth{\sqrt{\frac{p}{1-p}} \lor \sqrt{\frac{1-p}{p}}} \tr(\kappa^4) .
\end{align*}
Therefore,
\begin{align*}
\left| \Cov_P(\bar A_{ij} \bar A_{ik} \bar A_{jk} , \bar A_{ij} \bar A_{ik'} \bar A_{jk'}) \right|
\le \left| \Expect_P[\bar A_{ij}^2 \bar A_{ik} \bar A_{jk} \bar A_{ik'} \bar A_{jk'}] \right| 
\le \tr(\kappa^4) \pth{1 + \frac{|1-2p|}{1-p} \lor \frac{|1-2p|}{p}} .
\end{align*}


\item
If $|\{i,j,k\} \cap \{i',j',k'\}| = 3$, i.e., $i=i'$, $j=j'$, $k=k'$, then the covariance can be bounded by
\begin{align*}
&\Var_P(\bar A_{ij} \bar A_{jk} \bar A_{ki}) 
\le \Expect_P[(\bar A_{ij} \bar A_{jk} \bar A_{ki})^2] \\ 
&= \Expect_X\left[ \left( 1 + \kappa(\iprod{x_i}{x_j}) \frac{1-2p}{\sqrt{p(1-p)}} \right) \left( 1 + \kappa(\iprod{x_j}{x_k}) \frac{1-2p}{\sqrt{p(1-p)}} \right) \left( 1 + \kappa(\iprod{x_k}{x_i}) \frac{1-2p}{\sqrt{p(1-p)}} \right) \right] \\
&= 1 + \frac{(1-2p)^3}{(p(1-p))^{3/2}} \tr(\kappa^3) . 
\end{align*}
\end{itemize}
\end{proof}

\section{Detection lower bound for  polynomial kernels}
\label{sec:polynomial-kernel}

This section establishes the following negative result when the kernel is a polynomial.

\begin{theorem}
Suppose that the standardized kernel
\[
\kappa(t)
\eqdef
\frac{K(t)-p}{\sqrt{p(1-p)}} 
\]
is a polynomial $\kappa(t) = \sum_{\ell=0}^L b_\ell t^\ell$, where the coefficients satisfy 
\begin{equation}
|b_\ell| \le \frac{B r^{\ell}}{\sqrt{\ell!}} \ \forall \, 1 \le \ell \le L, 
\quad |b_0| \le \frac{B r^2}{d} ,
    \label{eq:coeffs-assumption}
\end{equation}
for some $B>0$.
Assume that 
$$d \gg n^{3/4} r^{3/2}, $$
$n^{-1/2 + \epsilon} \le r \le d^{1/3 - \epsilon}$, and $L \le d^{1/6 - \epsilon}$.
Then $\TV(P_A, Q_A) = o(1)$.

In the special case of linear kernel, i.e., $\kappa(t) = b_1 t$, if $d \gg n^{3/4} b_1^{3/2}$ 
and $b_1\gtrsim n^{-1/2} \log n$, then $\TV(P_A, Q_A) = o(1)$.


\label{thm:poly}
\end{theorem}

\begin{remark}
    We comment that the standing assumption on $b_\ell$'s for $\ell\ge 1$  already implies that 
$|b_0| \le 2B \frac{r^2}{d}$ as long as $d \ge 2 r^2$. To see this,
$\Expect[\kappa(\iprod{x_1}{x_2})] = 0$ by definition and hence,
$$
b_0 = - \sum_{\ell=1}^{L/2} b_{2\ell}
\Expect[\iprod{x_1}{x_2}^{2\ell}]
=- \sum_{\ell=1}^{L/2} b_{2\ell} \frac{(2\ell-1)!!}{d(d+2)\cdots (d+2\ell-2)}.
$$
Therefore,
\begin{equation}
|b_0| \leq \sum_{\ell=1}^{L/2} |b_{2\ell}|
\frac{(2\ell-1)!!}{d^\ell}
\le B \sum_{\ell=1}^{L/2} \frac{r^{2\ell}}{\sqrt{(2\ell)!}}
\frac{(2\ell-1)!!}{d^\ell}
\le B \sum_{\ell=1}^{L/2} \left(\frac{r^2}{d}\right)^\ell
\le 2B \frac{r^2}{d},
    \label{eq:b0-auto}
\end{equation}
where the second inequality follows from~\prettyref{eq:coeffs-assumption}; the third inequality holds due to $(2\ell-1)!!\le \sqrt{(2\ell)!}$; and the last inequality holds due to $d \ge 2r^2$.
\end{remark}



\subsection{Analysis of the posterior overlap}
\label{sec:posterior-analysis}


Before proving Theorem~\ref{thm:poly}, we establish the following bound on the overlap between $X$ drawn from the prior and $\tilde{X}$, an independent redraw from the posterior.
While the direct consequence is a \emph{recovery} lower bound, this result is also a key step in the proof of the \emph{detection} lower bound as we will see in \prettyref{sec:proof-theorem-poly}.

\begin{proposition} \label{prop:posterior-concentration}
Let $X \in \reals^{n \times n}$ be defined by $X_{ij}=\iprod{x_i}{x_j}$ for $i \ne j$ and $X_{ii}=0$.
Let $\tilde{X}$ be a fresh random draw from the posterior distribution $\mu_A$ of $X$ conditional on $A$. 
For any constant $D>0$, there exists a constant $C>0$ depending only on $D$ such that the following holds.
Assume that the standardized kernel is a polynomial $\kappa(t) = \sum_{\ell=0}^L b_\ell t^\ell$ whose coefficients satisfy \eqref{eq:coeffs-assumption}.
If
\nbb{
$$
d \ge C B^2 r^2 + n^{1/2} B r^{3/2} + n^{1/2} (B + B^{1/2}) r + n^{1/3} (B^{4/3} + B^{2/3}) r^{4/3} + \log n + r^2 (\log n)^3 ,
$$}
then
\begin{align}
\prob{\iprod{\tilde{X}}{X} \ge \delta \frac{n^2}{d}} \le \gamma ,
\label{eq:posterior_probability}
\end{align}
where $\gamma \eqdef n^{-D}$ and 
\begin{equation}
\delta \eqdef C \max\{B(B+1) r^2/d , (\log n)^{1/2}/n \} .
\label{eq:def-delta}
\end{equation}
For a linear kernel $\kappa(t) = b_1 t$, if 
\nbb{$$d \gtrsim n^{1/2} b_1 + n^{1/3} b_1^{4/3} + \log n,$$}
then \prettyref{eq:posterior_probability} holds with $\gamma \eqdef n^{-D}$ and $\delta \eqdef C \max\{ b_1^2/d , (\log n)^{1/2} / n \}$.
\end{proposition}

\begin{remark}[Intuition for the scaling of $\delta$]
We briefly explain the intuition behind the scaling of $\delta$ in~\prettyref{eq:def-delta}. 
First, each term $\tilde{X}_{ij} X_{ij}$ typically has magnitude of order $1/d$. Consequently, even if we pretend that $\tilde{X}$ and $X$ are independent matrices with independent entries, the inner product $\langle \tilde{X}, X \rangle$ would be of order $n/d$. This suggests that $\delta$ should scale at least as $1/n$, which corresponds to the second term in~\prettyref{eq:def-delta}. 
To understand the first term in~\prettyref{eq:def-delta}, note that 
\begin{align*}
\mathbb{E}\!\left[\langle \tilde{X}, X \rangle \right] = n(n-1)\, \mathbb{E}\!\left[\tilde{X}_{12} X_{12}\right] = n(n-1)\, \mathbb{E}\!\left[\big(\mathbb{E}[ X_{12} \mid A ]\big)^2 \right].
\end{align*}
By Jensen's inequality,
\[
\mathbb{E}\!\left[\big(\mathbb{E}[ X_{12} \mid A_{12}] \big)^2 \right]
\le
\mathbb{E}\!\left[\big(\mathbb{E}[ X_{12} \mid A ] \big)^2 \right]
\le
\mathbb{E}\!\left[X_{12}^2\right].
\]
A direct calculation shows that 
\begin{align*}
\mathbb{E}\!\left[\big(\mathbb{E}[ X_{12} \mid A_{12}] \big)^2 \right]
= \big(\mathbb{E}[\kappa(X_{12}) X_{12}]\big)^2 = \left(\sum_{\ell} b_\ell \, \mathbb{E}[\langle x_1, x_2\rangle^{\ell+1}]\right)^2
\asymp \left(\frac{b_1}{d}\right)^2,
\end{align*}
where the first equality holds because $\expect{X_{12} \mid A_{12}}=
(-1)^{1-A_{12}} [p/(1-p)]^{1/2-A_{12}} \expect{X_{12}\kappa(X_{12})}$ 
and
the last asymptotic follows from~\prettyref{eq:coeffs-assumption} under the condition $d \gtrsim r^2$. This yields $\delta \gtrsim b_1^2/d$, consistent with the first term in~\prettyref{eq:def-delta}. In other words, here the entire graph $A$ provides only negligible more information than a single edge $A_{12}$ to estimate $X_{12}$.
\end{remark}

As an immediate corollary of~\prettyref{prop:posterior-concentration}, we obtain the following bounds on the moments of the posterior overlap. 
\begin{corollary} \label{cor:posterior-moments}
In the setting of \prettyref{prop:posterior-concentration}, for any constant $C_1 > 0$, there is a constant $C_2 > 0$ depending only on $C_1$ such that for all positive integers $\ell \le C_1$, we have
$$
\expect{\iprod{\tilde{X}}{X}^\ell}
\le C_2 \bigg( \max \bigg\{ \frac{B(B+1) r^2 n^2}{d^2} , \frac{n (\log n)^{1/2}}{d} \bigg\} \bigg)^\ell .
$$

In the case where the kernel is linear, i.e., $\kappa(t) = b_1 t$, we have
$$
\expect{\iprod{\tilde{X}}{X}^\ell}
\le C_2 \bigg( \max \bigg\{ \frac{b_1^2 n^2}{d^2} , \frac{n (\log n)^{1/2}}{d} \bigg\} \bigg)^\ell .
$$
\end{corollary}

\begin{proof}
Let $\delta$ and $\gamma$ be as defined in \prettyref{prop:posterior-concentration}. 
Note that 
\begin{align*}
\expect{\iprod{\tilde{X}}{X}^\ell}
& =\expect{\iprod{\tilde{X}}{X}^\ell \indc{\iprod{\tilde{X}}{X} \ge \frac{\delta n^2}{d}}}+ 
\expect{\iprod{\tilde{X}}{X}^\ell \indc{\iprod{\tilde{X}}{X} < \frac{\delta n^2}{d}}} \\
& \le \sqrt{\expect{\iprod{\tilde{X}}{X}^{2\ell}} \prob{\iprod{\tilde{X}}{X} \ge \frac{\delta n^2}{d}}}+  \left(\frac{\delta n^2}{d}\right)^\ell. 
\end{align*}
Moreover, 
\begin{align*}
\expect{\iprod{\tilde{X}}{X}^{2\ell}}
\le \expect{ \Fnorm{\tilde{X}}^{2\ell} \Fnorm{X}^{2\ell}}
=\expect{  
\left(\expect{ \Fnorm{X}^{2\ell} \mid A}\right)^2  }
\le \expect{\Fnorm{X}^{4\ell} }
\le \left( \frac{n^2}{d} \right)^{2\ell} (4\ell-1)!!,
\end{align*}
where the last inequality holds because, by Jensen's inequality,
$$
\expect{\Fnorm{X}^{4\ell} }
=\expect{ \left( \sum_{i\neq j} X_{ij}^2 \right)^{2\ell}}
\le [n(n-1)]^{2\ell-1} \expect{\sum_{i\neq j} X_{ij}^{4\ell}}
\le n^{4\ell} \expect{X_{ij}^{4\ell}}
\le n^{4\ell} (4\ell-1)!! d^{-2\ell}.
$$
Therefore, we deduce that 
$$
\expect{\iprod{\tilde{X}}{X}^\ell}
\le \left( \frac{n^2}{d}\right)^\ell
\left( \sqrt{(4\ell-1)!! \cdot \gamma} + \delta^\ell \right).
$$
Recall the choices of $\delta$ and $\gamma = n^{-D}$ in \prettyref{prop:posterior-concentration}. We can choose $D$ to be sufficiently large depending only on $C_1$ such that $\sqrt{(4\ell-1)!! \cdot \gamma} \le \delta^\ell$ for all $\ell \le C_1$. The conclusion then follows.
\end{proof}

Before proving \prettyref{prop:posterior-concentration}, let us first establish a few lemmas.

\begin{lemma} \label{lem:X-spectral-norm-bound}
Let $X \in \reals^{n \times n}$ be defined by $X_{ij}=\iprod{x_i}{x_j}$ for $i \ne j$ and $X_{ii}=0$. 
Then there is an absolute constant $c>0$ such that with probability at least $1 - 2 \exp(-cd) - 2 \exp(-cn)$, 
$$
\|X\| \lesssim \sqrt{\frac{n}{d}} + \frac{n}{d} 
$$
and
$$
\Fnorm{X}^2 \lesssim \frac{n^2}{d} .
$$
\end{lemma}

\begin{proof}
We have $X=Z^\top Z - I_n$, where $Z \eqdef [x_1,\ldots, x_n]$, so  
$$
\|X\| \le \max\{1 - \lambda_{\min}(Z^\top Z), \lambda_{\max}(Z^\top Z) - 1\} = \max\{ 1-\sigma_{\min}^2(Z), \sigma_{\max}^2(Z)-1\} .
$$
If $d \ge n$, then by Theorem~5.58 in \cite{vershynin2010introduction}, with probability at least $1-2\exp(-cn)$, 
$$
 1- O(\sqrt{n/d}) \le \sigma_{\min}(Z) \le \sigma_{\max}(Z) \le 1+ O(\sqrt{n/d}).
$$
If $d < n$, then by Theorem~5.39 in \cite{vershynin2010introduction}, with probability at least $1-2\exp(-cd)$, 
$$
 \sigma_{\max}(Z) \le \sqrt{n/d} + C .
$$
Combining the above bounds, we obtain
$\norm{X} \lesssim \sqrt{n/d} + n/d$. 
Furthermore,
$$
\Fnorm{X}^2= \sum_{i=1}^{n\wedge d} \left( \sigma_i^2(Z)-1  \right)^2 + \max\{0, n-d\}
\lesssim (n \wedge d) \left( \sqrt{n/d} + n/d \right)^2 + \max\{0, n-d\}
\lesssim n^2/d ,
$$
completing the proof.
\end{proof}

\begin{lemma} \label{lem:fourth_moment_hypercontractivity}
Consider i.i.d.\ uniformly random $x_1, \dots, x_n$ over $S^{d-1}$. There is an absolute constant $c>0$ such that 
with probability at least $1 - n \exp(-cd) - \exp(- \min \{ (nd)^{1/8}, n^{1/4} \})$, 
$$
\sum_{i<j} \iprod{x_i}{x_j}^4 
\lesssim \frac{n^2}{d^2}.
$$
\end{lemma}

\begin{proof}
Let $g_1, \dots, g_n$ be i.i.d.\ standard Gaussian vectors in $\reals^d$ such that $x_i = g_i/\|g_i\|$ for each $i \in [n]$. By the concentration of a $\chi^2_d$ random variable together with a union bound, we have $\|g_i\|^2 \ge d/2$ for all $i \in [n]$ with probability at least $1 - n \exp(-cd)$ for an absolute constant $c>0$. On this event, $\sum_{i<j} \iprod{x_i}{x_j}^4 
\lesssim \frac{1}{d^4} \sum_{i<j} \iprod{g_i}{g_j}^4$.

Moreover, by Theorem~6.7 in \cite{janson1997gaussian} applied to the degree-$8$ polynomial $\sum_{i<j} ( \iprod{g_i}{g_j}^4 - \Expect[\iprod{g_i}{g_j}^4] )$ of Gaussian random variables, we obtain 
$$
\sum_{i<j} \iprod{g_i}{g_j}^4 \lesssim \sum_{i<j} \Expect[\iprod{g_i}{g_j}^4] + \sqrt{\Var(\sum_{i<j} \iprod{g_i}{g_j}^4)} \cdot (\log(1/\gamma))^4
$$
with probability at least $1-\gamma$. 
The expectation can be computed to give
$$
\Expect[\iprod{g_i}{g_j}^4] = 3 d(d+2) .
$$
For the variance, we have
$$
\Var(\sum_{i<j} \iprod{g_i}{g_j}^4) 
= \sum_{i<j} \sum_{i'<j'} \Cov(\iprod{g_i}{g_j}^4, \iprod{g_{i'}}{g_{j'}}^4) 
= \sum_{i<j} \sum_{i'<j'} \left( \Expect[\iprod{g_i}{g_j}^4 \iprod{g_{i'}}{g_{j'}}^4] - 9 d^2 (d+2)^2 \right) .
$$
If $i,j,i',j'$ are all distinct, then the above covariance is zero.
If $|\{i,j\} \cap \{i',j'\}| = 1$, then we can compute
$$
\Expect[\iprod{g_i}{g_j}^4 \iprod{g_{i'}}{g_{j'}}^4] = 9 d(d+2)(d+4)(d+6) .
$$
If $(i,j) = (i',j')$, then we have $\Expect[\iprod{g_i}{g_j}^8] \lesssim d^4$.
Combining these cases yields
$$
\Var(\sum_{i<j} \iprod{g_i}{g_j}^4) \lesssim n^3 d^3 + n^2 d^4. 
$$

Putting it together, on the intersection of the above two high-probability events, we have
$$
\sum_{i<j} \iprod{x_i}{x_j}^4 
\lesssim \frac{1}{d^4} \sum_{i<j} \iprod{g_i}{g_j}^4 
\lesssim \frac{1}{d^4} \left( n^2 d^2 + \sqrt{n^3 d^3 + n^2 d^4} \cdot (\log(1/\gamma))^4 \right)
$$
with probability at least $1 - \gamma - n \exp(-cd)$. Choosing $\gamma = \exp(- \min \{ (nd)^{1/8}, n^{1/4} \})$ completes the proof.
\end{proof}

\begin{lemma} \label{lem:kappa-X-b1-X-bound}
Let $X$ be as defined in \prettyref{lem:X-spectral-norm-bound}. 
For the kernel $\kappa(t)=\sum_{\ell=0}^L b_\ell t^\ell$, assume that \prettyref{eq:coeffs-assumption} holds. 
Let $\kappa(X)$ denote the $n\times n$ matrix with $\kappa(X)_{ij} = \kappa(X_{ij})$ if $i \ne j$ and $\kappa(X)_{ii} = 0$. 
There is an absolute constant $c>0$ such that with probability at least $1 - n \exp(-c \min\{d, d/r^2, d^{1/3}/r^{2/3}\} ) - \exp(- \min \{ (nd)^{1/8}, n^{1/4} \})$, 
$$
\Fnorm{\kappa(X) - b_1 X }^2 \lesssim \frac{n^2 B^2 r^4}{d^2} .
$$
\end{lemma}

\begin{proof}
With probability at least $1 - \gamma$, for an absolute constant $C_1>0$, we have $|\iprod{x_i}{x_j}| \le C_1 \sqrt{\frac{\log(n/\gamma)}{d}}$ for all pairs of distinct $i,j \in [n]$ by sub-Gaussian concentration. 
On this event, by assumption \prettyref{eq:coeffs-assumption}, we obtain 
$$
\left| \sum_{\ell=3}^L b_\ell \iprod{x_i}{x_j}^\ell \right|
\le \sum_{\ell=3}^L \frac{B}{\sqrt{\ell!}} r^{\ell} C_1^\ell \left( \frac{\log(n/\gamma)}{d} \right)^{\ell/2} 
\le C_1^3 B r^3 \frac{(\log(n/\gamma))^{3/2}}{d^{3/2}}
$$
provided that $C_1 r \sqrt{\frac{\log(n/\gamma)}{d}} \le 1/2$, i.e., $\gamma \ge n \exp(-\frac{d}{4 C_1^2 r^2})$. 
Moreover, by \prettyref{lem:fourth_moment_hypercontractivity}, with probability at least 
$1 - n \exp(-cd) - \exp(- \min \{ (nd)^{1/8}, n^{1/4} \})$, 
we have $\sum_{i \neq j} \iprod{x_i}{x_j}^4 \lesssim n^2/d^2$. 
Then
\begin{align*}
\Fnorm{\kappa(X) - b_1 X }^2
&= \sum_{i\neq j} (\kappa(X) - b_1 X)_{ij}^2 \\ 
&= \sum_{i \neq j} \left( b_0 + b_2 \iprod{x_i}{x_j}^2 + \sum_{\ell=3}^L b_\ell \iprod{x_i}{x_j}^\ell \right)^2 \\
&\lesssim n^2 b_0^2 + \sum_{i \neq j} b_2^2 \iprod{x_i}{x_j}^4 + \sum_{i \neq j} \left( \sum_{\ell=3}^L b_\ell \iprod{x_i}{x_j}^\ell \right)^2 \\
&\lesssim n^2 B^2 \frac{r^4}{d^2} + B^2 r^4 \frac{n^2}{d^2} + n^2 B^2 r^6 \frac{(\log(n/\gamma))^3}{d^3} 
\lesssim n^2 B^2 \frac{r^4}{d^2} 
\end{align*}
by assumption \prettyref{eq:coeffs-assumption}, if $\gamma \ge n \exp( - d^{1/3}/r^{2/3} )$. 
Taking a union bound then completes the proof.
\end{proof}

The following is a generic lemma that is useful for evaluating the joint probability of two posterior replicas by means of a change of measure to two replicas drawn from the prior.
\begin{lemma}[Change of measure]
\label{lmm:com}
Consider $(X,A)$ from a joint distribution $P_{X,A}$. 
Let $\tilde X$ be a fresh draw from the posterior distribution $P_{X|A}$.
Let $Q_A$ be a reference density for $A$, such that $P_{A|X=x} \ll Q_A$ for every $x$. Define the likelihood ratio
$L(a|x) = \frac{P_{A|X}(a|x)}{Q_A(a)}$.
Then for any joint event $\calF$,
\[
\pprob{(X,\tilde X) \in \calF}
\leq 
2 \sqrt{
\Expect_{X\indep X^*}[g(X,X^*) \indc{(X,X^*) \in \calF}]}
\]
where $X^*$ is an i.i.d.\  copy of $X$ independent of everything else, and
\[
g(x,x^*) \triangleq \Expect_{A\sim Q_A}[L(A|x^*)L(A|x)].
\]

\end{lemma}
\begin{proof}
The posterior density is given by
\[
P_{X|A}(x|a) = \frac{P_X(x) P_{A|X}(a|x)}{P_A(a)}
= \frac{P_X(x) L(a|x)}{Z(a)},
\]
where $Z(a) \equiv \frac{P_A(a)}{Q_A(a)}$ is the likelihood ratio of the marginals.
Then
\begin{align*}
 \pprob{(X,\tilde X) \in \calF} 
= & \Expect_{X,A}\qth{\int P_{X|A}(d\tilde x|A) \indc{(X,\tilde x) \in \calF}} \\    
= & \Expect_{X,A}\qth{\int P_X(d\tilde x) \frac{L(A|\tilde x)}{Z(A)}\indc{(X,\tilde x) \in \calF}} \\    
\leq & \prob{Z(A)\leq \epsilon}
+ \frac{1}{\epsilon} \Expect_{X,A}\qth{\int P_X(d\tilde x) L(A|\tilde x)\indc{(X,\tilde x) \in \calF}}
\\    
= & \prob{Z(A)\leq \epsilon}
+ \frac{1}{\epsilon} \iint P_X(d\tilde x)P_X(d x) \indc{(x,\tilde x) \in \calF}
\underbrace{\Expect_{A\sim Q_A}[L(A|\tilde x) L(A|x)]}_{g(x,\tilde x)}.
\end{align*}
It remains to show $\prob{Z(A)\leq \epsilon} \leq \epsilon$, which, upon optimizing $\epsilon$, yields the desired result.
This is a well-known property of the likelihood ratio:
\[
\prob{Z(A)\leq \epsilon} 
= \Expect_{A\sim Q_A}\qth{\frac{P_A(A)}{Q_A(A)} \indc{\frac{P_A(A)}{Q_A(A)} \leq \epsilon}} \leq \epsilon.
\]
    
\end{proof}


We are ready to prove \prettyref{prop:posterior-concentration}.

\begin{proof}[Proof of \prettyref{prop:posterior-concentration}]
Define events
\begin{align}
\calE \eqdef \left\{ 
\Fnorm{X}^2 \lesssim \frac{n^2}{d}, \; \norm{X} \lesssim \frac{n}{d} +\sqrt{ \frac{n}{d} } , 
\; \Fnorm{\kappa(X) - b_1 X }^2 \lesssim \frac{B^2 r^4 n^2}{d^2}\right\} \label{eq:event_three_bounds}
\end{align}
and 
$$
\calF \eqdef \left\{ X \in \calE, \, \tilde X \in \calE, \, \Iprod{X}{\tilde X} \geq \frac{\delta n^2}{d} \right\} .
$$
It follows that 
\begin{equation}
\prob{\iprod{\tilde{X}}{X} \ge \delta \frac{n^2}{d}} \le \prob{X \notin \calE} + \prob{\tilde X \notin \calE} + \prob{(X, \tilde X) \in \calF} . \label{eq:EcEcF}
\end{equation}

To bound the last term in \prettyref{eq:EcEcF}, we apply \prettyref{lmm:com} with $Q_A = G(n,p)$, $P_{X,A}$ being the joint law of $(X,A)$, and $\calF$ as defined above. To compute $g$, we note that 
$$
L(A|X) = \prod_{i<j} \left(
\frac{K(X_{ij})}{p}\right)^{A_{ij}}
\left( \frac{1-K(X_{ij})}{1-p} \right)^{1-A_{ij}} ,
$$
so 
$$
L(A|X)L(A|X^*)
= \prod_{i<j} \left(
\frac{K(X_{ij})K(X_{ij}^*)}{p^2}\right)^{A_{ij}}
\left( \frac{(1-K(X_{ij})(1-K(X_{ij}^*)}{(1-p)^2} \right)^{1-A_{ij}}
$$
and
\begin{align}
g(X,X^*) &= \Expect_{A \sim Q_A}[L(A|X)L(A|X^*)|X,X^*] \nonumber \\
&= \prod_{i<j} \left(
\frac{K(X_{ij})K(X_{ij}^*)}{p} + \frac{(1-K(X_{ij})(1-K(X_{ij}^*)}{1-p} \right) \nonumber \\
&= 
\prod_{i<j}\left( 1+ \kappa (X_{ij}) \kappa (X^*_{ij})\right)
\le \exp\left( \sum_{i<j} \kappa (X_{ij}) \kappa (X^*_{ij}) \right), \label{eq:g_computation}
\end{align}
As a result of \prettyref{lmm:com}, we then obtain
\begin{equation}
\pprob{(X,\tilde X) \in \calF}
\leq 
2 \sqrt{
\Expect_{X\indep X^*}\left[\exp\left( \iprod{\kappa (X)}{ \kappa (X^*)} \right) \indc{(X,X^*) \in \calF}\right]} .
\label{eq:X-tilde-X-in-calF}
\end{equation}

Recall that $\kappa(t)=\sum_{\ell=0}^L b_\ell t^\ell$. We have
\begin{align*}
\iprod{ \kappa(X)}{\kappa(X^*)} 
& =b_1^2 \iprod{X}{X^*} \tag*{\textbf{(Term I)}} \\
& + b_1 \iprod{X}{\kappa(X^*)-b_1X^*} 
 + b_1 \iprod{X^*}{\kappa(X)-b_1X} \tag*{\textbf{(Terms II and III)}} \\
& + \iprod{\kappa(X^*)-b_1X^*}{\kappa(X)-b_1X} \tag*{\textbf{(Term IV)}} .
\end{align*}
We proceed by separately considering the four terms above, denoted by $Y_1, \ldots, Y_4$. Recall the event $\calE$ defined in \prettyref{eq:event_three_bounds}.


\paragraph{Term IV:}  By the Cauchy--Schwarz inequality, for $X, X^* \in \calE$, we have
\begin{align*}
Y_4\triangleq \iprod{\kappa(X^*)-b_1X^*}{\kappa(X)-b_1X}
\le \Fnorm{\kappa(X^*)-b_1X^*}
\Fnorm{\kappa(X)-b_1X} 
\lesssim B^2 r^4 n^2 /d^2.
\end{align*}

\paragraph{Terms II and III:} Since Terms II and III are symmetric in $X$ and $X^*$, it suffices to consider one of them. We focus on Term II. By conditioning on $X^* \in \calE$ and applying the Hanson--Wright inequality (see Exercise~6.8 in \cite{vershynin2018high} and Theorem~6 in \cite{li2024simple}), we get 
$$
\prob{  | b_1 \iprod{X}{\kappa(X^*)-b_1X^*} |  \ge t}
\le  \exp \left( -c \min\{ t^2/S, t/M\} \right),
$$
where
$$
S \eqdef \frac{b_1^2}{d} \Fnorm{\kappa(X^*)-b_1X^*}^2 \lesssim 
\frac{b_1^2 n^2   B^2 r^4}{d^3}  
$$
and 
$$
M \eqdef \frac{|b_1|}{d} \norm{ \kappa(X^*)-b_1 X^*}
\lesssim \frac{|b_1| n B r^2}{d^2} .
$$
It follows that 
\begin{align*}
\expects{\exp\left( b_1 \iprod{X}{\kappa(X^*)-b_1 X^*} \right)}{X}
& =\int_{-\infty}^\infty \exp(t) \prob{  \left| b_1 \iprod{X}{\kappa(X^*)-b_1X^*}  \right| \ge t} \diff t  \\
& \le \int_{-\infty}^{S/M} \exp\left( t - c t^2/S \right) \diff t + \int_{S/M}^\infty \exp(t - c t/M) \diff t \\
& \lesssim \exp ( C S),
\end{align*}
where $C>0$ is a large constant, and the last inequality holds under the condition that $M \lesssim 1$. To ensure $M \lesssim 1$, since $|b_1| \le Br$, it suffices to have 
\nbb{$d \gtrsim n^{1/2} B r^{3/2}$}. 

\paragraph{Term I:} Again, we condition on $X^* \in \calE$ and apply the Hanson-Wright inequality to get that 
$$
\prob{\left| b_1^2 \iprod{X}{X^*} \right| \ge t}
\le 2 \exp \left( - c \min \{t^2/S', t/M'\}\right),
$$
where 
$$
S' \eqdef \frac{b_1^4}{d} \Fnorm{X^*}^2 \lesssim \frac{b_1^4 n^2}{d^2} 
$$
and
$$
M' \eqdef \frac{b_1^2}{d} \norm{X^*}
\lesssim \frac{b_1^2}{d} \left( \frac{n}{d} + \sqrt{\frac{n}{d}} \right).
$$
It follows that 
\begin{align*}
\expects{\exp\left( b^2_1 \iprod{X}{X^*} \right) \indc{(X,X^*) \in \calF} }{X} 
& \le \int_{\delta n^2b_1^2/d}^\infty \exp(t) \prob{   b^2_1 \iprod{X}{X^*}   \ge t} \diff t  \\
& \le \int_{\delta n^2b_1^2/d}^\infty \exp\left( t - c \min\{ t^2/S', t/M'\} \right) \diff t \\
&\lesssim \exp ( - (c/2) \min\{ \delta^2 n^4 b_1^4 / (d^2 S') , \delta n^2 b_1^2 / (d M') \} ) \\
& \lesssim \exp ( - c' \min\{ \delta^2 n^2, \delta n d, \delta n^{3/2} d^{1/2} \} ),
\end{align*}
where the second-to-last inequality holds under the condition that $\delta n^2b_1^2/d \gtrsim S'$ and $M' \lesssim 1$. Since $|b_1| \le Br$, it suffices to have \nbb{$\delta \gtrsim B^2 r^2/d$ and $d \gtrsim n^{1/2} B r + n^{1/3} B^{4/3} r^{4/3}$}.

\bigskip

Combining all four terms, we get that 
\begin{align*}
&\expects{\exp \left(  \iprod{\kappa(X)}{\kappa(X^*)}\right) \indc{(X,X^*) \in \calF}}{X, X^*} \\
& = \expects{\exp\left(  \left[ Y_1 + Y_2 +Y_3+Y_4 \right]\right)\indc{(X,X^*) \in \calF}}{X, X^*} \\
 & \le e^{B^2 r^4 n^2/d^2}
 \expects{\exp\left( \left[Y_1+Y_2+Y_3\right] \right)\indc{(X,X^*) \in \calF}}{X, X^*}  \\
 & \le e^{B^2 r^4 n^2/d^2}
 \left(\expects{\exp\left( Y_1 \right) \indc{(X,X^*) \in \calF}}{X, X^*}\right)^{1/2}
 \left( \expects{ \exp\left( Y_2+Y_3 \right) \indc{(X,X^*) \in \calF}}{X, X^*}\right)^{1/2} \\
  & \lesssim \exp\left(B^2 r^4 n^2/d^2
 + C n^2 b_1^2 B^2 r^4 /d^3 - (c'/2) \min\{ \delta^2 n^2, \delta n d, \delta n^{3/2} d^{1/2} \} \right) \\
 &\le \exp\left( - (c'/4) \min\{ \delta^2 n^2, \delta n d, \delta n^{3/2} d^{1/2} \} \right),
\end{align*}
where the last inequality holds under the conditions that 
\nbb{$d \ge B^2 r^2$} (so that $B^2 r^4 n^2/d^2 \ge n^2 b_1^2 B^2 r^4 /d^3$) and \nbb{$\delta \gtrsim B r^2/d + B^2 r^4 n/d^3 + B^2 r^4 n^{1/2}/d^{5/2}$}.


\paragraph{Finishing up:} 
In summary, the above bound together with \prettyref{eq:X-tilde-X-in-calF} implies that
%
$$
\pprob{(X,\tilde X) \in \calF}
\le 2\exp\left( - (c'/8) \min\{ \delta^2 n^2, \delta n d, \delta n^{3/2} d^{1/2} \} \right).
$$
It remains to note that for the event $\calE$ defined in \eqref{eq:event_three_bounds}, by Lemmas~\ref{lem:X-spectral-norm-bound} and \ref{lem:kappa-X-b1-X-bound}, we have 
$$
\prob{X \notin \calE}=\prob{\tilde{X} \notin \calE} \le n \exp(-c \min\{d, d/r^2, d^{1/3}/r^{2/3}\} ) + 2 \exp(- c \min \{(nd)^{1/8}, n^{1/4} \}).
$$
Combining the above bounds with \eqref{eq:EcEcF}, we see that if \nbb{
$$
d \gtrsim B^2 r^2 + n^{1/2} B r^{3/2} + n^{1/2} B r + n^{1/3} B^{4/3} r^{4/3}, 
$$}
then the desired bound \prettyref{eq:posterior_probability} holds for 
\begin{align}
\delta \gtrsim B^2 r^2/d + B r^2/d + B^2 r^4 n/d^3 + B^2 r^4 n^{1/2}/d^{5/2} \label{eq:condition_delta}
\end{align}
and 
\begin{align}
\gamma \ge & 2\exp( -c \min\{ \delta^2 n^2, \delta n d, \delta n^{3/2} d^{1/2} \} ) \nonumber \\
& + 2 n \exp(-c \min\{d, d/r^2, d^{1/3}/r^{2/3}\} ) + 4 \exp(- c \min \{(nd)^{1/8}, n^{1/4} \}) . \label{eq:condition_gamma}
\end{align}
We now verify that conditions~\prettyref{eq:condition_delta}
and~\prettyref{eq:condition_gamma} hold under our choice of $\gamma = n^{-D}$ for any constant $D>0$ and 
$$
\delta \ge C \max\{ B(B+1) r^2/d , (\log n)^{1/2} / n \}
$$
for a large constant $C>0$ depending only on $D$.

If we, in addition, require that \nbb{
$$
d \gtrsim n^{1/2} B^{1/2} r + n^{1/3} B^{2/3} r^{4/3} ,
$$ }
then $B^2 r^4 n/d^3 \le B r^2/d$ and $B^2 r^4 n^{1/2}/d^{5/2} \le B r^2/d$; hence \prettyref{eq:condition_delta} holds under our choice of $\delta$. 

It remains to check~\prettyref{eq:condition_gamma} holds under our choice of $\gamma = n^{-D}$:
\begin{itemize}
\item 
$\exp(-c \delta^2 n^2) \le \exp(-c C \log n) \le n^{-D}$ if $C$ is sufficiently large depending on $D$;

\item
$\exp(-c\delta nd) \le \exp(-c C (\log n)^{1/2} d) \le n^{-D}$ if \nbb{$d \ge \sqrt{\log n}$};

\item
$\exp(-c\delta n^{3/2} d^{1/2}) \le \exp(-c C (n d \log n)^{1/2}) \ll n^{-D}$;

\item
$n \exp(-c d) \le n^{-D}$ if \nbb{$d \ge C \log n$};

\item
$n \exp(-c d/r^2) \le n^{-D}$ if $d \ge C r^2 \log n$;

\item
$n \exp(-c d^{1/3}/r^{2/3}\} ) \le n^{-D}$ if \nbb{$d \ge C r^2 (\log n)^3$};

\item 
$\exp(- c \min \{(nd)^{1/8}, n^{1/4} \}) \ll n^{-D}$. 
\end{itemize}
Putting together all the required conditions on $d$ completes the proof.

\paragraph{Linear kernel:} In the case where the kernel is linear, i.e., $\kappa(\langle x_i, x_j \rangle) = b_1 \langle x_i, x_j \rangle$, we can instead define $\calE \eqdef \left\{ X: \Fnorm{X}^2 \lesssim \frac{n^2}{d}, \; \norm{X} \lesssim \frac{n}{d} +\sqrt{ \frac{n}{d} } \right\} $, and it suffices to consider only \textbf{Term I} above. Therefore, by the same argument, if \nbb{$d \gtrsim n^{1/2} b_1 + n^{1/3} b_1^{4/3}$}, $\delta \gtrsim b_1^2/d$, 
and 
$$
\gamma \ge 2\exp( -c \min\{ \delta^2 n^2, \delta n d, \delta n^{3/2} d^{1/2} \} ) + 2 \exp(-c d\} ) + 2 \exp(- c n \}) ,
$$
then \prettyref{eq:posterior_probability} holds. 
If we further assume \nbb{$d \gtrsim \log n$} and choose
$$
\delta \ge C \max\{ b_1^2/d , (\log n)^{1/2} / n \} ,
$$
then the condition on $\gamma$ can be simplified to $\gamma = n^{-D}$ as before.
\end{proof}

\begin{remark}
    \label{rmk:recovery-gaussian}

The conclusion of \prettyref{prop:posterior-concentration} for the linear kernel continues to hold if the Bernoulli observations are replaced by Gaussian observations, namely $A_{ij} = b_1 \Iprod{x_i}{x_j} + z_{ij}$ with $z_{ij}\iiddistr N(0,1)$ for $i<j$. 
The proof proceeds verbatim with only minor modifications. In this case, the likelihood function becomes
\[
L(A\mid X)=\prod_{i<j}\exp\!\left(
b_1 A_{ij} X_{ij}-\frac{b_1^2}{2}X_{ij}^2
\right),
\]
and therefore
\[
g(X,X^*)
=\Expect_{A_{ij}\iiddistr N(0,1)}\!\left[
L(A\mid X)L(A\mid X^*) \mid X,X^*
\right]
=\exp\!\left(
b_1^2 \sum_{i<j} X_{ij}X^*_{ij}
\right).
\]
Consequently, \prettyref{eq:g_computation} holds with equality, and the remainder of the proof is identical to that for the linear kernel $\kappa(t)=b_1 t$.

\end{remark}

\subsection{Proof of Theorem~\ref{thm:poly}}
\label{sec:proof-theorem-poly}

In this subsection, we prove the detection lower bounds in Theorem~\ref{thm:poly}. Following the outline in \prettyref{sec:outline_proof}, the argument proceeds in five steps.

\subsubsection{Step 1: KL expansion} \label{sec:step-1}
Denote by $Q_A$ and $P_A$ the law of $A = (A_{ij}: 1 \leq i < j \leq n)$ under $H_0$ and $H_1$ respectively. 
Then $Q_A = \Bern(p)^{\otimes \binom{n}{2}}$.
To prove the impossibility of detection, i.e., $\TV(P_A,Q_A) = o(1)$, it suffices to show that $\KL(P_A\|Q_A)=o(1)$ by Pinsker's inequality.

\begin{lemma} \label{lem:kl-expansion}
For $t \in [n]$, denote by $A^t$ the adjacency matrix induced by the first $t$ nodes, i.e., the $t \times t$ principal minor of $A$.  
Let $x_{n+1}$ be an independent copy  of $x_1,\ldots,x_n$. Then 
\begin{equation}
\KL(P_A\|Q_A) 
\leq  \sum_{k=2}^{n-1} \binom{n}{k+1} g(k) 
    \label{eq:KLexp3}
\end{equation}
where 
\begin{equation}
g(k) \triangleq 
\Expect_{A}
\qth{
\pth{\Expect_{x_1,\ldots,x_k, x_{n+1}|A}
\qth{\prod_{i=1}^k
 \kappa(\Iprod{x_i}{x_{n+1}})
}}^2} .
\label{eq:def-g(k)}
\end{equation}
\end{lemma}

\begin{proof}
Recall the KL expansion \eqref{eq:kl-chain-rule-chi-sq}. Therein, the law of $a^t$
conditional on $A^{t-1}$ is a mixture of products of Bernoulli distributions, 
\[
P_{a^t|A^{t-1}}= 
\Expect_{x_1,\ldots,x_t|A^{t-1}}
\qth{\prod_{i=1}^{t-1} \Bern(K(\Iprod{x_i}{x_t})) },
\]
where the mixing distribution is  the \textit{posterior} distribution of the latent points $(x_1,\ldots,x_t)$ given $A^{t-1}$. 
Furthermore, 
$(x_1,\ldots,x_{t-1}, A^{t-1})$ defines an RGG with the same kernel with $t-1$ nodes and 
$x_t$ is independent of $(x_1,\ldots,x_{t-1}, A^{t-1})$.

The $\chi^2$-divergence between a mixture of products of Bernoullis and a single product Bernoulli distribution can be computed as follows by relating the $\chi^2$-divergence to   the second moment of the likelihood ratio:
Let $(p_1,\ldots,p_k) \in [0,1]^k$ be a random vector and $0<p<1$ a constant. Then
\begin{equation}
\chi^2\pth{
\Expect_{p_1,\ldots,p_k}
\qth{\prod_{i=1}^k
\Bern(p_i)} \Bigg\| 
\Bern(p)^{\otimes k} 
}
= 
\Expect_{p_1,\ldots,p_k,\tilde p_1,\ldots,\tilde p_k}
\qth{\prod_{i=1}^k 
1 + \frac{(p_i-p)(\tilde p_i-p)}{p(1-p)}}-1,
\label{eq:chi2prod}
\end{equation}
where 
$(\tilde p_1,\ldots,\tilde p_k)$ is an independent copy (replica) of $(p_1,\ldots,p_k)$.

Applying the fact \prettyref{eq:chi2prod} to each individual term in \prettyref{eq:kl-chain-rule-chi-sq}, we get
\begin{align}
  \KL(P_A\|Q_A) 
\leq  \sum_{t=2}^n 
\Expect_{A^{t-1}}
\qth{
\Expect_{x_1,\ldots,x_t,\tilde x_1,\ldots,\tilde  x_t|A^{t-1}}
\qth{\prod_{i=1}^{t-1} 
\left(1 + \kappa(\Iprod{x_i}{x_t})
\kappa(\Iprod{\tilde x_i}{\tilde x_t}) \right)
}-1
},
\label{eq:KLexp2}
\end{align}
where, conditioned on $A^{t-1}$, 
$(\tilde x_1,\ldots,\tilde x_t)$ is another independent draw from the posterior law $P_{x_1,\ldots,x_{t}|A^{t-1}}$.
So this inner conditional expectation can be expanded as follows:
\begin{align*} &\Expect_{x_1,\ldots,x_t,\tilde x_1,\ldots,\tilde  x_t|A^{t-1}}
\qth{\prod_{i=1}^{t-1} 
\left(1 + \kappa(\Iprod{x_i}{x_t})
\kappa(\Iprod{\tilde x_i}{\tilde x_t}) \right)
} \\
= & 
\sum_{S \subset [t-1]}
\Expect_{x_1,\ldots,x_{t},\tilde x_1,\ldots,\tilde x_{t}|A^{t-1}}
\qth{\prod_{i\in S}
 \kappa(\Iprod{x_i}{x_t})
\kappa(\Iprod{\tilde x_i}{\tilde x_t})}\\
= & 
\sum_{S \subset [t-1]}
\pth{\Expect_{x_1,\ldots,x_{t}|A^{t-1}}
\qth{\prod_{i\in S}
 \kappa(\Iprod{x_i}{x_t})}}^2\\
= & 
\sum_{k=0}^{t-1}
\binom{t-1}{k} \pth{\Expect_{x_1,\ldots,x_k, x_t|A^{t-1}}
\qth{\prod_{i=1}^k
 \kappa(\Iprod{x_i}{x_t})}}^2,
\end{align*}
where the second equality 
follows from the conditional independence of $x_i$'s and $\tilde x_i$'s, and the third equality is due to symmetry.
Putting these together, we arrive at 
\begin{align*}
\KL(P_A\|Q_A) 
&\leq  \sum_{t=2}^n 
\sum_{k=1}^{t-1}
\binom{t-1}{k} 
\Expect_{A^{t-1}}
\qth{
\pth{\Expect_{x_1,\ldots,x_k, x_t|A^{t-1}}
\qth{\prod_{i=1}^k
 \kappa(\Iprod{x_i}{x_t})
}}^2} \\
& \le \sum_{t=2}^n 
\sum_{k=1}^{t-1}
\binom{t-1}{k} \Expect_{A}
\qth{
\pth{\Expect_{x_1,\ldots,x_k, x_{n+1}|A}
\qth{\prod_{i=1}^k
 \kappa(\Iprod{x_i}{x_{n+1}})
}}^2} ,
\end{align*}
where $x_{n+1}$ is an independent copy of $x_1,\ldots,x_n$, and the last step follows from Jensen's inequality. 
With the definition of $g(k)$ in \eqref{eq:def-g(k)}, we obtain
\[
\KL(P_A\|Q_A)
\le \sum_{k=1}^{n-1} \sum_{t=k+1}^n \binom{t-1}{k} g(k) 
= \sum_{k=1}^{n-1} \binom{n}{k+1} g(k)  = \sum_{k=2}^{n-1} \binom{n}{k+1} g(k), \qedhere
\]
where the last equality holds because $\Expect_{x_1,x_{n+1}|A}[\kappa(\iprod{x_1}{x_{n+1}}] =0$ so that $g(1)=0$ by definition.
\end{proof}

As a side observation, we note that without analyzing the posterior law, one can drop the conditioning in \prettyref{eq:KLexp3} 
and show that \textit{for any kernel} we always have
(see \prettyref{app:dn3} for details):
\begin{equation}
\KL(P_A\|Q_A) \lesssim \frac{n^3}{d}.
\label{eq:dn3}
\end{equation}
This is consistent with the prior result using data processing inequality and the Wigner-Wishart comparison \cite{bubeck2016testing}. Clearly, this  approach does not yield kernel-dependent bound and so we proceed differently next.

\subsubsection{Step 2: Averaging over $x_{n+1}$ using generalized Wick's formula}
Next, we rewrite $g(k)$ as follows.

\begin{lemma} \label{lem:g(k)-a-ell}
For $g(k)$ defined in \prettyref{eq:def-g(k)}, we have
\begin{equation*}
g(k) =\Expect_{A}
\qth{
\pth{ \sum_{\boldsymbol \ell } a_{\boldsymbol \ell} 
\cdot \Expect_{x_1,\ldots,x_k|A}
\qth{ \prod_{1 \le i<j \le k} 
\iprod{x_i}{x_j}^{\ell_{ij}}}
}^2} ,
\end{equation*}
where 
\begin{equation}
a_{\boldsymbol \ell}
\triangleq \frac{1}{ \prod_{i<j} \ell_{ij}!}
\sum_{\mathbf{m}} 
\frac{\prod_{i=1}^k b_{\ell_i} \ell_i! (2m_i-1)!!/(2m_i)!  }{d(d+2) \cdots (d+\sum_{i=1}^k \ell_i -2)} .
\label{eq:def-a-ell}
\end{equation}
In the above, $\boldsymbol{\ell} = (\ell_{ij})$ is a multi-index whose entries take values in ${0,\ldots,L}$, and $\mathbf{m} = (m_i)$ is a multi-index whose entries are nonnegative integers bounded by $L/2$.

In the special case of a linear kernel $\kappa(t) = b_1 t$, 
if $k$ is odd, then $g(k) = 0$; if $k$ is even, then
\begin{equation*}
g(k) = \left(\frac{b_{1}^k}{d(d+2) \cdots (d+k-2)}\right)^2  \Expect_{A}
\qth{
\pth{ \sum_{\boldsymbol \ell } \Expect_{x_1,\ldots,x_k|A}
\qth{ \prod_{1 \le i<j \le k} 
\iprod{x_i}{x_j}^{\ell_{ij}}}
}^2} ,
\end{equation*}
where $\boldsymbol \ell = (\ell_{ij})$  is a multi-index 
such that $\ell_{ij} \in \{0,1\}$ and $\sum_{j \neq i} \ell_{ij} = 1$ for each $i$.  
\end{lemma}

\begin{proof}
Since $x_{n+1}$ in \prettyref{eq:def-g(k)}
is uniformly distributed over the sphere and independent of everything else, 
we can first average it using the following generalized Wick's formula (see, e.g.~\cite{vignat2008extension}): 
\begin{align}
\expects{\prod_{i=1}^k \iprod{x_{n+1}}{x_i}^{\ell_i}}{x_{n+1}} = \frac{1}{d(d+2) \cdots (d+\sum_{i=1}^k \ell_i -2)}
\sum_{\pi \in \Pi(\ell_1,\ldots, \ell_k)}
\prod_{\{i,j\} \in \pi} \Iprod{x_i}{x_j},
\end{align}
where $\Pi(\ell_1,\ldots, \ell_k)$ denotes 
the set of all pairings among $\ell_i$ copies of integer $i$ for $1 \le i \le k$.
This is obtained by representing spherical uniform as normalized Gaussian and apply the standard Wick's formula for Gaussians.

Since $\kappa(t)=\sum_{\ell=0}^L b_\ell t^\ell$, it follows that 
\begin{align}
\expects{\prod_{i=1}^k \kappa(\iprod{x_i}{x_{n+1}})}{x_{n+1}}
&=\sum_{\ell_1, \ldots, \ell_k=0}^L
\left(\prod_{i=1}^k b_{\ell_i} \right) 
\expects{\prod_{i=1}^k \iprod{x_{n+1}}{x_i}^{\ell_i}}{x_{n+1}} \nonumber \\
& = \sum_{\ell_1, \ldots, \ell_k=0}^L
 \frac{\prod_{i=1}^k b_{\ell_i}}{d(d+2) \cdots (d+\sum_{i=1}^k \ell_i -2)}
\sum_{\pi \in \Pi(\ell_1,\ldots, \ell_k)}
\prod_{\{i,j\} \in \pi} \iprod{x_i}{x_j},
\label{eq:avg_x_t}
\end{align}

We can further simplify the above expression as follows. Observe that for a given pairing $\pi \in \Pi(\ell_1, \ldots, \ell_k)$,  let $\ell_{ij}(\pi)$ denote the number of $\{i,j\}$ pairs  and $m_i(\pi)$ denote the number of $\{i,i\}$ pairs appearing in $\pi$.  Then $\ell_i = \sum_{j \neq i} \ell_{ij}(\pi) + 2m_i(\pi)$, and 
$$
\prod_{\{i,j\} \in \pi} \iprod{x_i}{x_j}
= \prod_{1 \le i<j \le k} \iprod{x_i}{x_j}^{\ell_{ij}(\pi)},
$$
where the equality holds because $\iprod{x_i}{x_i}=1$. Thus, in~\prettyref{eq:avg_x_t}, we can regroup the summands in terms of ${\boldsymbol \ell}=(\ell_{ij})$
and $\mathbf{m}=(m_i)$ as follows:
\begin{align}
& \sum_{\ell_1, \ldots, \ell_k=0}^L
 \frac{\prod_{i=1}^k b_{\ell_i}}{d(d+2) \cdots (d+\sum_{i=1}^k \ell_i -2)}
\sum_{\pi \in \Pi(\ell_1,\ldots, \ell_k)}
\prod_{\{i,j\} \in \pi} \iprod{x_i}{x_j} \nonumber \\
& = 
\sum_{\boldsymbol \ell } \prod_{1 \le i<j \le k} \iprod{x_i}{x_j}^{\ell_{ij}}
\sum_{\mathbf{m}} 
\frac{ \left(\prod_{i=1}^k b_{\ell_i}\right) | \Pi(\boldsymbol\ell, \mathbf{m})| }{d(d+2) \cdots (d+\sum_{i=1}^k \ell_i -2)},
\end{align}
where $ \Pi(\boldsymbol\ell, \mathbf{m})$ is the set of all possible pairings $\pi$ in which $\{i,j\}$ pairs appearing $\ell_{ij}$ times and $\{i,i\}$ pairs appearing $m_i$ times. 

Next, we compute $| \Pi(\boldsymbol\ell, \mathbf{m})|$. To determine a pairing $\pi \in \Pi(\boldsymbol\ell, \mathbf{m})$, among $\ell_i$ copies of integer $i$, we first choose which copies are paired with $j$ for $j \neq i$ and which copies are paired with $i$ itself. There are exactly $\binom{\ell_i}{(\ell_{ij})_{j\neq i}, 2m_i}$ many different such choices. Then among the $\ell_{ij}$ copies of $i$ and $j$ to be paired, there are $\ell_{ij}!$ different pairings. Similarly, among the $2m_i$ copies of $i$ to be paired among themselves, there are $(2m_i-1)!!$ different pairings. Therefore, in total, we have
\begin{align}
| \Pi(\boldsymbol\ell, \mathbf{m})|
& = \prod_{i=1}^k \binom{\ell_i}{(\ell_{ij})_{j\neq i}, 2m_i}  
(2m_i-1)!! \times\prod_{1 \le i<j \le k}
\ell_{ij}! \\
& = \prod_{i=1}^k  \frac{\ell_i! (2m_i-1)!!}{(2m_i)!} \times \frac{1}{\prod_{1 \le i<j \le k}
\ell_{ij}!} \label{eq:pairing_number}
\end{align}

In summary, we obtain that 
\begin{align*}
\expects{\prod_{i=1}^k \kappa(\iprod{x_i}{x_{n+1}})}{x_{n+1}}
&=\sum_{\boldsymbol \ell } \prod_{1 \le i<j \le k} 
\frac{\iprod{x_i}{x_j}^{\ell_{ij}}}{ \ell_{ij}!}
\sum_{\mathbf{m}} 
\frac{\prod_{i=1}^k b_{\ell_i} \ell_i! (2m_i-1)!!/(2m_i)!  }{d(d+2) \cdots (d+\sum_{i=1}^k \ell_i -2)} ,
\end{align*}
which implies the desired result.

In the case of a linear kernel, it suffices to note that $\ell_i = 1$ throughout that proof, so $m_i = 0$ and  
$a_{\boldsymbol \ell}
= 
\frac{b_{1}^k}{d(d+2) \cdots (d+k-2)}$.
\end{proof}

\prettyref{lem:g(k)-a-ell} together with the Cauchy--Schwarz inequality yields the following.

\begin{lemma} \label{lem:g(k)-cauchy-schwarz}
Let $g(k)$ and $a_{\boldsymbol{\ell}}$ be defined in \eqref{eq:def-g(k)} and \eqref{eq:def-a-ell} respectively. 
Also recall the notation $\ell_i \eqdef \sum_{j\neq i} \ell_{ij}$ for $\boldsymbol{\ell} = (\ell_{ij})_{i \ne j}$ where $0 \le \ell_{ij} \le L$.  
We have 
$$
g(k) \le \Big( \frac{e}{e-1} \Big)^k  \cdot 
\sum_{\boldsymbol \ell }
\beta_{\boldsymbol{\ell}} a^2_{\boldsymbol \ell} M_{\boldsymbol{\ell}} ,
$$
where
$$
\beta_{\boldsymbol \ell} \triangleq 
e^{\sum_{i=1}^k \ell_i} \cdot \left(\sum_{i=1}^k \ell_i -1\right)!! \cdot \frac{\prod_{i<j} \ell_{ij}!}{\prod_{i=1}^k \ell_i! } 
$$
and 
\begin{equation}
M_{\boldsymbol{\ell}}
\triangleq \Expect_{A}
\pth{\Expect_{x_1,\ldots,x_k|A}
\qth{ \prod_{1 \le i<j \le k} 
\iprod{x_i}{x_j}^{\ell_{ij}}}
}^2 .
\label{eq:def-M_ell}
\end{equation}
\end{lemma}

\begin{proof}
Applying \prettyref{lem:g(k)-a-ell} and the Cauchy--Schwarz inequality, we have 
$$
g(k) \le \sum_{\boldsymbol \ell } \beta_{\boldsymbol \ell}^{-1} \cdot
\sum_{\boldsymbol \ell }
\beta_{\boldsymbol \ell} a^2_{\boldsymbol \ell} 
\Expect_{A}
\pth{\Expect_{x_1,\ldots,x_k|A}
\qth{ \prod_{1 \le i<j \le k} 
\iprod{x_i}{x_j}^{\ell_{ij}}}
}^2
= \sum_{\boldsymbol \ell } \beta_{\boldsymbol \ell}^{-1} \cdot
\sum_{\boldsymbol \ell }
\beta_{\boldsymbol \ell} a^2_{\boldsymbol \ell} 
M_{\boldsymbol{\ell}} .
$$
We now bound $\sum_{\boldsymbol \ell } \beta_{\boldsymbol \ell}^{-1}$.
Note that given $(\ell_1, \ldots, \ell_k)$,  we have
\begin{align}
\sum_{\boldsymbol \ell: \ell_i=\sum_{j\neq i} \ell_{ij}} \frac{\prod_{i=1}^k \ell_i! }{\prod_{i<j} \ell_{ij}!} \le \left(\sum_{i=1}^k \ell_i -1\right)!!. \label{eq:pairing_bound}
\end{align}
To see why this is true, for the LHS, given $\boldsymbol \ell$, there are $\prod_{i=1}^k \ell_i! /\prod_{i<j} \ell_{ij}!$ distinct pairings in which $\{i,j\}$ pairs appear $\ell_{ij}$ times, according to~\prettyref{eq:pairing_number}; the RHS is an upper bound to the total number of pairings among $\ell_i$ copies of integer $i$ for $1 \le i \le k$. 
Therefore, 
$$
\sum_{\boldsymbol \ell } \beta_{\boldsymbol \ell}^{-1}
=\sum_{(\ell_1, \ldots, \ell_k)}
\sum_{\boldsymbol \ell: \ell_i=\sum_{j\neq i} \ell_{ij}}
\frac{e^{-\sum_{i=1}^k \ell_i}}{
(\sum_{i=1}^k \ell_i -1)!!}
\frac{\prod_{i=1}^k \ell_i!}{\prod_{i<j} \ell_{ij}!}
\le \sum_{(\ell_1, \ldots, \ell_k)}
e^{-\sum_{i=1}^k \ell_i}
\le \frac{1}{(1-1/e)^k} .
$$
The desired result follows.
\end{proof}

To proceed from \prettyref{lem:g(k)-cauchy-schwarz}, we next bound the combinatorial coefficients $a_{\boldsymbol \ell}^2$. 

\begin{lemma} \label{lem:a_ell-bounds}
Assume \prettyref{eq:coeffs-assumption} and \nbb{$d \ge 4 r^2$}. The quantities $a_{\boldsymbol \ell}$ defined in \prettyref{eq:def-a-ell} satisfy 
\begin{equation*}
| a_{\boldsymbol \ell} |
\le \frac{(3B)^{k}}{  \prod_{i<j} \ell_{ij}!}   \prod_{i=1}^k \sqrt{\big(\sum_{j\neq i} \ell_{ij} \big)!} \cdot \left( \frac{2 r^2}{d} \right)^{\sum_{j\neq i} \ell_{ij} /2+  \cdot \indc{\sum_{j\neq i} \ell_{ij}=0}} .
\end{equation*}
\end{lemma}

\begin{proof}
By the definition of 
$a_{\boldsymbol \ell}$ given in~\prettyref{eq:def-a-ell}, 
\begin{align*}
|a_{\boldsymbol \ell}|
& \le \frac{1}{ \prod_{i<j} \ell_{ij}!}
\sum_{\mathbf{m}} 
\frac{\prod_{i=1}^k |b_{\ell_i}| \ell_i! (2m_i-1)!!/(2m_i)!  }{d(d+2) \cdots (d+\sum_{i=1}^k \ell_i -2)} \\
&\le  \frac{1}{ \prod_{i<j} \ell_{ij}!}
\sum_{\mathbf{m}} 
\frac{\prod_{i=1}^k B r^{\ell_i} (r^2/d)^{\indc{\ell_i=0}} \sqrt{\ell_i!} (2m_i-1)!!/(2m_i)!  }{d^{\sum_{i=1}^k \ell_i /2}}\\
& \le \frac{1}{ \prod_{i<j} \ell_{ij}!} B^k
\sum_{\mathbf{m}} 
\prod_{i=1}^k \frac{(2m_i-1)!!}{(2m_i)!}  \sqrt{\ell_i!}
\left( \frac{r^2}{d} \right)^{\ell_i/2+ \indc{\ell_i=0}}.
\end{align*}
Since $(\sum_{j\neq i} \ell_{ij} +2m)! \le 2^{\sum_{j\neq i} \ell_{ij} +2m} (\sum_{j\neq i} \ell_{ij})! (2m)!$ and $(2m-1)!! \le \sqrt{(2m)!}$, we obtain
\begin{align*}
\sum_{\mathbf{m}} 
& \prod_{i=1}^k \frac{(2m_i-1)!!}{(2m_i)!}  \sqrt{\ell_i!}
\left( \frac{r^2}{d} \right)^{\ell_i/2+ \indc{\ell_i=0}} \\
&= \prod_{i=1}^k  
\sum_{m} \frac{(2m-1)!!}{(2m)!}  \sqrt{(\sum_{j\neq i} \ell_{ij} +2m)!}
\left( \frac{ r^2}{d} \right)^{\sum_{j\neq i} \ell_{ij}/2+ m+ \indc{\sum_{j\neq i} \ell_{ij}+2m=0}} \\
&\le \prod_{i=1}^k \sum_m \sqrt{(\sum_{j\neq i} \ell_{ij})!}
\left( \frac{2 r^2}{d} \right)^{\sum_{j\neq i} \ell_{ij}/2+ m+ \indc{\sum_{j\neq i} \ell_{ij}+2m=0}} \\
& \le 3^k \prod_{i=1}^k \sqrt{(\sum_{j\neq i} \ell_{ij})!}
\left( \frac{2 r^2}{d} \right)^{\sum_{j\neq i} \ell_{ij}/2 + \indc{\sum_{j\neq i} \ell_{ij}=0}}
\end{align*}
assuming \nbb{$d \ge 4 r^2$}. 
The claimed bound on $|a_{\boldsymbol \ell}|$ then follows.
\end{proof}

Moreover, for the special case of $\boldsymbol{\ell}=\boldsymbol{0}$, we can improve the above upper bound as follows. 

\begin{lemma} \label{lem:a_0-bound}
Assume \prettyref{eq:coeffs-assumption}, \nbb{$d \ge 8 r^2$, and $d \ge kL$}. Then we have
$$
|a_{\boldsymbol 0}|
\le \frac{B^{k} (5 r^2)^{k}}{d^{k+1}}.
$$
\end{lemma}

\begin{proof}
By definition, 
$$
a_{\boldsymbol 0}
\triangleq 
\sum_{\mathbf{m}} 
\frac{\prod_{i=1}^k b_{2m_i}  (2m_i-1)!! }{d(d+2) \cdots (d+2\sum_{i=1}^k m_i -2)}.
$$
Define
$$
\tilde{a}_{\boldsymbol 0}
\triangleq  
\left( \sum_{m} 
\frac{ b_{2m}  (2m-1)!! }{d(d+2) \cdots (d+2m -2)} \right)^{k} =\left( \expect{\kappa(\iprod{x_1}{x_2})}\right)^k=0.
$$
Note that 
$$
a_{\boldsymbol 0}
-\tilde{a}_{\boldsymbol 0}
= \sum_{\mathbf{m}} 
\prod_{i=1}^k b_{2m_i}  (2m_i-1)!! 
\left( 
\frac{1}{d(d+2) \cdots (d+2\sum_{i=1}^k m_i -2)}
- \prod_{i=1}^k \frac{1}{d(d+2) \cdots (d+2 m_i -2)}
\right).
$$
It is easy to check that 
\begin{align*}
&\left|\frac{1}{d(d+2) \cdots (d+2\sum_{i=1}^k m_i -2)}
- \prod_{i=1}^k \frac{1}{d(d+2) \cdots (d+2 m_i -2)}\right| \\
&\le \frac{d(d+2) \cdots (d+2\sum_{i=1}^k m_i -2) - d^{\sum_{i=1}^k m_i}}{d^{2 \sum_{i=1}^k m_i}} \\
&\le \frac{1}{d^{2 \sum_{i=1}^k m_i}} \sum_{\ell \ge 1} \binom{\sum_{i=1}^k m_i}{\ell }\left( 2 \sum_{i=1}^k m_i \right)^{\ell} d^{\sum_{i=1}^k m_i - \ell} \\
&\le  \frac{2 \sum_{i=1}^k m_i}{d^{\sum_{i=1}^k m_i +1}}  \sum_{\ell \ge 1} \binom{\sum_{i=1}^k m_i}{\ell } \le  \frac{2^{2\sum_{i=1}^k m_i}}{d^{\sum_{i=1}^k m_i +1}},
\end{align*}
where the second to the last inequality holds by the assumption that $d \ge kL \ge \sum_{i=1}^k \ell_i \ge 2\sum_{i=1}^k m_i$.
It then follows from \prettyref{eq:coeffs-assumption} that 
\begin{align*}
\left|a_{\boldsymbol 0}
-\tilde{a}_{\boldsymbol 0}\right|
&\le  
\sum_{\mathbf{m}} 
\left( \prod_{i=1}^k \frac{B r^{2m_i}}{\sqrt{(2m_i)!}} 
(r^2/d)^{\indc{m_i=0}}
(2m_i-1)!! \right) 
\frac{2^{2\sum_{i=1}^k m_i}}{d^{\sum_{i=1}^k m_i +1}}
\\
&\le \frac{ B^k}{d} 
\sum_{\mathbf{m}} 
(r^2/d)^{\sum_{i=1}^k (m_i + \indc{m_i=0})} 2^{2 \sum_{i=1}^k m_i} \\
&= \frac{ B^k}{d} \left( \sum_{m} (r^2/d)^{m + \indc{m=0}} 2^{2 m} \right)^k \\
&\le \frac{ B^k (5 r^2)^{k}}{d^{k+1}}
\end{align*}
if \nbb{$d \ge 8 r^2$}. 
Therefore, we obtain the desired bound on $|a_{\boldsymbol 0}|$.
\end{proof}




\subsubsection{Step 3: Bounding the posterior moments}



Our next goal is to bound the posterior moments $M_{\boldsymbol{\ell}}$ defined by \eqref{eq:def-M_ell} in \prettyref{lem:g(k)-cauchy-schwarz}. We will derive two bounds. The first is a simple looser bound that holds for any $\boldsymbol{\ell}$. The second bound is sharper but applies only when $\boldsymbol{\ell}$ has small total sum; its derivation relies crucially on our posterior overlap analysis in \prettyref{sec:posterior-analysis}. This refinement will be the key when we bound $g(k)$ for small $k$ in the next step.

We begin with the crude bound. By Jensen's inequality, we can move the conditional expectation outside the square to obtain a simple upper bound: 
\begin{align}
M_{\boldsymbol{\ell}}
 \le \Expect_{x_1,\ldots,x_k}
\qth{ \prod_{1 \le i<j \le k} 
\iprod{x_i}{x_j}^{2\ell_{ij}}}. \label{eq:M_ell_upper_bound_simple}
\end{align}
Since $\iprod{x_i}{x_j}$ typically has magnitude on the order of $1/\sqrt{d}$, we expect this upper bound to scale roughly as $d^{-\sum_{i<j}\ell_{ij}}$. 
To make this heuristic precise, we first prove a technical lemma.
\begin{lemma}
\label{lem:gaussian-vector-product-expectation}
Let $g_1, \dots, g_k$ be i.i.d.\ $N(0, I_d)$ random vectors.
For a multigraph $H$ on $[k]$, let $E(H)$ denote the edge set of $H$ where edges between the same pair of vertices are distinct elements. 
Then we have
\begin{equation}
\label{eq:gaussian-vector-product-expectation-equality}
\expect{\prod_{(i,j) \in E(H)} \iprod{g_i}{g_j}} = \sum_{\pi \in \Pi(H)} (d)_{|\pi|} \prod_{B \in \pi} \prod_{i \in V(B)} \mu_{\deg_B(i)} ,
\end{equation}
where $\Pi(H)$ denotes the set of partitions of $E(H)$, $(d)_m$ denotes the falling factorial, $V(B)$ denotes the vertex set of (the graph induced by) $B$, $\deg_B(i)$ denotes the degree of vertex $i$ in $B$, and $\mu_m$ denotes the $m$th moment of $N(0,1)$. Moreover,
$$
\expect{\prod_{(i,j) \in E(H)} \iprod{g_i}{g_j}} \le d^{|E(H)|/2} \prod_{i \in V(H)} (\deg_H(i) - 1)!! .
$$
\end{lemma}

\begin{proof}
We have 
\begin{equation*}
\expect{\prod_{(i,j) \in E(H)} \iprod{g_i}{g_j}}
= \expect{\prod_{(i,j) \in E(H)} \sum_{k=1}^d (g_i)_k (g_j)_k} 
= \sum_{\substack{ \mathbf{k} = (k_\gamma)_{\gamma \in E(H)} : \\ 1 \le k_\gamma \le d }} \expect{\prod_{\gamma \in E(H)} (g_{i_\gamma})_{k_\gamma} (g_{j_\gamma})_{k_\gamma}} ,
\end{equation*}
where we use $\gamma$ to denote an edge between $i_\gamma$ and $j_\gamma$ to emphasize that the edges between the same pair of vertices are distinct. 
For any fixed $\mathbf{k}$, we have a partition $\pi = \pi(\mathbf{k}) \in \Pi(H)$ according to the values of $k_{\gamma}$. That is, any two edges $\gamma$ and $\gamma'$ are in the same edge set $B \in \pi$ if and only if $k_\gamma = k_{\gamma'}$. 
Conversely, for any partition $\pi \in \Pi(H)$, there are $(d)_{|\pi|}$ multi-indices $\mathbf{k}$ such that $\pi(\mathbf{k}) = \pi$, because it amounts to choosing the value of $k_\gamma$ for one $\gamma \in B$ for every $B \in \pi$. Moreover, note that the quantity $\expect{\prod_{\gamma \in E(H)} (g_{i_\gamma})_{k_\gamma} (g_{j_\gamma})_{k_\gamma}}$ depends only on $\pi(\mathbf{k})$ and is equal to 
$\prod_{B \in \pi(\mathbf{k})} \prod_{i \in V(B)} \mu_{\deg_B(i)}$.
Therefore, \eqref{eq:gaussian-vector-product-expectation-equality} holds.

Furthermore, since $\mu_m = 0$ if $m$ is odd, in \eqref{eq:gaussian-vector-product-expectation-equality}, it suffices to consider $\pi \in \Pi(H)$ such that every vertex of every $B \in \pi$ has an even degree; let $\Pi'(H)$ be the subset of $\Pi(H)$ consisting of all such $\pi$. Then we have
\begin{align*}
\expect{\prod_{(i,j) \in E(H)} \iprod{g_i}{g_j}} &= \sum_{\pi \in \Pi'(H)} (d)_{|\pi|} \prod_{B \in \pi} \prod_{i \in V(B)} (\deg_B(i) - 1)!! \\
&= \sum_{\substack{ \mathbf{k} = (k_\gamma)_{\gamma \in E(H)} : \\ 1 \le k_\gamma \le d, \, \pi(\mathbf{k}) \in \Pi'(H) }} \prod_{B \in \pi(\mathbf{k})} \prod_{i \in V(B)} (\deg_B(i) - 1)!! .
\end{align*}

Note that $(\deg_B(i) - 1)!!$ is the number of pairings of edges in $B$ incident to $i$. 
Hence, the right-hand side above is counting the number of elements of $\calH$ defined as follows. Each element of $\calH$ is a decoration $(\mathbf{k}, \mathbf{p})$ on the multigraph $H$. More precisely, each edge $\gamma$ of the multigraph $H$ is labeled by $k_\gamma \in [d]$. Let $\mathbf{k} = (k_\gamma)_{\gamma \in E(H)}$. Moreover, for each vertex $i$ of $H$, and for each set of edges $\gamma$ incident to $i$ with $k_\gamma$ taking the same value $\ell \in [d]$, there is a pairing $p_{i,\ell}$ of these edges. Let $\mathbf{p} = (p_{i,\ell})_{i \in V(H), \, \ell \in [d]}$. 

Alternatively, we can define the elements $(\mathbf{k}, \mathbf{p})$ of $\calH$ as follows. For each vertex $i$ of the multigraph $H$, let $p_i$ be a pairing of all edges incident to $i$. The edges $\gamma$ of $H$ are then labeled by $k_\gamma \in [d]$ such that the following compatibility condition is satisfied: if two edges $\gamma$ and $\gamma'$ are both incident to some vertex $i$ and are paired in $p_i$, then $k_\gamma = k_{\gamma'}$. 
We can then split the pairing $p_i$ into a list of pairings $(p_{i,\ell})_{\ell \in [d]}$ according to the values of $k_\gamma \in [d]$ for edges $\gamma$ incident to $i$. 
Finally, let $\mathbf{k} = (k_\gamma)_{\gamma \in E(H)}$ and $\mathbf{p} = (p_{i,\ell})_{i \in V(H), \, \ell \in [d]}$. 

The second definition of $\calH$ allows us to bound $|\calH|$ as follows. There are $\prod_{i \in V(H)} (\deg_H(i) - 1)!!$ pairings of edges $(p_i)_{i \in V(H)}$ over all the vertices. 
For any fixed $(p_i)_{i \in V(H)}$, there is a maximal partition $\pi$ of $E(H)$ such that if edges $\gamma$ and $\gamma'$ are paired in some $p_i$, then they must belong to the same block $B \in \pi$. 
We claim that any compatible labeling $\mathbf{k} = (k_\gamma)_{\gamma \in E(H)}$ must assign the same value to $k_\gamma$ for all $\gamma$ in the same block $B$. 
If not, we can further partition $B$ into subsets according to the values taken by $(k_\gamma)_{\gamma \in B}$, and so the partition $\pi$ can be further refined, contradicting its maximality. 
Note that $|B| \ge 2$ for all $B \in \pi$, so $|\pi| \le |E(H)|/2$. Consequently, there are at most $d^{|E(H)|/2}$ choices of $\mathbf{k} = (k_\gamma)_{\gamma \in E(H)}$ where $k_\gamma \in [d]$. It follows that
$$
|\calH| \le \prod_{i \in V(H)} (\deg_H(i) - 1)!! \cdot d^{|E(H)|/2} ,
$$
which completes the proof.
\end{proof}

Leveraging Lemma \ref{lem:gaussian-vector-product-expectation}, we can then further upper-bound the RHS of~\prettyref{eq:M_ell_upper_bound_simple}, yielding our coarse first bound on  $M_{\boldsymbol{\ell}}$. 

\begin{lemma} \label{lem:M_ell-naive-bound}
Let $\ell_i \eqdef \sum_{j:j\neq i} \ell_{ij}$ and
$\ell \eqdef \sum_{i<j}\ell_{ij}$. It holds that 
$$
M_{\boldsymbol{\ell}}
\le 
d^{-\ell} \prod_{i=1}^k (2\ell_i -1 )!! .
$$

In the special case of a linear kernel $\kappa(t) = b_1 t$,  
we have 
$$
M_{\boldsymbol{\ell}}
\le 
d^{-k/2} .
$$
\end{lemma}

\begin{proof}
To upper-bound the RHS of~\prettyref{eq:M_ell_upper_bound_simple}, write $x_i=g_i/\norm{g_i}$ for $i \in [k]$, where $g_1, \dots, g_k$ are i.i.d.\ $N(0, I_d)$ random vectors.
Let $H$ be the multigraph on $[k]$ where the edge between vertices $i$ and $j$ has multiplicity $2 \ell_{ij}$. 
Then
\begin{align}
\expect{\prod_{(i,j) \in E(H)} \iprod{g_i}{g_j}}
& =
\expect{\prod_{(i,j) \in E(H)} \norm{g_i} \norm{g_j}\iprod{x_i}{x_j}} \notag \\
& = \expect{\prod_{v \in V(H)}\norm{g_v}^{\deg_H(v)}}
\expect{\sum_{(i,j) \in E(H)} \iprod{x_i}{x_j}} \notag \\
& = \prod_{v \in V(H)} 
\left( d(d+2) \cdots (d+\deg_H(v)-2)\right)\expect{\sum_{(i,j) \in E(H)} \iprod{x_i}{x_j}} , \label{eq:spherical-to-gaussian-moment}
\end{align}
where $\deg_H(v)$ denotes the degree of vertex $v$ in $H$. 
Therefore, 
\begin{equation*}
\expect{\prod_{(i,j) \in E(H)} \iprod{x_i}{x_j}}
\le d^{-|E(H)|} \expect{\prod_{(i,j) \in E(H)} \iprod{g_i}{g_j}} 
\le d^{-|E(H)|/2} \prod_{i \in V(H)} (\deg_H(i) - 1)!! ,
\end{equation*}
where the second inequality follows from \prettyref{lem:gaussian-vector-product-expectation}.
Plugging the above bound into~\prettyref{eq:M_ell_upper_bound_simple} yields the claimed upper bound.

In the special case of linear kernels, it suffices to note that
$\ell_i = 1$ and $\ell = k/2$.
\end{proof}

Next, we derive a sharper second bound, which applies when $\ell$ is a fixed constant. Its proof relies crucially on the posterior-overlap analysis in \prettyref{sec:posterior-analysis}, in particular, \prettyref{cor:posterior-moments}. As an example, consider a linear kernel with $b_1=O(1)$. When $d \ll n/\sqrt{\log n}$,  the lemma below yields an upper bound on 
$M_{\boldsymbol{\ell}}$ that scales as $d^{-k}$, which decays much faster than the crude upper bound $d^{-k/2}$ from \prettyref{lem:M_ell-naive-bound} above.  

\begin{lemma}
\label{lem:M_ell-better-bound}
Let 
$\ell \eqdef \sum_{i<j}\ell_{ij}$ and $v \eqdef \sum_{i=1}^k \indc{\ell_i >0}$.
Assume \eqref{eq:coeffs-assumption} and \nbb{ 
$$
d \gtrsim B^2 r^2 + n^{1/2} B r^{3/2} + n^{1/2} (B + B^{1/2}) r + n^{1/3} (B^{4/3} + B^{2/3}) r^{4/3} + \log n + r^2 (\log n)^3.
$$}
For any constant $C_1 > 0$, there is a constant $C_2 > 0$ depending only on $C_1$ such that if $\ell \le C_1$, then
\begin{equation*}
M_{\boldsymbol{\ell}}
\le C_2 \left( \max \bigg\{ \frac{B(B+1) r^2 n^2}{d^2} , \frac{n (\log n)^{1/2}}{d} \bigg\} \right)^\ell \frac{1}{n^v} .
\end{equation*}

In the special case of a linear kernel $\kappa(t) = b_1 t$, assume 
\nbb{$d \gtrsim n^{1/2} b_1 + n^{1/3} b_1^{4/3} + \log n$.} 
For any constant $C_1 > 0$, there is a constant $C_2 > 0$ depending only on $C_1$ such that if $k \le C_1$, then
\begin{equation*}
M_{\boldsymbol{\ell}}
\le C_2 \left( \max \bigg\{ \frac{b_1^2}{d^2} , \frac{(\log n)^{1/2}}{n d} \bigg\} \right)^{k/2} .
\end{equation*}
\end{lemma}

\begin{proof}
Let $\tilde{X}$ be a random draw from the posterior distribution $P_{ X|A}$. It follows that 
\begin{align}
\expect{\iprod{X}{\tilde{X}}^\ell}
& = \expect{ \left( 2 \sum_{i<j} X_{ij} \tilde{X}_{ij} \right)^\ell } \notag \\
&= 2^\ell \sum_{i_1<j_1} \cdots \sum_{i_\ell<j_\ell} \expect{ X_{i_1j_1} \cdots X_{i_\ell j_\ell} \tilde{X}_{i_1 j_1} \cdots \tilde{X}_{i_\ell j_\ell} } \notag \\
& \ge 2^\ell \binom{n}{v} 
\expect{ \prod_{i<j} X_{ij}^{\ell_{ij}} \tilde{X}_{ij}^{\ell_{ij}} }. \label{eq:X-tilde-X-symmetry-bound}
\end{align}
To see why the last inequality holds, first note that 
$$\expect{ X_{i_1j_1} \cdots X_{i_\ell j_\ell} \tilde{X}_{i_1 j_1} \cdots \tilde{X}_{i_\ell j_\ell}\mid A} = 
(\expect{ X_{i_1j_1} \cdots X_{i_\ell j_\ell} \mid A})^2 \ge 0.
$$
Moreover, for any injective mapping $f:[v] \to [n]$, by symmetry, 
$$
\expect{ \prod_{i<j} X_{ij}^{\ell_{ij}}\tilde{X}_{ij}^{\ell_{ij}} } = \expect{ \prod_{i<j} X_{f(i)f(j)}^{\ell_{ij}}\tilde{X}_{f(i)f(j)}^{\ell_{ij}} },
$$
and there are at least $\binom{n}{v}$ distinct $f$. 
%
Combining \eqref{eq:def-M_ell}, \eqref{eq:X-tilde-X-symmetry-bound}, and \prettyref{cor:posterior-moments} then yields that 
\begin{align*}
M_{\boldsymbol{\ell}}
= \expect{\prod_{i<j} X_{ij}^{\ell_{ij}}\tilde{X}_{ij}^{\ell_{ij}} } 
\le \frac{1}{\binom{n}{v} } \expect{\iprod{X}{\tilde{X}}^\ell}
\le C_2 \left( \max \bigg\{ \frac{B(B+1) r^2 n^2}{d^2} , \frac{n (\log n)^{1/2}}{d} \bigg\} \right)^\ell \frac{1}{\binom{n}{v}} .
\end{align*}
Since $v \le 2\ell \le 2C_1$, we have $\binom{n}{v} \ge (n/v)^v \ge n^{v}/(2C_1)^{2C_1}$. Absorbing $(2C_1)^{2C_1}$ into the constant $C_2$ yields the desired bound. 

In the case of a linear kernel, it suffices to note that $\ell_i = 1$, $\ell = k/2$, and $v = k$. 
\end{proof}




\subsubsection{Step 4: Controlling $g(k)$}
We now combine the above estimates to bound $g(k)$.
%
%
%
%
To apply \prettyref{lem:g(k)-cauchy-schwarz}, we fix a constant $k_0 > 0$ to be chosen later, and consider the two cases $k \ge k_0$ and $k < k_0$ separately. 

For large $k$, it suffices to apply the crude bound on
$M_{\boldsymbol{\ell}}$ from \prettyref{lem:M_ell-naive-bound}.

\begin{lemma}
\label{lem:g(k)-k-large}
There is an absolute constant $C>0$ such that the following holds for all $k \ge k_0$. 
If we assume \prettyref{eq:coeffs-assumption} and \nbb{$d^{3/2} \ge 8 e r^2 L n$}, then
$$
g(k) \le \left( \frac{ C B^2 r^2 Lk}{d^{3/2}} + \frac{C B^2 r^4}{d^2} \right)^k .
$$
\end{lemma}

\begin{proof}
We apply \prettyref{lem:g(k)-cauchy-schwarz} together with \prettyref{lem:M_ell-naive-bound} which bounds $M_{\boldsymbol{\ell}}$ and \prettyref{lem:a_ell-bounds} which states that
\begin{equation*}
a^2_{\boldsymbol \ell}
\le \frac{(3B)^{2k}}{  (\prod_{i<j} \ell_{ij}!)^2}    \prod_{i=1}^k \ell_i! \left( \frac{2 r^2}{d} \right)^{\ell_i + 2 \cdot \indc{\ell_i=0}} .
\end{equation*}
In the sequel, let $C$ denote a constant whose value may vary between lines. 
Then we obtain
\begin{align}
    g(k) & \le C^k B^{2k}
    \sum_{\boldsymbol \ell }
    e^{\sum_{i=1}^k \ell_i}\frac{(\sum_{i=1}^k \ell_i -1)!!}{\prod_{i<j} \ell_{ij}! }  \left( \frac{2 r^2}{d} \right)^{\ell_i+ 2 \cdot \indc{\ell_i=0}}
    d^{-\sum_{i=1}^k \ell_i/2} \prod_{i=1}^k (2\ell_i -1 )!! \notag \\
    & \le 
    C^k B^{2k}
    \sum_{\ell_1,\ldots, \ell_k }
    e^{\sum_{i=1}^k \ell_i}
    \frac{\left[(\sum_{i=1}^k \ell_i -1)!!\right]^2}{\prod_{i=1}^k \ell_i! } \left( \frac{2 r^2}{d} \right)^{\ell_i+ 2 \cdot \indc{\ell_i=0}}
    d^{-\sum_{i=1}^k \ell_i/2} \prod_{i=1}^k (2\ell_i -1 )!! 
    \notag \\
    & \le C^k B^{2k}
    \sum_{\ell_1, \ldots, \ell_k}
    (2e)^{\sum_{i=1}^k \ell_i} 
    \bigg[\Big(\sum_{i=1}^k \ell_i -1\Big)!!\bigg]^2 
    \left( \frac{2 r^2}{d} \right)^{\sum_{i=1}^k \ell_i+ 2\sum_{i=1}^k \indc{\ell_i=0}} d^{-\sum_{i=1}^k \ell_i/2} \label{eq:g(k)-bound-intermediate} \\
& \le 
C^k B^{2k} 
\prod_{i=1}^k 
\left[ \sum_{\ell_i}  \left( 
\frac{4 e r^2 Lk}{d^{3/2}}\right)^{\ell_i}  \left( \frac{2 r^2}{d}\right)^{2 \cdot \indc{\ell_i=0}}\right] \notag \\
& \le C^k B^{2k}
\left( \frac{ r^2 Lk}{d^{3/2}} + \frac{r^4}{d^2} \right)^k ,
\end{align}
where the second inequality follows from~\prettyref{eq:pairing_bound};
the third equality holds due to $(2\ell_i-1)!! \le 2^{\ell_i}\ell_i!$;
the fourth one uses the fact that  $\sum_{i=1}^k \ell_i \le Lk$; and the last one holds due to the assumption that 
\nbb{$d^{3/2} \ge 8 e r^2 L n \ge 8 e r^2 L k$}.
\end{proof}

For small $k$, we derive a tighter bound on $g(k)$ by further leveraging our sharper bound on $M_{\boldsymbol{\ell}}$ for small $\ell$ from Lemma \ref{lem:M_ell-better-bound}.

\begin{lemma}
\label{lem:g(k)-k-small}
There is a constant $C>0$ depending only on $k_0$ such that the following holds for all $k \le k_0$. 
If we assume \prettyref{eq:coeffs-assumption} and \nbb{
\begin{align*}
d &\ge C \Big[ L^2 + n^{1/2} r^2 + (B(B+1))^{1/4} r^{3/2} n^{1/2} + r^{4/3} n^{1/3} (\log n)^{1/6} + r^{4/3} L^{2/3} \\
& \quad + B^2 r^2 + n^{1/2} B r^{3/2} + n^{1/2} (B + B^{1/2}) r + n^{1/3} (B^{4/3} + B^{2/3}) r^{4/3} + \log n + r^2 (\log n)^3 \Big] ,
\end{align*} }
then
\begin{align*}
g(k) 
&\le C \bigg[ \frac{B^{2k} r^{4k}}{d^{2k+2}} + \frac{B^{2k+1} (B+1) r^{4k-2}}{d^{2k}} + \frac{B^{2k} r^{4k-4} \log n}{d^{2k-1} n} \\
&\qquad + \bigg(\frac{B^{5/2} (B+1)^{1/2} r^3}{d^2}\bigg)^k + \bigg( \frac{B^2 r^2 (\log n)^{1/4}}{d^{3/2} n^{1/2}} \bigg)^k + \bigg( \frac{B^{2} r^{4} L^{2}}{d^3} \bigg)^k \bigg] .
\end{align*}
\end{lemma}

\begin{proof}
Throughout the proof, we use the notation $\lesssim$ to hide a multiplicative constant that may depend only on $k_0$. 
Let us consider different ranges of $\ell=\sum_{i<j} \ell_{ij}$. Recall that $\ell_i = \sum_{j \ne i} \ell_{ij}$, $\sum_i \ell_i = 2\ell$, and 
$$
v = \sum_{i=1}^k \indc{\ell_i >0} \le \min\{2\ell, k\}.
$$

{\bf Case 0}: $\ell=0$. In this case, $\ell_{ij}=0$ for all $i<j$. We have $\beta_{\boldsymbol{0}}=1$, $a^2_{\boldsymbol{0}} \le \frac{ B^{2k} (5 r^2)^{2k}}{d^{2k+2}}
$ by \prettyref{lem:a_0-bound} if \nbb{$d \gtrsim r^2 + L$},  
and $M_{\boldsymbol{0}}=1$. 

{\bf Case 1}: $1 \le \ell \le k/2$. 
It follows from Lemmas~\ref{lem:g(k)-cauchy-schwarz}, \ref{lem:a_ell-bounds}, and \ref{lem:M_ell-better-bound} that 
\begin{align}
& \sum_{\boldsymbol \ell: 1 \le \ell \le k/2 }
\beta_{\boldsymbol \ell} a^2_{\boldsymbol \ell} M_{\boldsymbol{\ell}} \notag \\
& \lesssim B^{2k}
    \sum_{\boldsymbol \ell }
    e^{\sum_{i=1}^k \ell_i}\frac{(\sum_{i=1}^k \ell_i -1)!!}{\prod_{i<j} \ell_{ij}! } \left( \frac{2 r^2}{d} \right)^{\ell_i+ 2 \cdot \indc{\ell_i=0}} 
\cdot 
\left( \max \bigg\{ \frac{B(B+1) r^2 n^2}{d^2} , \frac{n (\log n)^{1/2}}{d} \bigg\} \right)^\ell \frac{1}{n^v} \notag \\
&\lesssim B^{2k}
    \sum_{\ell_1, \ldots, \ell_k }
    e^{\sum_{i=1}^k \ell_i}
   \frac{ \left[(\sum_{i=1}^k \ell_i -1)!!\right]^2}{\prod_{i=1}^k \ell_i! } \left( \frac{2 r^2}{d} \right)^{\ell_i+ 2 \cdot \indc{\ell_i=0}} 
    \notag \\
& \qquad \qquad \qquad \cdot
    \max \bigg\{ \bigg(\frac{\sqrt{B(B+1)} r}{d}\bigg)^{2 \ell} n^{2\ell - v} , \bigg( \frac{\sqrt{\log n}}{d} \bigg)^\ell n^{\ell-v} \bigg\}  \notag \\
& \lesssim B^{2k}
\sum_{\ell} \sum_{v}
 e^{2\ell}
    \left( (2\ell -1)!!\right)^2  \left( \frac{2 r^2}{d} \right)^{2\ell+ 2(k-v)} n^{2\ell-v}
     \notag \\
& \qquad \qquad \qquad \cdot \max \bigg\{ \bigg(\frac{\sqrt{B(B+1)} r}{d}\bigg)^{2 \ell} , \bigg( \frac{\sqrt{\log n}}{d n} \bigg)^\ell \bigg\} \sum_{\ell_1, \ldots,\ell_k:\sum_{i} \ell_i=2\ell,  \sum_{i} \indc{\ell_i>0}=v}
\frac{1}{\prod_{i=1}^k \ell_i!} . \label{eq:case-1-intermediate-bound}
\end{align}
Note that 
$$
\sum_{\ell_1, \ldots,\ell_k:\sum_{i} \ell_i=2\ell,  \sum_{i} \indc{\ell_i>0}=v}
\frac{1}{\prod_{i=1}^k \ell_i!}
\le 
\sum_{\ell_1, \ldots,\ell_k:\sum_{i} \ell_i=2\ell}
\frac{1}{\prod_{i=1}^k \ell_i!}
= \frac{k^{2\ell}}{(2\ell)!},
$$
where the last equality holds by the multinomial theorem.
Moreover, under the assumption \nbb{$d^2 \ge 8nr^4$}, we have 
$$
\sum_{1 \le v \le 2\ell} 
\left( \frac{4r^4 n}{d^2} \right)^{2\ell-v}
\le 2,
$$
Using the assumption $\ell \le k/2 \le k_0/2$ to hide all the constants depending on $k,\ell$,
we obtain
\begin{align*}
\sum_{\boldsymbol \ell: 1 \le \ell \le k/2 }
\beta_{\boldsymbol \ell} a^2_{\boldsymbol \ell} M_{\boldsymbol{\ell}} 
& \lesssim B^{2k}
\sum_{1 \le \ell \le k/2} 
 \left( \frac{2 r^2}{d} \right)^{2k-2\ell} 
    \max \bigg\{ \bigg(\frac{\sqrt{B(B+1)} r}{d}\bigg)^{2 \ell} , \bigg( \frac{\sqrt{\log n}}{d n} \bigg)^\ell \bigg\} \\
& \lesssim B^{2k} r^{4k} d^{-2k}
\sum_{1 \le \ell \le k/2} 
 \bigg[ \bigg(\frac{\sqrt{B(B+1)}}{r}\bigg)^{2 \ell} + \bigg( \frac{d \sqrt{\log n}}{n r^4} \bigg)^\ell \bigg] \\
 &\lesssim B^{2k+1} (B+1) r^{4k-2} d^{-2k} + B^{5k/2} (B+1)^{k/2} r^{3k} d^{-2k} \\
 &\quad + B^{2k} r^{4k-4} d^{-2k+1} \frac{\log n}{n} + B^{2k} r^{2k} d^{-3k/2} \bigg( \frac{\sqrt{\log n}}{n} \bigg)^{k/2} ,
\end{align*}
where the last inequality holds since the sum is dominated by either the first or the last term.

{\bf Case 2:} $k/2<\ell \le L_0$, where $L_0$ is a large constant to be specified. 
In this case, the bound \prettyref{eq:case-1-intermediate-bound} with the range $k/2<\ell \le L_0$ is still valid. 
The difference is that now $v$ can be as large as $k$. Thus, under the assumption that $d^2 \ge 8nr^4$, we have
$
\sum_{1 \le v \le k} 
\left( \frac{4r^4 n}{d^2} \right)^{k-v}
\le 2,
$
and hence
\begin{align*}
\sum_{\boldsymbol \ell: k/2 \le \ell \le L_0 } \beta_{\boldsymbol \ell} a^2_{\boldsymbol \ell} M_{\boldsymbol{\ell}} 
& \lesssim B^{2k}
\sum_{\ell} 
 \left( \frac{2 r^2}{d} \right)^{2\ell} n^{2\ell-k}
    \max \bigg\{ \bigg(\frac{\sqrt{B(B+1)} r}{d}\bigg)^{2 \ell} , \bigg( \frac{\sqrt{\log n}}{d n} \bigg)^\ell \bigg\} \\
& \lesssim B^{2k} n^{-k}
\sum_{\ell} 
    \bigg[ \bigg(\frac{\sqrt{B(B+1)} r^3 n}{d^2}\bigg)^{2 \ell} + \bigg( \frac{r^2 n^{1/2} (\log n)^{1/4}}{d^{3/2}} \bigg)^{2\ell} \bigg] \\
& \lesssim
\bigg(\frac{B^{5/2} (B+1)^{1/2} r^3}{d^2}\bigg)^k + \bigg( \frac{B^2 r^2 (\log n)^{1/4}}{d^{3/2} n^{1/2}} \bigg)^k ,
\end{align*}
where the last step holds under the assumptions \nbb{$d^2 \ge 2 \sqrt{B(B+1)} r^3 n$ and $d^{3/2} \ge r^2 n^{1/2} (\log n)^{1/4}$}. 

{\bf Case 3:} $L_0<\ell \le Lk/2$. 
Note that we always have $\ell = \frac 12 \sum_{i=1}^k \ell_i \le Lk/2$. 
In this case, we apply \prettyref{lem:M_ell-naive-bound} to bound $M_{\boldsymbol \ell }$. Analogous to the case $k \ge k_0$ (see \prettyref{eq:g(k)-bound-intermediate} more specifically), we have
\begin{align*}
& \sum_{\boldsymbol \ell:  L_0 < \ell \le Lk/2}
   \beta_{\boldsymbol \ell} a^2_{\boldsymbol \ell} M_{\boldsymbol{\ell}} \\
&\le C^k B^{2k}
    \sum_{\ell_1, \ldots, \ell_k : L_0 < \ell \le Lk/2}
    (2e)^{\sum_{i=1}^k \ell_i} 
    \bigg[\Big(\sum_{i=1}^k \ell_i -1\Big)!!\bigg]^2 
    \left( \frac{2 r^2}{d} \right)^{\sum_{i=1}^k \ell_i+ 2\sum_{i=1}^k \indc{\ell_i=0}} d^{-\sum_{i=1}^k \ell_i/2} \\
& \le C^k B^{2k}
    \sum_{L_0 < \ell \le Lk/2}
    (2e)^{2\ell} 
    [(2\ell -1)!!]^2 
    \left( \frac{2 r^2}{d} \right)^{2\ell} d^{-\ell} \binom{2\ell+k-1}{k-1} 
\end{align*}
if \nbb{$d \ge 2 r^2$}, since given $\ell$, there are at most $\binom{2\ell+k-1}{k-1}$ choices of $(\ell_1, \dots, \ell_k)$ such that $\sum_i \ell_i = 2 \ell$. 
Let $C$ denote a constant that may vary between lines.
Using the naive bound $\binom{2\ell+k-1}{k-1} \le 2^{2\ell+k-1}$, we obtain 
\begin{align*}
\sum_{\boldsymbol \ell:  L_0 < \ell \le Lk/2}
   \beta_{\boldsymbol \ell} a^2_{\boldsymbol \ell} M_{\boldsymbol{\ell}} 
& \le C^k B^{2k}
    \sum_{L_0 < \ell \le Lk/2}
    (2 e Lk)^{2\ell}  
    \left( \frac{4 r^2}{d} \right)^{2\ell} d^{-\ell} 
\le C^k B^{2k}
    \left( \frac{8 e L k r^2}{d^{3/2}} \right)^{2L_0}
\end{align*}
if \nbb{$d^{3/2} \ge 16 e r^2 Lk$}.
Choosing $L_0=k$ yields
$$
\sum_{\boldsymbol \ell:  L_0<\ell \le Lk/2 }
   \beta_{\boldsymbol \ell} a^2_{\boldsymbol \ell} M_{\boldsymbol{\ell}} 
   \le C^{k} B^{2k} r^{4k} L^{2k} d^{-3k} k^{2k} .
$$

Combining all the cases together with \prettyref{lem:g(k)-cauchy-schwarz}, we get that  for $2 \le k \le k_0$, 
\begin{align*}
g(k) 
&\lesssim \frac{B^{2k} r^{4k}}{d^{2k+2}} + \frac{B^{2k+1} (B+1) r^{4k-2}}{d^{2k}} + \frac{B^{2k} r^{4k-4} \log n}{d^{2k-1} n} \\
&\quad + \bigg(\frac{B^{5/2} (B+1)^{1/2} r^3}{d^2}\bigg)^k + \bigg( \frac{B^2 r^2 (\log n)^{1/4}}{d^{3/2} n^{1/2}} \bigg)^k + \bigg( \frac{B^{2} r^{4} L^{2}}{d^3} \bigg)^k .
\end{align*}
\end{proof}

In the special case of a linear kernel, we can similarly derive separate upper bounds on $g(k)$ in the two regimes $k \ge k_0$ and $k \le k_0$.

\begin{lemma}
\label{lem:g(k)-all-bounds-linear-kernel}
Suppose that the kernel is linear, i.e., $\kappa(t) = b_1 t$. 
Then we have
$$
g(k) \le \bigg(\frac{b_1^2 k}{d^{3/2}}\bigg)^{k}.
$$
Moreover, there is a constant $C>0$ depending only on $k_0$ such that the following holds for all $k \le k_0$. 
If 
\nbb{$d \gtrsim n^{1/2} b_1 + n^{1/3} b_1^{4/3} + \log n$,}  
then
$$
g(k) \le C \bigg( \max \bigg\{ \frac{b_1^3}{d^2} , \frac{b_1^2 (\log n)^{1/4}}{n^{1/2} d^{3/2}} \bigg\} \bigg)^k .
$$
\end{lemma}

\begin{proof}
In the case of a linear kernel, recall Lemma~\ref{lem:g(k)-a-ell}, where $\boldsymbol \ell = (\ell_{ij})$ is a multi-index such that $\ell_{ij} \in \{0,1\}$ for $1 \le i <j \le k$ and $\ell_i = \sum_{j \neq i} \ell_{ij} = 1$ for all $i$. There are $(k-1)!!$ choices of such $\boldsymbol \ell$. Therefore, we can apply Lemma~\ref{lem:g(k)-a-ell} together the Cauchy--Schwarz inequality (as in the proof of Lemma~\ref{lem:g(k)-cauchy-schwarz} but with $\beta_{\boldsymbol{\ell}} = 1$) to obtain 
\begin{equation*}
g(k) \le \left(\frac{b_{1}^k}{d(d+2) \cdots (d+k-2)}\right)^2 (k-1)!! 
\sum_{\boldsymbol \ell} M_{\boldsymbol{\ell}} 
\le \frac{b_{1}^{2k}}{d^k} (k-1)!! 
\sum_{\boldsymbol \ell} M_{\boldsymbol{\ell}} .
\end{equation*}

Next, for $k \ge k_0$, we apply Lemma~\ref{lem:M_ell-naive-bound} to obtain $M_{\boldsymbol{\ell}} \le d^{-k/2}$, so
$$
g(k) \le \frac{b_{1}^{2k}}{d^k} ((k-1)!!)^2 d^{-k/2}
\le \left(\frac{b_1^2 k}{d^{3/2}}\right)^{k} .
$$
For $k \le k_0$, we can hide all the dependencies on $k$ in a constant $C>0$ and apply the above bound on $g(k)$ together with Lemma~\ref{lem:M_ell-better-bound} to obtain
\begin{equation*}
g(k) \le C \frac{b_{1}^{2k}}{d^k} \bigg( \max \bigg\{ \frac{b_1^2}{d^2} , \frac{(\log n)^{1/2}}{n d} \bigg\} \bigg)^{k/2} .
\end{equation*}
This desired result then follows.
\end{proof}

\subsubsection{Step 5: Finishing up}
By \prettyref{lem:kl-expansion}, we have
$$
\KL(P_A\|Q_A) 
\leq  
\sum_{k=2}^{n-1} \binom{n}{k+1}
g(k) .
$$
We splits the above bound into two parts according to the size of $k$. 
Applying \prettyref{lem:g(k)-k-large} gives
\begin{align*}
\sum_{k=k_0+1}^{n-1} \binom{n}{k+1} 
g(k) 
&\le 
\sum_{k=k_0+1}^{n-1} \binom{n}{k+1} 
\left( \frac{ C B^2 r^2 Lk}{d^{3/2}} + \frac{C B^2 r^4}{d^2} \right)^k \\
&\le \sum_{k \ge k_0} \left( \frac{ C n^{1+1/k} B^2 r^2 L}{d^{3/2}} + \frac{C n^{1+1/k} B^2 r^4}{d^2} \right)^k .
\end{align*}
For the above bound to be $o(1)$, it suffices to have 
\nbb{
$$
d \gg n^{\frac{2}{3} + \frac{2}{3 k_0}} B^{4/3} r^{4/3} L^{2/3} + n^{\frac{1}{2} + \frac{1}{2 k_0}} B r^2 .
$$
}
Applying 
\prettyref{lem:g(k)-k-small}, we obtain 
\begin{align*}
&
\sum_{k=2}^{k_0} \binom{n}{k+1}
g(k) \\
&\lesssim 
\max_{2 \le k \le k_0} n^{k+1} \bigg[ \frac{B^{2k} r^{4k}}{d^{2k+2}} + \frac{B^{2k+1} (B+1) r^{4k-2}}{d^{2k}} + \frac{B^{2k} r^{4k-4} \log n}{d^{2k-1} n} \\
&\qquad \qquad \qquad \qquad + \bigg(\frac{B^{5/2} (B+1)^{1/2} r^3}{d^2}\bigg)^k + \bigg( \frac{B^2 r^2 (\log n)^{1/4}}{d^{3/2} n^{1/2}} \bigg)^k + \bigg( \frac{B^{2} r^{4} L^{2}}{d^3} \bigg)^k \bigg] \\
&\lesssim 
\max_{2 \le k \le k_0} \bigg[ \bigg(\frac{B^{2} r^{4} n^{1+1/k}}{d^{2+2/k}}\bigg)^k + \bigg(\frac{B^{2+1/k} (B+1)^{1/k} r^{4-2/k} n^{1+1/k}}{d^{2}}\bigg)^k + \bigg(\frac{B^{2} r^{4-4/k} n (\log n)^{1/k}}{d^{2-1/k}}\bigg)^k \\
&\qquad \qquad + \bigg(\frac{B^{5/2} (B+1)^{1/2} r^3 n^{1+1/k}}{d^2}\bigg)^k + \bigg( \frac{B^2 r^2 n^{1/2+1/k} (\log n)^{1/4}}{d^{3/2}} \bigg)^k + \bigg( \frac{B^{2} r^{4} L^{2} n^{1+1/k}}{d^3} \bigg)^k \bigg] .
\end{align*}
For the above bound to be $o(1)$, we need \nbb{
\begin{align*}
d &\gg 
\max_{2 \le k \le k_0} \left(B^{\frac{1}{1+1/k}} r^{\frac{2}{1+1/k}} n^{1/2} + B^{1+\frac{1}{2k}} (B+1)^{\frac{1}{2k}} r^{2-\frac{1}{k}} n^{\frac{1}{2} + \frac{1}{2 k}} + B^{\frac{2}{2-1/k}} r^{\frac{4-4/k}{2-1/k}} n^{\frac{1}{2-1/k}} (\log n)^{\frac{1}{2k-1}} \right) \\
&\quad + B^{5/4} (B+1)^{1/4} r^{3/2} n^{3/4} + B^{4/3} r^{4/3} n^{2/3} (\log n)^{1/6} + B^{2/3} r^{4/3} L^{2/3} n^{1/2} .
\end{align*}
}
We also need the conditions assumed by the lemmas: \nbb{ 
\begin{align*}
d &\gtrsim r^{4/3} L^{2/3} n^{2/3} + L^2 + n^{1/2} r^2 + (B(B+1))^{1/4} r^{3/2} n^{1/2} + r^{4/3} n^{1/3} (\log n)^{1/6} \\
& \quad + B^2 r^2 + n^{1/2} B r^{3/2} + n^{1/2} (B + B^{1/2}) r + n^{1/3} (B^{4/3} + B^{2/3}) r^{4/3} + \log n + r^2 (\log n)^3 .
\end{align*}
}


Under the assumption that $B$ is a constant, the above conditions on $d$ can be simplified to
\begin{align*}
d &\gg n^{\frac{2}{3} + \frac{2}{3 k_0}} r^{4/3} L^{2/3} + n^{\frac{1}{2} + \frac{1}{2 k_0}} r^2 + n^{1/2} r + r^{3/2} n^{3/4} \\
&\quad + \max_{2 \le k \le k_0} \left( r^{\frac{2}{1+1/k}} n^{1/2} + r^{2-\frac{1}{k}} n^{\frac{1}{2} + \frac{1}{2 k}} + r^{\frac{4-4/k}{2-1/k}} n^{\frac{1}{2-1/k}} (\log n)^{\frac{1}{2k-1}} \right) 
\end{align*}
and 
$$
d \gtrsim L^2 + \log n .
$$
The former condition above can be simplified to 
\begin{align*}
d &\gg n^{2/3 + \epsilon} r^{4/3} L^{2/3} + n^{1/2 + \epsilon} r^2 + n^{1/2} r + r^{3/2} n^{3/4} 
\end{align*}
for an arbitrarily small constant $\epsilon > 0$ if $k_0$ is taken to be sufficiently large. 
We would like this to be dominated by 
$$
d \gg n^{3/4} r^{3/2} .
$$
\begin{itemize}
\item
If $r \ge 1/\sqrt{n}$, then $n^{1/2} r \le r^{3/2} n^{3/4}$. 
\item
If $r \le d^a$, then $d \gg n^{1/2+\epsilon} r^2$ follows from $d \gg n^{1/2+\epsilon} r^{2-b} d^{ab}$, i.e., $d \gg n^{\frac{1/2+\epsilon}{1-ab}} r^{\frac{2-b}{1-ab}} = n^{3/4} r^{3/2}$ if $ab = 1/3 - 4\epsilon/3$ and $b = 1-2 \epsilon$. Hence, $a = \frac{1-4\epsilon}{3-6\epsilon} = \frac 13 - \frac{2\epsilon}{3-6\epsilon}$. 
\item
If $L \le d^c$, then $d \gg n^{2/3 + \epsilon} r^{4/3} L^{2/3}$ follows from $d \gg n^{2/3 + \epsilon} r^{4/3} d^{2c/3}$, i.e., $d \gg n^{\frac{2+3\epsilon}{3-2c}} r^{\frac{4}{3-2c}}$. This condition is dominated by $d \gg r^{3/2} n^{3/4}$ if $n^{\frac{2+3\epsilon}{3-2c}} r^{\frac{4}{3-2c}} = (n r^2)^{\frac{2}{3-2c}} n^{\frac{3\epsilon}{3-2c}} \le (n r^2)^{3/4}$, i.e., $n r^2 \ge n^{\frac{3\epsilon}{(3-2c)(3/4-2/(3-2c))}}$. Therefore, it suffices to have $c \le 1/6 - \epsilon_1$ and $r \ge n^{-1/2+\epsilon_2}$ for any constants $\epsilon_1$ and $\epsilon_2$. 
\end{itemize}
To conclude, the set of conditions we need is
\nbb{$n^{-1/2 + \epsilon} \le r \le d^{1/3 - \epsilon}$, $L \le d^{1/6 - \epsilon}$, $ d \gtrsim \log n$, $ d \gg n^{3/4} r^{3/2}$.}


Finally, let us consider the special case of a linear kernel, i.e., $\kappa(t) = b_1 t$.
By \prettyref{lem:g(k)-all-bounds-linear-kernel}, 
\begin{align*}
\sum_{k=k_0+1}^{n-1} \binom{n}{k+1} 
g(k) 
\le 
\sum_{k=k_0+1}^{n-1} \binom{n}{k+1}  \left(\frac{b_1^2 k}{d^{3/2}}\right)^{k} 
\le \sum_{k \ge k_0} \bigg(\frac{e n^{1+1/k} b_1^2}{d^{3/2}}\bigg)^{k} ,
\end{align*}
which is $o(1)$ if \nbb{
$$
d \gg n^{\frac{2}{3} + \frac{2}{3 k_0}} b_1^{4/3} .
$$}
Moreover, by Lemmas~\ref{lem:g(k)-a-ell} and~\ref{lem:g(k)-all-bounds-linear-kernel}, 
\begin{align*}
\sum_{k=2}^{k_0} \binom{n}{k+1} 
g(k) 
&\lesssim 
\max_{2 \le k \le k_0} n^{k+1} \bigg( \max \bigg\{ \frac{b_1^3}{d^2} , \frac{b_1^2 (\log n)^{1/4}}{n^{1/2} d^{3/2}} \bigg\} \bigg)^k \\
&\lesssim \max_{2 \le k \le k_0} \bigg[ \bigg(\frac{b_1^3 n^{1+1/k}}{d^2}\bigg)^k + \bigg( \frac{b_1^2 n^{1/2+1/k} (\log n)^{1/4}}{d^{3/2}} \bigg)^k \bigg] ,
\end{align*}
which is $o(1)$ if \nbb{
\begin{equation*}
d \gg b_1^{3/2} n^{3/4} + b_1^{4/3} n^{2/3} (\log n)^{1/6} .
\end{equation*} }
We also need the condition in \prettyref{lem:g(k)-all-bounds-linear-kernel} to hold:
$d \gtrsim n^{1/2} b_1 + n^{1/3} b_1^{4/3} + \log n$,
which, in view of the condition $d \gg b_1^{4/3} n^{2/3} (\log n)^{1/6}$, can be simplified to
\nbb{$$d \gtrsim n^{1/2} b_1 + \log n.$$}
Moreover, if $b_1 \ge n^{-1/2} \log n$, then we have $b_1^{4/3} n^{2/3} (\log n)^{1/6} \le n^{3/4} b_1^{3/2}$ and $n^{1/2} b_1 \le n^{3/4} b_1^{3/2}$. 
In conclusion, the set of conditions we need is \nbb{$d \gg n^{3/4} b_1^{3/2}$ and $d \gtrsim \log n$.}

\section{Proofs for detection with a general kernel}
\label{sec:general-kernel}

In this section we prove the main result, \prettyref{thm:fixed}, and its extensions in Theorems \ref{thm:lowsnr}--\ref{thm:highsnr}.

\subsection{Proof of \prettyref{thm:fixed}}

\paragraph{Upper bound: $1\ll d \ll n^{3/4}$.}
For detection, we consider the signed triangle count as the test statistic. To  apply \prettyref{thm:ub}, we compute the trace of the operator $\kappa$ and show that  $\tr(\kappa^3) = \Theta(d^{-2})$ 
and $\tr(\kappa^4) = O(d^{-3})$.

Expand $\kappa$ under the Gegenbauer basis as
\[
\kappa(t) = \sum_{k=1}^\infty \alpha_k C_k^\lambda(t),
\]
where the constant term vanishes because $\kappa$ is centered. Applying \prettyref{lmm:rodrigues} and the fact that $\Iprod{x_1}{x_2}$ is $O(\frac{1}{d})$-subgaussian, we have
\begin{equation}
\alpha_1 = \frac{\kappa'(0)+o_{d}(1)}{d}, 
\quad
|\alpha_k| \leq C d^{-k}, k=2,3,4, 
\quad
|\alpha_k| \leq \frac{d+2k-2}{d-2}
\frac{\|\kappa\|_{L_2(\mu)}}{\sqrt{\dim(\calH_d^k)}}, k\geq 5, 
\label{eq:gegenbauer-coefficient-bound}
\end{equation}
for some constant $C$.
Note that $\|\kappa\|_{L_2(\mu)}^2 = 
\frac{1}{p(1-p)}\var(K(\Iprod{x_1}{x_2}))$,
where 
$\var(K(\Iprod{x_1}{x_2})) \leq \Expect[K(\Iprod{x_1}{x_2})^2]
\leq \Expect[K(\Iprod{x_1}{x_2})]=p$
and similarly 
$\var(K(\Iprod{x_1}{x_2})) \leq 1-p$.
We get $\|\kappa\|_{L_2(\mu)}^2 \leq \frac{1}{\max\{p,1-p\}}\leq 2$.

Without loss of generality, assume that 
$\kappa'(0)>0$.
Recalling the eigenvalue and 
multiplicity formulas in \prettyref{eq:eigenvalues-kappa} and \prettyref{eq:dim-harmonics}, 
we have 
\begin{align*}
    \tr(\kappa^3) = 
    \sum_{k\geq 1}^\infty \pth{\frac{d-2}{d-2+2k} \alpha_k}^3
    \dim(\calH_d^k)
\geq c d^{-2}
- 2^{3/2} \sum_{k\geq 5}^\infty 
    \dim(\calH_d^k)^{-1/2}
\end{align*} 
for some constant $c>0$.
We claim that the remainder sum is $o(d^{-2})$, which shows 
$\tr(\kappa^3) = \Theta(d^{-2})$. Indeed, since $\dim(\calH_d^k)=
\frac{d-2+2k}{d-2} 
\binom{d-3+k}{k} \geq \binom{d-3+k}{k}$, it suffices to consider 
$\sum_{k\geq 5} a_k$, where $a_k = 1/\sqrt{\binom{d-3+k}{k}}$. 
Note that $a_{k+1}/a_k = \sqrt{\frac{k}{d-2+k} 
}
\le \sqrt{\frac{d}{2d-2} 
} \le \sqrt{\frac{5}{8}}<1$ for $5 \le k \le d$. Thus $\sum_{k=5}^d a_k \leq C a_5 = O(d^{-5/2})$.
For $k>d$,
$\binom{d-3+k}{k} = \binom{d-3+k}{d-3} \geq \frac{k^{d-3}}{(d-3)!}$. 
Thus $
\sum_{k\geq d} a_k \leq 
\sqrt{(d-3)!}
\sum_{k\geq d} k^{-(d-3)/2}$.
Since $\sum_{k\geq d} k^{-(d-3)/2} \leq 
\int_d^{\infty} x^{-(d-3)/2} dx
= \frac{2}{d-5} d^{-(d-5)/2}$. 
Applying Stirling's approximation, we conclude that 
$\sum_{k\geq d} a_k \leq e^{-c' d}$ for some absolute constant $c'>0$.


Similarly, 
we have $\tr(\kappa^4) = O(d^{-3})$.
Finally, using \prettyref{thm:ub}, the signed triangle count $T$ satisfies
$\Expect_Q[T]=0$
and 
$\Var_Q[T]=O(n^3)$ under the \ER model,
and 
$|\Expect_P[T]| = \Theta(n^3 \tr(\kappa^3))$
and 
$\Var_P[T]=O(n^3+n^4 \tr(\kappa^4))$ under the RGG.
Thus,
\begin{equation}
|\Expect_P[T]-\Expect_Q[T]|^2 
\gg 
\Var_P[T]+\Var_Q[T],
\label{eq:mean-sep}
\end{equation}
provided that $d\ll n^{3/4}$.
This implies the desired strong detection
$\TV(P,Q) \to 1$.

\paragraph{Lower bound: $d \gg n^{3/4}$.}
We aim to apply \prettyref{thm:poly} for a polynomial kernel together with a polynomial approximation argument to deal with a general kernel $\kappa$.

Let $x_1,\ldots,x_n$ be i.i.d. 
Given two kernels $K$ and $\tilde K$, denote the induced distributions on RGG $G$ by $P$ and $\tilde P$.
Let 
$p_{ij}=K(x_i,x_j)$ and 
$\tilde p_{ij}=\tilde{K}(x_i,x_j)$. By conditioning on the coordinates and the convexity of $\TV$,
\begin{align*}
    \TV(P,\tilde P) 
    \leq&  \Expect\qth{\TV\pth{\prod_{i<j} \Bern(p_{ij}), \prod_{i<j} \Bern(\tilde p_{ij})}} \\
     \leq&  \sum_{i<j} \Expect\qth{\TV\pth{\Bern(p_{ij}), \Bern(\tilde p_{ij})}} 
     =  \binom{n}{2} \Expect\qth{|p_{12}-\tilde p_{12}|} 
    \leq \binom{n}{2} \sqrt{\Expect\qth{|p_{12}-\tilde p_{12}|^2}}.    
\end{align*}
For inner-product kernels $K$ and $\tilde K$,
we have, with 
$\mu$ being the law of $\Iprod{x_1}{x_2}$:
\begin{equation}
\TV(P,\tilde P) 
    \leq
    \binom{n}{2} \|K-\tilde K\|_{L_2(\mu)}
    \label{eq:kernel-approx}
\end{equation}
Let $Q$ and $\tilde Q$ denote the corresponding \ER graphs, namely, 
$G(n,p)$ and $G(n,\tilde p)$, 
where $p = \Expect[K(\Iprod{x_1}{x_2})]$
and $\tilde p = \Expect[\tilde K(\Iprod{x_1}{x_2})]$.
Then 
$\TV(Q,\tilde Q) \leq \binom{n}{2} 
|p-\tilde p| \leq \binom{n}{2} \|K-\tilde K\|_{L_2(\mu)}$.
By the triangle inequality, we have
\begin{equation}
    |\TV(P,Q) - \TV(\tilde P,\tilde Q)|
\leq 2 \binom{n}{2} \|K-\tilde K\|_{L_2(\mu)}.
\label{eq:kernel-approx-twosided}
\end{equation}
As long as $\|K-\tilde K\|_{L_2(\mu)} = o(n^{-2})$, we may replace the kernel $K$ by $\tilde K$.
The next lemma is useful for finding such a polynomial kernel $\tilde K$.

\begin{lemma} \label{lem:polynomial-approximation}
Let $K: [-1,1]\to [0,1]$ be a fixed $C^\infty$ function such that $c \le K(t) \le 1-c$ for all $t \in [-1,1]$ and a fixed constant $c>0$. For any positive integer $\ell$, there exists a polynomial $\tilde K$ of degree $L = L_{K,\ell}$ such that (i) $c/2 \le \tilde K(t) \le 1-c/2$ for all $t \in [-1,1]$, and (ii) $|\tilde K(t) - K(t)| \le \frac{c}{2} |t|^\ell$. 
\end{lemma}

\begin{proof}
Let $K_\ell$ be the degree-$\ell$ Taylor approximation of $K$ at $0$. 
By Taylor's theorem, there is a continuous function $S:[-1,1]\to\mathbb R$ such that $K(t) = K_\ell(t) + t^\ell S(t)$ for $t \in [-1,1]$. 
The Weierstrass approximation theorem
provides a polynomial $P$ of degree $L'=L'_{K,\ell}$ such that $|S(t) - P(t)| \le c/2$ for $t \in [-1,1]$. 
Define $\tilde K(t) = K_\ell(t)+ t^\ell P(t)$.
It is straightforward to verify the two claims.
\end{proof}

Recall that $K: [-1,1]\to [0,1]$ is  a $C^\infty$ function such that $c \leq  K(t) \leq 1-c$ for all $t\in[-1,1]$ and  some constant $c$. 
Thus
$p = \Expect[K(\Iprod{x_1}{x_2})] \in [c,1-c]$.
Applying \prettyref{lem:polynomial-approximation}, there exists a constant $L$ and a degree-$L$ polynomial $\tilde K$ such that
(a) $\min_{t\in[-1,1]} \tilde K(t) \geq c/2$ and $\max_{t\in[-1,1]} \tilde K(t) \leq 1-c/2$;
(b) $|K(t) - \tilde K(t)| \leq \frac{c}{2}|t|^6$. 
Then $\tilde p \equiv \Expect[\tilde K(\Iprod{x_1}{x_2})] $ satisfies $\tilde p \in [c/2,1-c/2]$ and
\[
|p-\tilde p| \leq 
\|K-\tilde K\|_{L_2(\mu)}
\lesssim \sqrt{\Expect[|\Iprod{x_1}{x_2}|^{12}]} \lesssim d^{-3}.
\]
Applying the triangle inequality and \prettyref{eq:kernel-approx},
\[
\TV(P,Q)
\leq 
\TV(P,\tilde P)+\TV(Q,\tilde Q)+\TV(\tilde P,\tilde Q)
\leq 2 \binom{n}{2} \|K-\tilde K\|_{L_2(\mu)} + \TV(\tilde P,\tilde Q).
\]
As the first term vanishes due to the assumption $d \gg n^{3/4}$, it remains to bound 
$\TV(\tilde P,\tilde Q)$. 
This is an RGG with a degree-$L$ polynomial kernel $\tilde K$.
The standardized version is 
\[
\tilde \kappa(t)
\eqdef
\frac{\tilde K(t)-\tilde p}{\sqrt{\tilde p(1-\tilde p)}} 
= \frac{\sum_{\ell=1}^L a_\ell t^\ell -\tilde p}{\sqrt{\tilde p(1-\tilde p)}} 
\]
where $c/2 \leq \tilde p \leq 1-c/2$. 
Since $L$ is a constant, the assumption 
\prettyref{eq:coeffs-assumption} is automatically satisfied for suitably large constants $B$ and $r$.
Applying \prettyref{thm:poly} concludes that $\TV(\tilde P,\tilde Q) = o(1)$ whenever $d \gg n^{3/4}$, completing the proof.

\subsection{Proofs of Theorems \ref{thm:lowsnr}--\ref{thm:highsnr}}

\begin{proof}[Proof of \prettyref{thm:lowsnr}]
We follow the same argument as in the proof of \prettyref{thm:fixed}, with the following modifications.

For the lower bound, since by assumption
 $d \gg n^{3/4} r^{3/2}$ and $r \gg n^{-1/2+\epsilon}$ for some constant $\epsilon$, we have $d \geq n^{3\epsilon/2}$.
 Again applying \prettyref{lem:polynomial-approximation}, there exists a constant $L$ and a degree-$L$ polynomial $\tilde K$ such that
(a) $\min_{t\in[-1,1]} \tilde K(t) \geq c/2$ and $\max_{t\in[-1,1]} \tilde K(t) \leq 1-c/2$;
(b) $|K(t) - \tilde K(t)| \leq \frac{c}{2}|t|^{10/\epsilon}$.
This ensures that $
|p-\tilde p| \leq 
\|K-\tilde K\|_{L_2(\mu)}
\lesssim d^{-3}$
is still satisfied. The lower bound follows from the same kernel polynomial approximation argument and 
\prettyref{thm:poly}.

For the upper bound, we still use the signed triangle count $T$ as the test statistic. In applying \prettyref{thm:ub} and 
bounding the trace, the first two parts in 
\prettyref{eq:gegenbauer-coefficient-bound} become now $
\alpha_1 = \frac{r \kappa'(0)+o_{d}(1)}{d}$ and $|\alpha_k| \leq C (r/d)^{k}$ for $ k=2,3,4$. The third inequality in 
\prettyref{eq:gegenbauer-coefficient-bound} continues to hold, and we note that 
$\|\kappa\|_{L_2(\mu)}^2 
= \frac{1}{p(1-p)}
\Var(K(r\Iprod{x_1}{x_2}))$,
where 
$c \leq p \leq 1-c$ and the variance equals
$
= \frac{1}{2} \Expect[(K(r\Iprod{x_1}{x_2})-K(r\Iprod{\tilde x_1}{\tilde x_2}))^2] \leq C r^2/d$,
where $C$ is a constant and $\Iprod{\tilde x_1}{\tilde x_2}$ 
is an i.i.d.\  copy of $\Iprod{ x_1}{x_2}$.
Thus, $\|\kappa\|_{L_2(\mu)} = O(r/\sqrt{d})$.
The same argument yields 
$\tr(\kappa^3) = \Theta(r^3 d^{-2})$
and 
$\tr(\kappa^4) = O(r^4 d^{-3})$.
Thus 
$\Var_P[T] = O(n^3 + n^4 r^4 d^{-3})$.
By assumption, $r \gg n^{-1/2}$. So the desired condition \prettyref{eq:mean-sep} holds provided that $d \ll n^{3/4} r^{3/2}$, implying strong detection.
\end{proof}



\begin{proof}[Proof of \prettyref{thm:linear}]
Recall that $K(t) = p + r t$, where $0<r<p<1/2$.
For the lower bound, suppose $d \gg (\frac{nr^2}{p})^{3/4}$, which implies, in view of the assumption $r\gtrsim \sqrt{\frac{p}{n}} \log n$, that 
$d \gtrsim (\log n)^{3/2}$. The lower bound then follows directly from the second part of \prettyref{thm:poly} by applying 
$b_1 = \frac{r}{\sqrt{p(1-p)}}$.

For the upper bound, since $\kappa(t) = \frac{r}{\sqrt{p(1-p)}} t$ and 
$C_1^{\frac{d-2}{2}}(t) = (d-2)t$. 
The only non-zero eigenvalue is $\alpha_1 = \frac{r}{(d-2)\sqrt{p(1-p)}} $, with multiplicity $d$.
Thus $\tr(\kappa^3) = d \alpha_1^3 \asymp \frac{r^3}{d^2 p^{3/2}}$
and 
$\tr(\kappa^4) = d \alpha_1^4 \asymp \frac{r^4}{d^3 p^2}$.
Applying \prettyref{thm:ub}, the signed triangle count $T$ satisfies
$|\Expect_P[T]-\Expect_Q[T]| \asymp \frac{n^3 r^3}{d^2 p^{3/2}}$,
$\var_Q(T)\asymp n^3$, and 
$\var_P(T)\lesssim \frac{n^3\tr(\kappa^3)}{p^{3/2}} + \frac{n^4\tr(\kappa^4)}{p} 
\asymp  \frac{n^3 r^3}{p^3 d^2} + \frac{n^4r^4}{p^3d^3} $.
To ensure
$|\Expect_P[T]-\Expect_Q[T]| \gg 
\var_Q(T)+\var_P(T)$, we require
$d \ll (\frac{nr^2}{p})^{3/4}$,
$d \ll (nr)^{3/2}$, and
$d \ll (nr)^2$.
The first (and desired) condition dominates the second and third because $p \ge r \gtrsim \sqrt{\frac{p}{n}} \log n$ by definition.
\end{proof}

\begin{proof}[Proof of \prettyref{thm:highsnr}]
Recall that the kernel $K(t) = \int_{-\infty}^t f(x)dx$ is the CDF for the density $f$, which is assumed to satisfy the lower bound 
$f(x) \geq c \exp(-Cx^2)$.
The scaled kernel
$K_r(t) = K(rt) $ defines the RGG, leading to average edge density $p =\Expect[K(r \Iprod{x_1}{x_2})]$.
By assumption, $1\ll r \ll \sqrt{d}$. Since $\Iprod{x_1}{x_2} = O_P(\frac{1}{\sqrt{d}})$, we have 
$p = K(0)+o(1)$ which is a constant in $(0,1)$.
The normalized kernel
$\kappa_r(t) = \frac{K(rt)-p}{\sqrt{p(1-p)}}$ satisfies
$\kappa_r'(0) = r K'(0) = rf(0)$ and $f(0) \geq c$.
From here, the upper bound calculation proceeds identically to that in the proof of \prettyref{thm:lowsnr}.

Next we prove the lower bound assuming the dimension exceeds the threshold with $d \gg n^{3/4} r^{3/2}$. The argument is similar to that of \prettyref{thm:fixed}, except that the polynomial approximation part is more involved and  requires choosing a degree that grows as $L\asymp r^2$.
The following lemma (see \prettyref{app:poly-highsnr} for a proof) ensures that the degree-$L$ Taylor expansion yields a polynomial kernel that
(a) is a valid probability kernel, taking values in $(0,1)$;
and 
(b) satisfies the coefficient decay condition required by \prettyref{thm:poly}.
\begin{lemma}
Suppose the kernel $K(t)$ satisfies the conditions in~\prettyref{thm:highsnr}. 
There exists a degree-$L$ polynomial $\tilde K(t) = \sum_{\ell\geq 0}^L a_\ell t^\ell$ 
that satisfies
(a) $\tilde K(t) \in (0,1)$ for all $t \in (-r,r)$;
(b) 
$\Expect[|K(r\Iprod{x_1}{x_2}) - \tilde K(r\Iprod{x_1}{x_2})|^2]
\leq (C_0r^2/d)^L$;
(c) $|a_\ell| \leq \frac{C_0^\ell}{\sqrt{\ell!}}$;
(d) $L = \lceil C_1 r^2 \rceil$,
where $C_0,C_1$ are positive constants.
\label{lmm:poly-highsnr}
\end{lemma}
With this lemma, the rest proof is identical to that of \prettyref{thm:fixed}.
Let $\tilde K_r(t) = \tilde K(rt)$ which is a valid probability kernel on $[-1,1]$ thanks to part (a). Using part (b), 
$\|K_r-\tilde K_r\|_{L_2(\mu)}=
(\Expect[\|K(r\Iprod{x_1}{x_2}) - \tilde K(r\Iprod{x_1}{x_2})|^2])^{1/2} 
\leq (C_0r^2/d)^{L/2} \ll n^{-2}$, 
the last step applying the assumption that $r \leq d^{1/12-\epsilon}$ and $d \geq n^{3/4}$.
Therefore, by 
\prettyref{eq:kernel-approx-twosided}, we can replace the original kernel $K_r$ by $\tilde K_r$.

Next, let $\tilde p = \Expect[\tilde K(r\Iprod{x_1}{x_2})]$.
Then $\tilde p = p + o(1)$ is also a constant. Taylor expanding the normalized kernel $\tilde \kappa(t) \equiv \frac{\tilde K(rt) - \tilde p}{\sqrt{\tilde p (1-\tilde p)}}$ as $\tilde \kappa(t) = 
\sum_{\ell=0}^L b_\ell t^\ell$, we have
$b_\ell = 
\frac{1}{\sqrt{\tilde p (1-\tilde p)}} a_\ell r^\ell$ for $\ell \geq 1$.
Thanks to \prettyref{lmm:poly-highsnr} part (c), the key condition \prettyref{eq:coeffs-assumption} on the coefficients $b_1,\ldots,b_L$ is satisfied with $r$ replaced by $C_0 r$.
(By \prettyref{eq:b0-auto}, the condition on $b_0$ is automatically satisfied because $r\ll \sqrt{d}$.)
By the assumption $r \le d^{1/12 - \epsilon}$, the degree $L = \lceil C_1 r^2 \rceil \leq d^{1/6 - 2\epsilon}$.
Furthermore, $d\gg n^{3/4}$ by assumption.
The proof is then completed by applying \prettyref{thm:poly}.    
\end{proof}

\section{Proofs for recovery}

\subsection{Proof of Theorem~\ref{thm:fixeded-kernel-spectral-recovery}: Lower bound}
Let $\tilde K$ be the degree-$L$ polynomial given by \prettyref{lem:polynomial-approximation}. 
With $\mu$ denoting the law of $\langle x_1, x_2 \rangle$, we have 
$$
\|K - \tilde K\|_{L_2(\mu)} 
\le \frac{c}{2} \sqrt{\Expect [\langle x_1, x_2 \rangle^{2\ell}]}
\le C_\ell d^{-\ell/2} .
$$
Let $P_{X,A}$ (resp.\ $\tilde P_{X,A}$) denote the joint law of $(X,A)$ with kernel $K$ (resp.\ $\tilde K$). 
Note that the proof of \eqref{eq:kernel-approx} is in fact valid for the joint laws, giving
$$
\TV(P_{X,A}, \tilde P_{X,A}) \le C_\ell n^2 d^{-\ell/2} .
$$
Consider the MMSE
$$
\mmse(P) \eqdef \Expect_{P_{X,A}} [\psi(X,A)], \quad \text{ where } \psi(X,A) \eqdef \|\Expect[X \mid A] - X\|_F^2 .
$$
Since $X$ is entrywise bounded in $[-1,1]$, we have $\psi(X,A) \le 4 n^2$. It follows that
$$
|\mmse(P) - \mmse(\tilde P)| \le 8 n^2 \TV(P_{X,A}, \tilde P_{X,A}) \le 8 C_\ell n^4 d^{-\ell/2} \le n^{-10}
$$
if $d \ge n^{\epsilon}$ for a constant $\epsilon > 0$ and $\ell$ is taken to be a sufficiently large constant depending on $\epsilon$. 
Therefore, we may assume without loss of generality that $K$ is a degree-$L$ polynomial kernel.

By Corollary~\ref{cor:posterior-moments}, if $d \gg \sqrt{n}$, then the posterior sample $\tilde{X}$ satisfies
$$
\expect{\iprod{\tilde{X}}{X}}
\le C \bigg( \frac{n^2}{d^2} + \frac{n (\log n)^{1/2}}{d} \bigg) .
$$
Since the MSE achieved by the posterior sample is twice the MMSE, it follows that
\begin{align*}
\mmse(P)
&= \frac 12 \Expect \left[ \big\| \tilde{X} - X \big\|_F^2 \right] \\
&= \frac 12 \left( \Expect \left[ \big\| \tilde{X} \big\|_F^2 \right] + \Expect \left[ \big\| X \big\|_F^2 \right] - 2 \expect{\iprod{\tilde{X}}{X}} \right) \\
&\ge \Expect \left[ \big\| X \big\|_F^2 \right] - C \bigg( \frac{n^2}{d^2} + \frac{n (\log n)^{1/2}}{d} \bigg) \\
&= \left( 1 - O\left( \frac{1}{d} + \frac{\sqrt{\log n}}{n} \right) \right) \Expect \left[ \big\| X \big\|_F^2 \right] ,
\end{align*}
where we used that $\Expect \left[ \big\| X \big\|_F^2 \right] = \frac{n(n-1)}{d}$.

\subsection{Proof of Theorem~\ref{thm:fixeded-kernel-spectral-recovery}: Upper bound}

For the upper bound in Theorem~\ref{thm:fixeded-kernel-spectral-recovery}, 
we consider the following spectral method. 
For $\bar A$ defined in \eqref{eq:def-A-bar}, 
let $U \in \reals^{n \times d}$ be the matrix whose columns are the top-$d$ eigenvectors of $\bar A$ (corresponding to the largest $d$ eigenvalues in absolute values). 
We will show that $U U^\top$ is close to $\frac dn X$ via spectral perturbation analysis. 

Recall that $K: [-1,1]\to [0,1]$ is a fixed $C^\infty$ function that is bounded away from 0 and 1 and satisfies $K'(0) \neq 0$, $p \eqdef \Expect K(\langle x_1, x_2 \rangle)$, and $\kappa(t) \eqdef \frac{K(t)-p}{\sqrt{p(1-p)}}$ for $t \in [-1,1]$. 
Let $X \in \reals^{n \times n}$ be defined by $X_{ij} = \langle x_i, x_j \rangle$ for all $i,j \in [n]$. 
Note that we slightly abuse the notation to define $X_{ii} = \langle x_i, x_i \rangle = 1$ instead of $0$ for convenience. 
This is clearly inessential because we can modify the diagonal of the estimator accordingly, and for $d \ll n$, we have $\Expect \|X\|_F^2 \approx n^2/d$ regardless of the choice of the diagonal. 
Let us first establish a few preliminary results.

\begin{lemma}
\label{lem:recovery-upper-kernel-approximation}
There exist constants $L$ and $C$ depending only on the kernel $K$ and a polynomial $\tilde \kappa$ of degree $L$ such that the following statements hold. 
First, 
$$
\Expect \left[ \|\tilde \kappa(X) - \kappa(X)\|_F^2 \right]
\le C (n^2/d^{10} + n) .
$$
Moreover, let the Gegenbauer polynomial expansion of $\tilde \kappa$ be 
$$
\tilde \kappa(t) = \sum_{k=0}^L \tilde \alpha_k C^\lambda_k(t)
$$
where $\lambda \eqdef (d-2)/2$. 
Then we have 
\begin{equation*}
|\tilde \alpha_0| \le C d^{-5}, \quad 
\tilde \alpha_1 = \frac{\kappa'(0)+o_{d}(1)}{d}, 
\quad
|\tilde \alpha_k| \leq C d^{-k} \text{ for } k=2,\dots,L.  
\end{equation*}
\end{lemma}

\begin{proof}
By the proof of Lemma~\ref{lem:polynomial-approximation}, there are constants $L$ and $C_0$ depending only on the kernel $K$ and a polynomial $\tilde \kappa$ of degree $L$ such that 
$$
|\tilde \kappa(t) - \kappa(t)| \le C_0 t^{10} 
$$
for all $t \in [-1,1]$. 
As a result, 
$$
\|\tilde \kappa(X) - \kappa(X)\|_F^2 = \sum_{i,j=1}^n (\tilde \kappa(\langle x_i, x_j \rangle) - \kappa(\langle x_i, x_j \rangle))^2 \le C_0^2 \sum_{i,j=1}^n \langle x_i, x_j \rangle^{20} .
$$
It follows that 
$$
\Expect \|\tilde \kappa(X) - \kappa(X)\|_F^2 
\le C_0^2 (n^2 \Expect \langle x_1, x_2 \rangle^{20} + n) 
\le C (n^2/d^{10} + n) .
$$
Moreover, 
$$
|\tilde \alpha_0| 
= |\tilde \alpha_0 - 0| 
= |\Expect \tilde \kappa(\langle x_1, x_2 \rangle) - \Expect \kappa(\langle x_1, x_2 \rangle)|
\le C_0 \Expect \langle x_1, x_2 \rangle^{10}
\le C/d^5. 
$$
For the approximation in Lemma~\ref{lem:polynomial-approximation}, we have $\tilde \kappa'(t) = \kappa'(t)$, and $\tilde \kappa(t)$ is a uniformly bounded polynomial. 
Therefore, the estimates for $\tilde \alpha_k$ follow from \prettyref{lmm:rodrigues} together with the fact that $\Iprod{x_1}{x_2}$ is $O(\frac{1}{d})$-sub-Gaussian.
\end{proof}

\begin{lemma}
\label{lem:recovery-QR-decomposition}
If $d \le n$, there is a matrix $Q \in \reals^{n \times d}$ with orthonormal columns and an absolute constant $C>0$ such that 
$$
\Expect \left[ \left\|\frac nd Q Q^\top - X\right\|^2 \right] \le C (n/d + 1) .
$$
\end{lemma}

\begin{proof}
Let $\Phi \eqdef [x_1 \cdots x_n]^\top \in \reals^{n \times d}$, and let $\phi_i$ denote the $i$th column of $\Phi$. 
Consider the QR decomposition $\Phi = Q R$ for a matrix $Q \in \reals^{n \times d}$ with orthonormal columns and an upper triangular matrix $R \in \reals^{d \times d}$. 
We have $X = \Phi \Phi^\top = Q R R^\top Q^\top$, so 
$$
\left\|\frac nd Q Q^\top - X\right\|
= \left\|Q \left(\frac nd I_d - R R^\top\right) Q^\top\right\|
\le \left\|R R^\top - \frac nd I_d\right\| 
= \left\|R^\top R - \frac nd I_d\right\| 
= \left\|\Phi^\top \Phi - \frac nd I_d\right\| .
$$
Note that $\Phi^\top \Phi = \sum_{i=1}^n x_i x_i^\top$. 
Since $\Expect[x_i x_i^\top] = \frac 1d I_d$ and $x_i$ is $O(\frac 1d)$-sub-Gaussian, by applying \cite[Remark~4.7.3]{vershynin2018high}, we see that with probability at least $1-\delta$, 
\begin{equation*}
\left\|\frac nd Q Q^\top - X\right\|
= \left\| \sum_{i=1}^n x_i x_i^\top - \frac nd I_d \right\|
\le \frac{Cn}{d} \left( \sqrt{\frac{d + \log(1/\delta)}{n}} + \frac{d + \log(1/\delta)}{n} \right) .
\end{equation*}
Integrating the tail yields the claimed bound.
\end{proof}

\begin{lemma} \label{lem:gegenbauer-term-spectral-bound}
There is an absolute constant $C>0$ such that for any $k \ge 1$, we have
$$
\Expect \left[ \|C^\lambda_k(X)\|^2 \right] \le C \Big( n^2 + d^{2k} \log^2(d^k) \Big) .
$$
\end{lemma}

\begin{proof}
Let $D_k$ denote the dimension of $\calH^d_n$, the space of the $k$th spherical harmonics. 
By the addition formula for spherical harmonics \cite[(1.2.8)]{dai2013approximation}, we have 
$$
\frac{\lambda + k}{\lambda} C^\lambda_k(\langle x_i, x_j \rangle) = \sum_{\ell=1}^{D_k} Y_\ell(x_i) Y_\ell(x_j) ,
$$
where $Y_1, \dots, Y_{D_k}$ form an orthonormal basis of $\calH^d_n$. 
Let $\tilde Y \in \reals^{n \times D_k}$ denote the matrix with the $(i,\ell)$th entry defined by $Y_\ell(x_i)$. Then
$$
\|C^\lambda_k(X)\| 
= \frac{\lambda}{\lambda + k} \| \tilde Y \tilde Y^\top \| 
= \frac{\lambda}{\lambda + k} \| \tilde Y^\top \tilde Y \| 
\le \| \tilde Y^\top \tilde Y \| . 
$$

Note that $\tilde Y^\top \tilde Y = \sum_{i=1}^n Y(x_i) Y(x_i)^\top$ where $Y(x_i) \in \reals^{D_k}$ denotes the vector with entries $Y_\ell(x_i)$. 
By the orthonormality of the basis $Y_1, \dots, Y_{D_k}$, we have
$$
\Expect[Y(x_i) Y(x_i)^\top] = I_{D_k} .
$$
Moreover, applying the addition formula again yields
$$
\|Y(x_i)\|_2^2 = \sum_{\ell=1}^{D_k} Y_\ell(x_i)^2 
= \frac{\lambda + k}{\lambda} C^\lambda_k(1) 
= \frac{(\lambda + k) \Gamma(2\lambda + k)}{\lambda \Gamma(2 \lambda) k!} 
= D_k 
$$
at any $x_i \in \reals^d$. 
Hence, $\|Y(x_i) Y(x_i)^\top - I_{D_k}\| \le D_k+1$ and
$$
\Expect[ (Y(x_i) Y(x_i)^\top - I_{D_k})^2 ] = (D_k - 1) I_{D_k} .
$$
By the matrix Bernstein inequality \cite[Theorem~1.6.2]{tropp2015introduction}, for any $\delta \in (0,1)$, with probability at least $1-\delta$, 
$$
\|\tilde Y^\top \tilde Y - n I_{D_k}\| 
= \left\| \sum_{i=1}^n \left( Y(x_i) Y(x_i)^\top - I_{D_k} \right) \right\| 
\le C \left( \sqrt{n D_k \log (D_k/\delta)} + D_k \log(D_k/\delta) \right) 
$$
for a constant $C>0$. 
Since $D_k = \binom{d-1+k}{d-1} -\binom{d-3+k}{d-1} \le d^k$, we have, with probability at least $1-\delta$,
$$
\|C^\lambda_k(X)\| \le n + C \Big( \sqrt{n d^k \log (d^k/\delta)} + d^k \log(d^k/\delta) \Big) .
$$
Finally, integrating the tail finishes the proof.
\end{proof}

\begin{lemma}
\label{lem:bernoulli-noise-spectral-bound}
There is a constant $C>0$ depending only on the kernel $K$ such that 
$$
\Expect \left[ \bigg\| \frac{A-K(X)}{\sqrt{p(1-p)}} \bigg\|^2 \right] \le C n .
$$

\end{lemma}

\begin{proof}
By assumption, $K$ is bounded away from $0$ and $1$, so $p = \Expect K(\langle x_1, x_2 \rangle)$ is a constant, and it suffices to bound $\Expect \left[ \| A-K(X) \|^2 \right]$. 
The adjacency matrix $A$ has a zero diagonal and $K(X)$ has diagonal entries equal to $K(1)$, so the diagonal of $A-K(X)$ is constant-sized and does not affect the result. 
The upper-triangular off-diagonal entries of $A-K(X)$ are bounded, independent conditional on $X$. Therefore, by \cite[Theorem~4.4.3]{vershynin2018high}, there is an absolute constant $C_0>0$ such that for any $\delta \in (0,1)$, 
$$
\|A-K(X)\| \le C_0 \sqrt{n + \log (1/\delta)} 
$$
with probability at least $1-\delta$. Integrating the tail completes the proof.
\end{proof}

We are ready to prove the upper bound in Theorem~\ref{thm:fixeded-kernel-spectral-recovery}. 
Let $\tilde \kappa$ be given by Lemma~\ref{lem:recovery-upper-kernel-approximation}. By \eqref{eq:def-A-bar}, we have
\begin{align*}
\bar A
&= \tilde \kappa(X) + \kappa(X) - \tilde \kappa(X) + \frac{A-K(X)}{\sqrt{p(1-p)}} \\
&= \tilde \alpha_1 C^\lambda_1(X) + \tilde \alpha_0 \mathbf{1} \mathbf{1}^\top + \sum_{k=2}^L \tilde \alpha_k C^\lambda_k(X) + \kappa(X) - \tilde \kappa(X) + \frac{A-K(X)}{\sqrt{p(1-p)}} .
\end{align*}
Let $Q$ be given by Lemma~\ref{lem:recovery-QR-decomposition}. Then 
$$
C^\lambda_1(X)
= 2 \lambda X 
= (d-2) X 
= \frac{n (d-2)}{d} Q Q^\top +  (d-2) \left( X - \frac nd Q Q^\top \right) .
$$
Therefore,
$$
\bar A = \tilde \alpha_1 \frac{n (d-2)}{d} Q Q^\top +  \tilde \alpha_1 (d-2) \left( X - \frac nd Q Q^\top \right) + \tilde \alpha_0 \mathbf{1} \mathbf{1}^\top + \sum_{k=2}^L \tilde \alpha_k C^\lambda_k(X) + \kappa(X) - \tilde \kappa(X) + \frac{A-K(X)}{\sqrt{p(1-p)}} .
$$

The matrix $\tilde \alpha_1 \frac{n (d-2)}{d} Q Q^\top$ has eigenvalue $\tilde \alpha_1 \frac{n (d-2)}{d}$ with multiplicity $d$ and remaining eigenvalues equal to $0$, so by the Davis--Kahan theorem \cite[Lemma~4.1.16]{vershynin2018high}, 
\begin{align*}
\|U U^\top - Q Q^\top\|
&\le \frac{2 d}{|\tilde \alpha_1| n (d-2)} \left\| \bar A - \tilde \alpha_1 \frac{n (d-2)}{d} Q Q^\top \right\| \\
&\le \frac{2 d}{|\tilde \alpha_1| n (d-2)} \bigg( |\tilde \alpha_1| (d-2) \left\| X - \frac nd Q Q^\top \right\| + n |\tilde \alpha_0| + \sum_{k=2}^L |\tilde \alpha_k| \cdot \|C^\lambda_k(X)\| \\
& \hspace{180pt} + \|\kappa(X) - \tilde \kappa(X)\| + \bigg\| \frac{A-K(X)}{\sqrt{p(1-p)}} \bigg\| \bigg) .
\end{align*}
It then follows from the triangle inequality that
\begin{align*}
\left\| U U^\top - \frac dn X \right\|
&\le \left\| U U^\top - Q Q^\top \right\| + \left\| Q Q^\top - \frac dn X \right\| \\
&\le \frac{3 d}{n} \left\| X - \frac nd Q Q^\top \right\| + \frac{2 d}{|\tilde \alpha_1| n (d-2)} \bigg( n |\tilde \alpha_0| + \sum_{k=2}^L |\tilde \alpha_k| \cdot \|C^\lambda_k(X)\| \\
& \hspace{180pt} + \|\kappa(X) - \tilde \kappa(X)\| + \bigg\| \frac{A-K(X)}{\sqrt{p(1-p)}} \bigg\| \bigg) .
\end{align*}
Note that $U U^\top - \frac dn X$ has rank at most $2d$, so 
\begin{align*}
\Big\| U U^\top - \frac dn X \Big\|_F^2
&\le 2d \cdot \Big\| U U^\top - \frac dn X \Big\|^2 \\
&\le \frac{C d^3}{n^2} \left\| X - \frac nd Q Q^\top \right\|^2 + \frac{C d}{\tilde \alpha_1^2 n^2} \bigg( n^2 \tilde \alpha_0^2 + \sum_{k=2}^L \tilde \alpha_k^2 \|C^\lambda_k(X)\|^2 \\
& \hspace{180pt} + \|\kappa(X) - \tilde \kappa(X)\|^2 + \bigg\| \frac{A-K(X)}{\sqrt{p(1-p)}} \bigg\|^2 \bigg) 
\end{align*}
for a constant $C>0$ depending only on the kernel $K$. 

It remains to take the expectation and apply Lemmas~\ref{lem:recovery-upper-kernel-approximation}, \ref{lem:recovery-QR-decomposition}, \ref{lem:gegenbauer-term-spectral-bound}, and~\ref{lem:bernoulli-noise-spectral-bound} to obtain 
\begin{align*}
\Expect \Big\| U U^\top - \frac dn X \Big\|_F^2
&\le \frac{C d^3}{n^2} \left( \frac nd + 1 \right) + \frac{C d^3}{\kappa'(0)^2 n^2} \bigg( \frac{n^2}{d^{10}} + \sum_{k=2}^L \frac{1}{d^{2k}} \Big( n^2 + d^{2k} (\log d)^2 \Big) + \frac{n^2}{d^{10}} + n \bigg) \\
&\le C \left( \frac{d^2}{n} + \frac{d^3}{n^2} + \frac{1}{d^7} + \frac{1}{d} + \frac{d^3 (\log d)^2}{n^2} + \frac{1}{d^7} + \frac{d^3}{n} \right) \\
&\le C \left( \frac{1}{d} + \frac{d^3 (\log d)^2}{n^2} + \frac{d^3}{n} \right) 
\end{align*}
where the constant $C$ varies between lines for ease of notation. 
Therefore, the MMSE satisfies
$$
\min_{\hat X} \Expect \|\hat X - X\|_F^2 = O \left( \frac{n^2}{d^3} + d (\log d)^2 + dn \right) 
= O \left( \frac{n^2}{d^3} + dn \right) 
= o(n^2/d) 
$$
if $1 \ll d \ll \sqrt{n}$.

\appendix

\section{Useful facts about spherical harmonics and Gegenbauer polynomials}
\label{app:Gegenbauer}

We follow the notation in the standard reference \cite{dai2013approximation}.
Let $\lambda \equiv \frac{d-2}{2}$ and let $C_k^{\lambda}(t)$ denote the degree-$k$ Gegenbauer polynomial ($C_0^\lambda(t)=1,C_1^\lambda(t)=2\lambda t$, etc.), which satisfies
\[
\Expect[C_k^{\lambda}(\Iprod{x}{y})C_m^{\lambda}(\Iprod{x}{y})]
= \frac{\lambda}{k+\lambda} C_k^{\lambda}(1) 
\indc{k=m}. 
\]
Let $\calH_k^d$ denote the space of spherical harmonics (namely, homogeneous harmonic polynomials) of degree $k$ on $S^{d-1}$.
Let $\{Y_i: 1 \le i \le \dim(\calH_k^d)\}$ be an orthonormal basis of $\calH_k^d$, which satisfies the so-called \textit{addition formula} of spherical harmonics \cite[Eq.~(1.2.8)]{dai2013approximation}:
\begin{align}
\sum_{j=1}^{\dim(\calH_k^d)} Y_j(x) Y_j(y) = \frac{k+\lambda}{\lambda} C_k^\lambda\left( \iprod{x}{y}\right). \label{eq:additive_formula}
\end{align}
In particular, it follows that 
\begin{equation}
C_k^{\lambda}(1) = 
\frac{\lambda}{k+\lambda} \dim(\calH_{k}^d)
= \binom{d-3+k}{k}.
\label{eq:Gegenbauer-norm}
\end{equation}
Thus
\begin{equation}
\dim(\calH_{k}^d) = 
\frac{d-2+2k}{d-2} 
\binom{d-3+k}{k} 
= \binom{d-1+k}{d-1} -\binom{d-3+k}{d-1}.
    \label{eq:dim-harmonics}
\end{equation}

Consider a function $\kappa$ expanded under the (orthogonal unnormalized) Gegenbauer basis:
\begin{align}
\kappa(t) = \sum_{k\geq 0} \alpha_k C_k^{\lambda}(t). \label{eq:kappa_expansion}
\end{align}
These coefficients determine the eigenvalues of the kernel operator 
\[
(\kappa f)(x)\equiv \Expect_{y\sim\Unif(S^{d-1})}[\kappa(\Iprod{x}{y}) f(y)].
\]
Indeed, combining \prettyref{eq:kappa_expansion} with  the addition formula \prettyref{eq:additive_formula}, we see that the eigenvalues of the $\kappa$ operator are precisely
\begin{equation}
\frac{\lambda
}{k+\lambda}\alpha_k, \text{ with multiplicity } \dim(\calH_{k}^d) ,
    \label{eq:eigenvalues-kappa}
\end{equation}
and the corresponding eigenfunctions are the orthonormal polynomial basis of $\calH_{k}^d$.

The following lemma relates the eigenvalues to the smoothness of the kernel. In particular, for fixed $k$ and large $d$, the $k$th eigenvalue of a smooth kernel behaves as $O(d^{-k})$ with multiplicity $\Theta(d^k)$. In contrast, for non-smooth kernels (such as a step function), the $k$th eigenvalue satisfies $O(d^{-k/2})$.
\begin{lemma}
\label{lmm:rodrigues}
For any $k\geq 1$,
\[
|\alpha_k| \leq \frac{d+2k-2}{d-2}
\frac{\sqrt{\expect{\kappa(\Iprod{x_1}{x_2})^2}}}{\sqrt{\dim(\calH_k^d)}}.
\]
Furthermore, assuming $\kappa$ is $k$-times differentiable. Then
\[
\alpha_k  = \frac{d+2k-2}{(d-2)(d-1)(d+1)\cdots (d+2k-3)} \expect{\kappa^{(k)}(\Iprod{x_1}{x_2})\pth{(1-\Iprod{x_1}{x_2}^2}^k}.
\]
\end{lemma}
\begin{proof}
Let $T \equiv \Iprod{x_1}{x_2}$, with  density $w(t) = \frac{1}{B(\frac{1}{2}, \lambda+\frac{1}{2})} (1-t^2)^{\lambda-\frac{1}{2}}$ and $B(a,b) = \frac{\Gamma(a)\Gamma(b)}{\Gamma(a+b)}$ is the beta function.
The Gegenbauer polynomials $C_k^{\lambda}(t)$ are orthogonal under this weight $w$.
Thus
$\alpha_k = \frac{\Expect[\kappa(T) C_k^\lambda(T)] }{\Expect[C_k^\lambda(T)^2] }$,
where $\Expect[C_k^\lambda(T)^2] = 
(\frac{d-2}{d+2k-2})^2
\dim(\calH_k^d)
= 
\frac{d-2}{d+2k-2}
\binom{d-3+k}{k}$ as given by \prettyref{eq:Gegenbauer-norm}.
The first upper bound follows from Cauchy-Schwarz.
(This is the same as using the trace formula $\Expect[\kappa(T)^2]=\tr(\kappa^2) = \sum_k
(\frac{\lambda
}{k+\lambda}\alpha_k)^2 \dim(\calH_{k}^d)$.)

For the second identity, recall the Rodrigues formula \cite[22.11.2]{AS}
\[
C_k^{\lambda}(t) = 
 \frac{\tau_k}{w(t)} \frac{d^k}{dt^k}[(1-t^2)^k w(t)], \quad
 \tau_k \triangleq (-1)^k \frac{1}{2^k k!} \frac{\Gamma(\lambda+\frac{1}{2})  \Gamma(k+2\lambda)}{\Gamma(\lambda+k+\frac{1}{2})  \Gamma(2\lambda)} 
\]
Applying this formula,
\begin{align*}
    \Expect[\kappa(T) C_k^\lambda(T)] 
= & \int_{-1}^1 w(t) \kappa(t) C_k^\lambda(t) dt \\
= & \tau_k
\int_{-1}^1 \kappa(t) \frac{d^k}{dt^k}[(1-t^2)^k w(t)] dt \\
= & \tau_k (-1)^k 
\int_{-1}^1 \kappa^{(k)}(t) (1-t^2)^k w(t) dt = 
\tau_k (-1)^k \Expect[\kappa^{(k)}(T) (1-T^2)^k],
\end{align*}
where the penultimate step applies integration by parts and the fact that all derivatives of $(1-t^2)^k w(t)$ up to order $k-1$ vanishes at the boundary $\pm 1$.
The final result follows after some simplification.
\end{proof}

\section{Heuristics on the spectral conjecture}
\label{app:heuristic}
Recall that the conjectured critical threshold for detection in \prettyref{conj:trace}
is given by $n^3\tr^2(\kappa^3) \asymp 1$.



\paragraph{Positive direction}
The conjectured threshold arises from a heuristic analysis of the signed triangle count. By~\prettyref{thm:ub}, the expected signed triangle count is zero  under the \ER model and of order $n^3\tr(\kappa^3)$ under the RGG model. Suppose further that the variance of the signed triangle count under the RGG model is $O(n^3)$, the same as under the \ER model. Then the mean difference dominates the standard deviation provided that $n^{3} \tr^2(\kappa^3) \gg 1$.


\paragraph{Negative direction}
The same threshold also emerges in a heuristic calculation of the KL divergence. Recall that the KL expansion in \prettyref{lem:kl-expansion} yields the generic upper bound: For any kernel,
$\KL(P_A\|Q_A)\le \sum_{k=2}^{n-1}\binom{n}{k+1} g(k)$. If this sum is dominated by the leading term of $k=2$, then the KL divergence vanishes whenever $n^3 g(2) =o(1)$.

Recall from \prettyref{eq:def-g(k)} that 
$$
g(2) = \expect{\left( \expect{\kappa(x_{n+1},x_1) \kappa(x_{n+1},x_2) \mid A } \right)^2}=\expect{\left( \expect{\eta(\iprod{x_1}{x_2})
 \mid A }\right)^2},
$$
where $\eta(\iprod{y}{z})\triangleq\Expect_x[\kappa(\Iprod{x}{y})\kappa(\Iprod{x}{z})]$.
As outlined in \prettyref{sec:outline_proof}, the key of the proof is to show that 
the information contained in the entire graph $A$ beyond the single edge $A_{12}$ is negligible for estimating $\langle x_1,x_2\rangle$; in other words, the variance reduction thanks to $A$ is on par with that thanks to $A_{12}$, i.e., 
$g(2) \asymp 
\Expect[\left( \expect{\eta(\iprod{x_1}{x_2})| A_{12} }\right)^2]$.
 As we show next, this latter quantity is precisely $\tr^2(\kappa^3)$.
Substituting $g(2)=O(\tr^2(\kappa^3))$ into $n^3 g(2) =o(1)$ yields the conjectured impossibility condition $n^3 \tr^2(\kappa^3) \ll 1$. 

To compute $\Expect[\left( \expect{\eta(\iprod{x_1}{x_2})| A_{12} }\right)^2]$, 
note that 
\begin{align*}
\expect{\eta(\iprod{x_1}{x_2}) \mid A_{12} =1
} 
& = \frac{ \expect{\eta(\iprod{x_1}{x_2}) \indc{A_{12} =1}
} }{\prob{A_{12}=1}}  \\
& = \frac{1}{p} 
\expect{\eta(\iprod{x_1}{x_2}) \left( 
\sqrt{p(1-p)}\kappa (\iprod{x_1}{x_2} ) +p \right)
} \\
& = \sqrt{\frac{1-p}{p}} \expect{\eta(\iprod{x_1}{x_2}) \kappa (\iprod{x_1}{x_2} )}
= \sqrt{\frac{1-p}{p}}  \tr(\kappa^3).
\end{align*}
Similarly, we can get  
$\expect{\eta(\iprod{x_1}{x_2}) \mid A_{12} =0
} = - \sqrt{\frac{p}{1-p}}  \tr(\kappa^3) $ and therefore 
$$
\expect{\left( \expect{\eta(\iprod{x_1}{x_2}) \mid A_{12}
} \right)^2} = p \times \frac{1-p}{p} 
{\mathop{\rm tr}}^2(\kappa^3) + (1-p) \times \frac{p}{1-p} {\mathop{\rm tr}}^2(\kappa^3) = {\mathop{\rm tr}}^2(\kappa^3).
$$

\section{Proof of \prettyref{eq:dn3} and 
\prettyref{thm:dn3-universal}
for arbitrary kernel}
\label{app:dn3}

In this appendix  
we prove the KL bound \prettyref{eq:dn3}
which implies non-detection at $d \gg n^3$ for any kernel with spherical points. Then, the proof is adapted to show universality for latent points with iid coordinates (\prettyref{thm:dn3-universal}).

\begin{proof}[Proof of \prettyref{eq:dn3}]    
We continue from \prettyref{eq:KLexp3}. 
Note that $x_{n+1} \sim \Unif(S^{d-1})$ is independent of $x_1,\ldots,x_n$ and $A$. 
Denote
\[
f(x_1,\ldots,x_k)\eqdef
\Expect_{x_{n+1}}
\qth{\prod_{i=1}^k
 \kappa(\Iprod{x_i}{x_{n+1}})
\Bigg| x_1,\ldots,x_k}.
\]
Then \prettyref{eq:KLexp3} is the same as 
\begin{align}
\KL(P_A\|Q_A) 
\leq  & \sum_{k=2}^{n-1}
\binom{n}{k+1} 
\Expect_{A}
\qth{
\left(\Expect_{x_1,\ldots,x_k|A}
[f(x_1,\ldots,x_k)] \right)^2} 
\nonumber \\
\leq & \sum_{k=2}^{n-1}
\binom{n}{k+1} 
\Expect_{x_1,\ldots,x_k}
[f(x_1,\ldots,x_k)^2] 
    \label{eq:KLexp32}
\end{align}
where the second line applies Jensen's inequality.
Write 
\[
f(x_1,\ldots,x_k)^2
=\Expect_{x_{n+1}\indep\tilde x_{n+1}}
\qth{\prod_{i=1}^k
 \kappa(\Iprod{x_i}{x_{n+1}})
 \kappa(\Iprod{x_i}{\tilde x_{n+1}})
\Bigg| x_1,\ldots,x_k}.
\]
We have
\begin{align*}
\Expect_{x_1,\ldots,x_k}
[f(x_1,\ldots,x_k)^2] 
=  &  
\Expect_{x_1,\ldots,x_k,x_{n+1},\tilde x_{n+1}}
\qth{\prod_{i=1}^k
 \kappa(\Iprod{x_i}{x_{n+1}})
 \kappa(\Iprod{x_i}{\tilde x_{n+1}})
} \nonumber \\
=  &  
\Expect
\qth{\eta(x_{n+1},\tilde x_{n+1})^k},
    \label{eq:KLexp33}
\end{align*}
where 
\[
\eta(y,z) = \eta(\Iprod{y}{z}) 
\triangleq \Expect_{x\sim \Unif(S^{d-1})}[\kappa(\Iprod{x}{y})
 \kappa(\Iprod{x}{z})].
\]
By \prettyref{lmm:maxcorr-sphere} below,  we have the deterministic bound
\begin{equation}
|\eta(\rho)| \leq C\pth{|\rho| + \frac{1}{d}}
\label{eq:claim-gyz}
\end{equation}
which holds for some constant $C>0$ depending only on $p$.
Thus 
$$
\Expect
\qth{\eta(x_{n+1},\tilde x_{n+1})^k} 
\le \Expect
\qth{C^k (|\Iprod{x_{n+1}}{\tilde x_{n+1}}| + 1/d)^k}
\le \bigg( \frac{C' \sqrt{k}}{\sqrt{d}} \bigg)^k 
$$
using the fact that $\Iprod{x_{n+1}}{\tilde x_{n+1}}$ is $O(1/d)$-sub-Gaussian. 
Moreover, $\Expect
\qth{\eta(x_{n+1},\tilde x_{n+1})} = 0$. 
It follows that 
\begin{align}
\KL(P_A\|Q_A) 
&\leq 
 \sum_{k=2}^{n-1}
\binom{n}{k+1} 
\Expect
[\eta(x_{n+1},\tilde x_{n+1})^k] \label{eq:KLdn3}\\
& \le \sum_{k=2}^{n-1}
\binom{n}{k+1} \bigg( \frac{C' \sqrt{k}}{\sqrt{d}} \bigg)^k 
\le \sum_{k=2}^{n-1} \frac{(C'' n)^{k+1}}{d^{k/2}} = O\bigg( \frac{n^3}{d} \bigg) \notag
\end{align}
provided that $d = \Omega(n^3)$. 
\end{proof}


\begin{lemma}
\label{lmm:maxcorr-sphere}
Let $d\geq 6$. Fix $y,z\in S^{d-1}$ and let $x \sim \Unif(S^{d-1})$.
There is a universal constant $C$ such that for any bounded $f$,
\[
\Cov(f(\Iprod{x}{y}), f(\Iprod{x}{z}))
\leq  C \|f\|_\infty^2 \pth{|\Iprod{y}{z}| + \frac{1}{d-5}}.
\]
\end{lemma}

\begin{proof}
Without loss of generality, let $y=e_1$ and $z =\rho e_1 + \sqrt{1-\rho^2} e_2$. Let $U=\Iprod{x}{y}=x_1,V=\Iprod{x}{z}=\rho x_1 + \sqrt{1-\rho^2} x_2$. 
Set 
$U'=x_1', V'=\rho x_1' + \sqrt{1-\rho^2} x_2'$, where $x_1',x_2'$ are i.i.d.\  $N(0,1/d)$. 
Then
\[
\TV\pth{\Law(U,V),\Law(U',V')} 
\leq 
\TV\pth{\Law(x_1,x_2),\Law(x_1',x_2')} 
\leq \frac{5}{d-5},
\]
the second step following from \cite{diaconis1987dozen}.

Replacing $f(U)$ by $f(U) - \Expect[f(U)]$, whose sup norm is at most $2\|f\|_\infty$, we may assume  $\Expect[f(U)]=\Expect[f(V)]=0$ so that $\Cov(f(U),f(V))=\Expect[f(U)f(V)]$.
Applying the preceding total variation bound, 
\[
|\Expect[f(U)f(V)]-\Expect[f(U')f(V')]|\leq \frac{10}{d-5} \|f\|_\infty^2.
\]
Furthermore, 
$|\Expect[f(U')| \leq \frac{10}{d-5} \|f\|_\infty $.
Finally, by maximal correlation of bivariate normal (see \cite[Theorem 33.12 and Example 33.7]{PW-it}), we have
\[
\Cov(f(U'),f(V')) \leq 
|\rho|\sqrt{\Var(f(U')) \Var(f(V'))} \leq |\rho| \|f\|_\infty^2.
\]
Combining these estimates completes the proof.
\end{proof}

Next we prove \prettyref{thm:dn3-universal} in \prettyref{sec:discussion}. 
The proof of this result is similar to the proceeding proof, with \prettyref{lmm:maxcorr-sphere} replaced by \prettyref{lmm:maxcorr-CLT}.
The key step in both lemmas is to approximate the two projections by a bivariate Gaussian and apply the Gaussian maximal correlation. For spherical, this Gaussian approximation was justified using Diaconis-Freedman. Here for iid we will do the same by a suitable application of CLT.


\begin{lemma}
    \label{lmm:maxcorr-CLT}
    Let $x=(x_1,\ldots,x_d)$ be a random vector in $\reals^d$ with i.i.d.\ components from a fixed distribution on $\reals$ with zero mean, unit variance, and $\Expect[|x_i|^3]<\infty$. Fix $y,z\in\reals^d$.
There is a constant $C>0$ such that for any CDF $f$ on $\reals$, we have 
\[
\Cov(f(\Iprod{x}{y}), f(\Iprod{x}{z}))
\leq  C \pth{|\rho| + (1-\rho^2)^{-3/2} 
\pth{\frac{\|y\|^2+\|z\|^2}{
\|y\|^2 \|z\|^2}}^{3/2} (\|y\|_3^3+\|z\|_3^3)
}
\]
where $\rho \equiv \frac{\Iprod{y}{z}}{\|y\|\|z\|}$.
\end{lemma}
\begin{proof}
Without loss of generality, assume that $f(u) = \mu((-\infty,u))$ for some probability measure $\mu$. Then
$\Expect[f(U)f(V)] = \iint 
\prob{U\geq s, V\geq t} \mu(ds)\mu(dt)$.
Thus
\begin{align*}
   |\Expect[f(U)f(V)]-\Expect[f(U')f(V')]|
&\leq \iint 
\left|\prob{U\geq s, V\geq t}-\prob{U'\geq s, V'\geq t}\right| \mu(ds)\mu(dt)\\
&\leq \sup_{A \text{ convex}} \left|\prob{(U, V)\in A}-\prob{(U', V')\in A}\right|  =: \dKS(P_{UV},P_{U'V'}) .
\end{align*}
Applying the same argument to the means, we conclude that 
\begin{align*}
   |\Cov(f(U),f(V))-\Cov(f(U'),f(V'))|
&\leq 3\dKS(P_{UV},P_{U'V'}).
\end{align*}

Next, we apply this to the joint distribution  of $(\Iprod{x}{y},\Iprod{x}{z})$ and its Gaussian approximation $\calN(0,\Sigma)$, 
where 
\[\Sigma = \begin{bmatrix}
    \|y\|^2 & \Iprod{y}{z}\\
    \Iprod{y}{z} & \|z\|^2\\
\end{bmatrix}.
\]
Identical to the proof of the preceding \prettyref{lmm:maxcorr-sphere}, the covariance under the bivariate Gaussian follows from the maximal correlation $\frac{|\Iprod{y}{z}|}{\|y\|\|z\|}$.
Thus, it remains to bound the KS distance.

Note that $(\begin{smallmatrix}
   \Iprod{x}{y}\\\Iprod{x}{z}
\end{smallmatrix}) = \sum_{i=1}^d X_i$, where $X_i = x_i (\begin{smallmatrix}
    y_i\\z_i
\end{smallmatrix})$ satisfies 
$\sum_{i=1}^d \Cov(X_i) = \Sigma$. 
Set $\tilde X_i = \Sigma^{-1/2} X_i$. We have
\[
\dKS(\Law(\Iprod{x}{y},\Iprod{x}{z}),\calN(0,\Sigma))
= 
\dKS\pth{\Law\pth{\sum_{i=1}^d \tilde X_i },\calN(0,I_2)}
\leq C \sum_{i=1}^d \Expect[\|\tilde X_i\|^3],
\]
for some universal constant $C>0$, by applying the multivariate Berry--Esseen theorem
\cite[Theorem 1.1]{raivc2019multivariate}.
Finally, we apply
$\Expect[\|\tilde X_i\|^3] \leq 
\expect{|x_i|^3} \lambda_{\min}(\Sigma)^{-3/2} (y_i^2+z_i^2)^{3/2}$,
and
\[
\lambda_{\min}(\Sigma) \geq \frac{\det(\Sigma)}{\tr(\Sigma)} = 
\frac{ \|y\|^2  \|z\|^2 - \Iprod{y}{z}^2}{\|y\|^2  + \|z\|^2}
\]
to conclude the proof.
\end{proof}

\begin{proof}[Proof of \prettyref{thm:dn3-universal}]
By \cite[Theorem 3.29]{folland1999real}, the function $K$, assumed to be right-continuous and have bounded variation, equals up to an additive constant the CDF of a finite signed measure on $\reals$. 
Since the goal is to compute the moments in \prettyref{eq:KLexp32}, we may assume without loss of generality that $K$ is the CDF of a probability measure.

The proof is identical to that of \prettyref{eq:dn3}
up to \prettyref{eq:claim-gyz}, which, by applying \prettyref{lmm:maxcorr-CLT}, is replaced by 
\begin{equation}
|\eta(y,z)| \leq C \pth{|\rho| + (1-\rho^2)^{-3/2} 
\pth{\frac{\|y\|^2+\|z\|^2}{
\|y\|^2 \|z\|^2}}^{3/2} \left(\|y\|_3^3+\|z\|_3^3\right)
}
\label{eq:claim-gyz1}
\end{equation}
where $\rho \equiv \rho(y,z) = \frac{\Iprod{y}{z}}{\|y\|\|z\|}$.

To bound the moment 
$\Expect
\qth{\eta(x_{n+1},\tilde x_{n+1})^k}$, define the good event
$$\calH = \{\|x_{n+1}\|^2 \geq c d, \|\tilde x_{n+1}\|^2 \geq c d, 
\|x_{n+1}\|_3^3 \leq C d, \|\tilde x_{n+1}\|_3^3 \leq C d, 
|\iprod{x_{n+1}}{\tilde x_{n+1}}| \leq c^2 d
\}$$
for an appropriately large constant $C>0$ and a small constant $c>0$. 
In particular, $|\rho(x_{n+1},\tilde x_{n+1})| \leq 0.1$  on the event $\calH$. 
Thus, using the boundedness of $\eta$,
\[
\eta(x_{n+1},\tilde x_{n+1})
\leq 
C\left(
|\rho(x_{n+1},\tilde x_{n+1})| + d^{-1/2}
+ 
\indc{\calH^c} \right)
\]
and hence 
\[
\Expect[|\eta(x_{n+1},\tilde x_{n+1})|^k]
\leq 
C_1^k \left(
\Expect[|\rho(x_{n+1},\tilde x_{n+1})|^k] + d^{-k/2}
+ 
\prob{\calH^c}\right)
\]

To bound $\prob{\calH^c}$, 
since $x_{n+1}$ and $\tilde x_{n+1}$ are independent with i.i.d.\ subgaussian coordinates. 
Using the fact that products of subgaussian random variables are 
subexponential, Bernstein's inequality yields
$\prob{\|x_{n+1}\|^2 \geq c d} \leq \exp(-c_0 d)$, 
$\prob{|\iprod{x_{n+1}}{\tilde x_{n+1}}| \geq c^2 d} \leq \exp(-c_0 d)$, and 
$\prob{\|x_{n+1}\|_3^3 \geq C d} \leq \exp(-c_0 d^{2/3})$, the last one applying \cite[Theorem 3.1]{kuchibhotla2022moving} (with $\alpha=2/3$).
Overall, we get 
$\prob{\calH^c} \leq \exp(-c_0 d^{2/3})$.

To bound the moment of the  correlation $\rho(x_{n+1},\tilde x_{n+1})$,
first we get from applying Bernstein's inequality 
$\Expect[|\Iprod{x_{n+1}}{\tilde x_{n+1}}|^k] \leq (Cdk)^{k/2} + (Ck)^k$.
Thus
$\Expect[|\rho(x_{n+1},\tilde x_{n+1})|^k]
\leq (Ck/\sqrt{d})^k + \exp(-c_0 d)$.
Overall, we get 
\[
\Expect[|\eta(x_{n+1},\tilde x_{n+1})|^k]
\leq 
(Ck/\sqrt{d})^k + \exp(-c_0 d^{2/3}).
\]
Plugging this into \prettyref{eq:KLdn3}, we conclude that $\KL(P_A\|Q_A) = O(\frac{n^3}{d})$, 
provided that $d = \Omega(n^3)$.    
\end{proof}

\section{Proof of \prettyref{lmm:poly-highsnr}}
\label{app:poly-highsnr}

Let $K(t) = \sum_{\ell\geq 0} a_\ell t^\ell$ denote the Taylor expansion of $K$, which is the CDF of the density $f$.
Then
\[
a_\ell = \frac{1}{\ell!} K^{(\ell)}(0)=  \frac{1}{\ell!} f^{(\ell-1)}(0).
\]

By assumption, the characteristic function 
$\phi(\omega) = \Expect_{X\sim f}[e^{i \omega X}]$
of $f$ satisfies
\begin{equation}
|\phi(\omega)| \leq C \exp(- c \omega^2/2).
    \label{eq:cf-assumption}
\end{equation}
Applying Fourier inversion 
\[
f(x) = \frac{1}{2\pi} \int_{-\infty}^{\infty} \phi(t) e^{-itx} dt
\]
and the dominated convergence theorem, we get
\[
f^{(k)}(x) = \frac{1}{2\pi} \int_{-\infty}^{\infty} (-i\omega)^k \phi(\omega) e^{-ix \omega} d\omega ,
\]
so
$|f^{(k)}(x)| \leq  \frac{C}{2\pi} \int_{-\infty}^{\infty} |\omega|^k e^{-c\omega^2/2} d\omega 
= \frac{C}{\sqrt{c}} \Expect[|N(0,1/c)|^k] \le Cc^{-(k+1)/2} (k-1)!! \le Cc^{-(k+1)/2} \sqrt{k!}$.
This leads to the desired bound 
\[
a_\ell \le 
Cc^{-\ell/2} \frac{1}{\sqrt{\ell!}}.
\]
In particular, the Taylor expansion $K(t) = \sum_{\ell\geq 0} a_\ell t^\ell$ converges everywhere.

Next, let $\tilde K(t) = \sum_{\ell\geq 0}^{L-1} a_\ell t^\ell$.
For any $t$, we have
$\tilde K(t)-K(t) = \frac{K^{(L)}(\xi)}{L!} \xi^L = \frac{f^{(L-1)}(\xi)}{L!} \xi^L
$ for some $\xi$ between $0$ and $t$.
Using the above derivative bound, we get for all $t$,
\begin{equation}
|\tilde K(t)-K(t)|
\leq  \frac{C c^{-L/2}}{\sqrt{L!}} |t|^L.
\label{eq:taylor-pointwise}
\end{equation}
In particular,
\begin{equation}
\sup_{t \in [-r,r]}|\tilde K(t)-K(t)|
\leq \frac{C c^{-L/2}}{\sqrt{L!}} r^L.
    \label{eq:taylor-uniform}
\end{equation}
By assumption, $c \exp(-Ct^2) \leq K(t) \leq 1-c \exp(-Ct^2)$.
Thus, $\tilde K(t) \in (0,1)$ for all $t \in (-r,r)$ as long as $L \geq C_0 r^2$ for a suitable constant $C_0$.

Using the fact that $\Iprod{x_1}{x_2}$ is $O(1/d)$-subgaussian so that
$\Expect[\Iprod{x_1}{x_2}^{2L}] \leq d^{-L} (C_1L)^L$ for some absolute constant $C_1$,
applying \prettyref{eq:taylor-pointwise} yields
\[
\Expect[
|\tilde K(r\Iprod{x_1}{x_2})-K(r\Iprod{x_1}{x_2})|^2]
\leq  \frac{C^2 c^{-L}}{L!} r^{2L} d^{-L} (C_1L)^L
\leq \pth{\frac{C_2 r^2}{d}}^L.
\]

\section{Non-universality of detection threshold for distance kernel} \label{app:non-universality}




In this appendix, we provide an example of a kernel for which the detection threshold is \textit{not universal} with respect to the latent point distribution.
Somewhat surprisingly, for distance kernels,
the detection threshold $d_* = n^{3/4}$ no longer holds if the latent points are isotropic Gaussian as opposed to spherical uniform. 
Correspondingly, the signed triangle count may not be the optimal test statistic. 

To this end, let us consider i.i.d.\ Gaussian latent points $x_1, \dots, x_n \sim \calN(0, \frac 1d I_d)$. 
Define the kernel $\tilde K : \reals^d \times \reals^d \to [0,1]$ by $\tilde K(x_i, x_j) \eqdef \gamma \exp(-\frac{\beta}{2}\|x_i-x_j\|^2)$ for constants $\gamma \in (0,1)$ and $\beta > 0$. This type of distance kernels arises in modeling communication networks motivated by the path loss in wireless communication \cite{waxman2002routing,dettmann2016random}.

Note that in the spherical counterpart with 
i.i.d.\ $x_1, \dots, x_n \sim \Unif(S^{d-1})$, 
we have $\tilde K(x_i, x_j)
= K(\Iprod{x_i}{x_j})$,
where $K(t) = \gamma \exp(-\beta(1-t))$.
This is a smooth inner product kernel, and \prettyref{thm:fixed} shows that the detection threshold is given by $d_* = n^{3/4}$, achieved by counting signed triangles.
However, Theorem~\ref{thm:signed-wedge-count-for-gaussian-points} next shows that for Gaussian latent points, the detection threshold is much higher, as counting signed wedges (paths of length two) succeeds whenever $d \ll n^{3/2}$.~\footnote{The phenomenon that counting signed wedges may outperform counting signed triangles has also been observed in \cite{bangachev2024fourier} for the hard RGG with Gaussian latent points. 
More specifically, in the regime $p \ll n^{-3/4}$, the detection threshold for low-degree polynomial test statistics (including, in particular, the signed wedge and signed triangle counts) is given by $d = n^{3/2+o(1)} p$. 
The signed wedge count succeeds under the nearly-optimal condition $d \le n^{3/2} p /(\log n)^5$, whereas the signed triangle count is only known to succeed when $d \le n^3p^3/(\log n)^5$.} 

Let $K_n$ denote the complete graph on $n$ vertices. For a subgraph $H \subset K_n$, we identify $H$ with its own edge set $E(H)$ when there is no ambiguity, and let $|H|$ denote the number of edges in $H$. 

\begin{theorem}
\label{thm:signed-wedge-count-for-gaussian-points} 
Let $P$ denote the model of the random geometric graph with conditionally independent edges $A_{ij} \sim \tilde K(x_i, x_j)$ given the latent points $x_1, \dots, x_n \sim \calN(0, \frac 1d I_d)$. Let $p \eqdef \E_P[A_{ij}]$ and $\bar A_{ij} = A_{ij} - p$. 
Let $Q = G(n,p)$. 
Define the signed wedge count to be 
$$
W(A) \eqdef \sum_{H \subset K_n, \, H \cong \mathrm{wedge}} \prod_{e \in H} \bar A_e.
$$
If $1 \ll d \ll n^{3/2}$, then $\TV(P,Q) = 1-o(1)$, and thresholding the signed wedge count $W(A)$ is a consistent test.
\end{theorem}

\begin{lemma}\label{lmm:subgraph_count_exp}
    Let $H\subset K_n$ be a subgraph.
    Let $L_H$ be its Laplacian, with eigenvalues $\lambda_i\geq 0$. 
    Then 
    \[
\E\qth{\prod_{e \in H}A_e}
= \gamma^{|H|} \prod_{i=1}^{v(H)}
\left(1 + \frac{\beta \lambda_i}{d}\right)^{-d/2} .
     \]
\end{lemma}
\begin{proof}
By definition,    $\sum_{(i,j) \in H} (x_i-x_j)^2 = x^\top L_H x$, so
    \begin{align*}
        \gamma^{-|H|} \E\qth{\prod_{e \in H}A_e}
= & \E\qth{\prod_{e \in H}\exp(-\frac{\beta}{2}\|x_i-x_j\|^2)}\\
= & \sth{\E_{x \sim N(0,I)} \qth{\exp(-\frac{\beta}{2d}x^\top L_H x)}}^d \\
= & \sth{\prod_{i=1}^{v(H)} \E\exp(- \frac{\beta\lambda_i}{2d} x_i^2)}^d \\
= & \prod_{i=1}^{v(H)} 
\left(1 + \frac{\beta \lambda_i}{d}\right)^{-d/2},
    \end{align*}
    applying the identity 
    that $\E[\exp(-\frac{\lambda}{2} N(0,1)^2)] = 1/\sqrt{1+\lambda}$.
\end{proof}


\begin{lemma}\label{lmm:signed_subgraph_count_exp}
    Let $H\subset K_n$ be a subgraph. 
    Then 
    \[
\E\qth{\prod_{e \in H} \bar{A}_e}
= p^{|H|}
\sum_{F \subset H}
\left(-1\right)^{|H \setminus F|}
\exp\left( 
\frac{\beta^2}{4d} \sum_{i=1}^{v(F)} (d_i^2(F) - d_i(F)) + O(d^{-2}) \right).
     \]
\end{lemma}

\begin{proof}
Note that 
\begin{align*}
\E\qth{\prod_{e \in H} \bar{A}_e}
&=\E\qth{\prod_{e \in H} (A_e -p) } \\
&=\E\qth{\sum_{F \subset H} \prod_{e \in E(F)} A_e (-p)^{|H\setminus F|} }\\
&=\sum_{F \subset H} \E\qth{\prod_{e \in E(F)} A_e} (-p)^{|H\setminus F|} \\
&= \sum_{F \subset H} \prod_{i=1}^{v(F)}
\left(1 + \frac{\beta \lambda_i(F)}{d}\right)^{-d/2} \gamma^{|F|} 
(-p)^{|H\setminus F|} ,
\end{align*}
where the last equality follows from Lemma~\ref{lmm:subgraph_count_exp}. 

Now, using the Taylor expansion of $\log(1+x)=x-x^2/2+O(x^3)$, we have
\begin{align*}
\prod_{i=1}^{v(F)}
\left(1 + \frac{\beta \lambda_i(F)}{d}\right)^{-d/2}
&=\exp \left( 
-\frac{d}{2} \left( \frac{\beta}{d} \sum_{i=1}^{v(F)} 
\lambda_i(F) - \frac{\beta^2}{2d^2 } \sum_{i=1}^{v(F)} 
\lambda_i^2(F) + O(d^{-3})\right)
\right) \\
& = 
\exp \left( 
 -\frac{\beta}{2} \sum_{i=1}^{v(F)} 
\lambda_i(F) + \frac{\beta^2}{4d } \sum_{i=1}^{v(F)} 
\lambda_i^2(F) + O(d^{-2})\right) .
\end{align*}
Furthermore, 
\begin{equation*}
p^{|F|}
= \gamma^{|F|} 
\left( 1 + \frac{2\beta}{d} \right)^{-d|F|/2} 
= \gamma^{|F|} 
\exp \left( 
 -\beta |F| + \frac{\beta^2}{d } |F| + O(d^{-2})\right) .
\end{equation*}

Combining the last three displayed equations yields that
\begin{align*}
&\E\qth{\prod_{e \in H} \bar{A}_e}\\
&=p^{|H|} \sum_{F \subset H}
(-1)^{|H\setminus F|}
\exp \left( 
 -\frac{\beta}{2} (\sum_{i=1}^{v(F)} 
\lambda_i(F) -2|F|) + \frac{\beta^2}{4d } (\sum_{i=1}^{v(F)} 
\lambda_i^2(F) -4|F|) + O(d^{-2})\right)\\
&=p^{|H|} \sum_{F \subset H}
(-1)^{|H\setminus F|}
\exp \left(  \frac{\beta^2}{4d} \sum_{i=1}^{v(F)} (d_i^2(F) - d_i(F)) + O(d^{-2})\right),
\end{align*}
where the last equality follows from the fact that $\sum_i \lambda_i(F) = \sum_i d_i(F) = 2|F|$ and that
\begin{align*}
\sum_{i} \lambda_i^2(F) =
\fnorm{L_F}^2
=\sum_{i} d_i^2(F) + \sum_i d_i(F) . 
\end{align*}
\end{proof}

\begin{lemma} \label{lem:expected-signed-wedge-or-not}
Let $H\subset K_n$ be a subgraph. 
If $H$ is a wedge, then
\[
\E\qth{\prod_{e \in H} \bar{A}_e}
= p^2 \frac{\beta^2}{2d} + O(d^{-2}) ,
\]
and otherwise, $\E\qth{\prod_{e \in H} \bar{A}_e} = O(d^{-2})$, where the hidden constant may depend on $H$, $\beta$, and $\gamma$. 
\end{lemma}

\begin{proof}
%
We use $\mathcal{W}(F)$ to denote the set of subgraphs of $F$ that are isomorphic to a wedge. 
Note that 
$$
\frac 12 \sum_{i=1}^{v(F)} (d_i^2(F) - d_i(F)) = \sum_{i=1}^{v(F)} \binom{d_i(F)}{2} = |\mathcal{W}(F)| .
$$ 
As a result, 
\begin{equation*}
p^{-|H|} \E\qth{\prod_{e \in H} \bar{A}_e} 
=\sum_{F \subset H}
\left(-1\right)^{|H \setminus F|}
\left(1 + 
\frac{\beta^2}{2d} |\mathcal{W}(F)| + O(d^{-2}) \right).
\end{equation*}
It is clear that $\sum_{F \subset H}
\left(-1\right)^{|H \setminus F|} = 0$. 
Moreover, we can write
\begin{align*}
\sum_{F \subset H}
\left(-1\right)^{|H \setminus F|} |\mathcal{W}(F)|
= \sum_{F \subset H} 
\left(-1\right)^{|H \setminus F|} \sum_{W \in \mathcal{W}(H)} \mathbf{1}\{W \subset F\} 
= \sum_{W \in \mathcal{W}(H)} \sum_{F: W \subset F \subset H}
\left(-1\right)^{|H \setminus F|} .
\end{align*}
The inner sum $\sum_{F: W \subset F \subset H}
\left(-1\right)^{|H \setminus F|}$ is zero unless $W = H$. 
Therefore, if $H$ is not a wedge, then
$$
\E\qth{\prod_{e \in H} \bar{A}_e} = p^{|H|} \sum_{F \subset H}
\left(-1\right)^{|H \setminus F|}  O(d^{-2}) = O(d^{-2}) .
$$
If $H$ is a wedge, then 
\begin{equation*}
\E\qth{\prod_{e \in H} \bar{A}_e} 
= p^{|H|} \frac{\beta^2}{2d} \sum_{F \subset H}
\left(-1\right)^{|H \setminus F|} |\mathcal{W}(F)| + O(d^{-2}) 
= p^2 \frac{\beta^2}{2d} + O(d^{-2}) .
\end{equation*}
\end{proof}

\begin{proof}[Proof of Theorem~\ref{thm:signed-wedge-count-for-gaussian-points}]
For the mean $\Expect_P [W(A)]$, 
it follows from \prettyref{lem:expected-signed-wedge-or-not} that
$$
\Expect_P [W(A)] = 3 \binom{n}{3} \left( p^2 \frac{\beta^2}{2d} + O(d^{-2}) \right) \asymp \frac{n^3}{d} .
$$
Next, we consider
\begin{align*}
\Var_P (W(A)) &= \sum_{H, H'} \Cov_P\bigg( \prod_{e \in H} \bar A_e , \prod_{e \in H'} \bar A_e \bigg) \\
&= \sum_{H, H'} \Bigg( \Expect_P \prod_{e \in H \cup H'} \bar A_e - \bigg( \Expect \prod_{e \in H} \bar A_e \bigg)^2 \Bigg) \\
&\le \sum_{H, H' : V(H) \cap V(H') \ne \varnothing} \Expect_P \prod_{e \in H \cup H'} \bar A_e ,
\end{align*}
where the sum is over two wedges $H$ and $H'$ in $K_n$, and $H \cup H'$ denotes the \emph{multigraph} obtained from taking the union. 
In what follows, we bound the contributions of all the summands according to the possible shape of $H \cup H'$. 
\begin{enumerate}
\item
$H \cup H'$ does not have any double edge: 
By \prettyref{lem:expected-signed-wedge-or-not}, we have
$\Expect_P \prod_{e \in H \cup H'} \bar A_e = O(d^{-2})$.
Since $V(H) \cap V(H') \ne \varnothing$, there are at most $O(n^5)$ choices of $(H,H')$.
Therefore, the contribution of all these terms is at most $O(n^5/d^2)$. 






\item 
$H \cup H'$ has one double edge: In this case, $H \cup H'$ may be a path, a star, or a cycle with exactly one double edge. 
The proof is similar for the three subcases, so we only consider the following scenario.
Let the vertices of $H \cup H'$ be $1,2,3,4$, with $(1,2)$ and $(3,4)$ being single edges, and $(2,3)$ being a double edge. 
Write $x = (x_1, x_2, x_3, x_4)$. 
Let $\tilde K(x_i, x_j) = \gamma \exp( - \frac{\beta}{2} \|x_i - x_j\|_2^2 )$. 
Then we have 
\begin{equation*}
\E_P[\bar A_{ij}^2 \mid x] = \tilde K(x_i, x_j) (1-p)^2 + (1 - \tilde K(x_i, x_j)) (-p)^2 = - p^2 + p + (1-2p) (\tilde K(x_i, x_j) - p) .
\label{eq:expected-signed-double-edge}
\end{equation*}
We obtain
\begin{align*}
\Expect_P \prod_{e \in H \cup H'} \bar A_e
&= \Expect\left[ \E_P[\bar A_{12} \mid x] \cdot \E_P[\bar A_{23}^2 \mid x] \cdot \E_P[\bar A_{34} \mid x] \right]  \\
&= \Expect[(\tilde K(x_1, x_2) - p) (- p^2 + p + (1-2p) (\tilde K(x_2, x_3) - p)) (\tilde K(x_3, x_4) - p)] \\
&= (-p^2+p) \Expect_P[ \bar A_{12} \bar A_{34} ] + (1-2p) \Expect_P[\bar A_{12} \bar A_{23} \bar A_{34}] 
= O(d^{-2}) 
\end{align*}
by \prettyref{lem:expected-signed-wedge-or-not}. 
There are at most $O(n^4)$ such choices of $(H,H')$, so the contribution of these terms is at most $O(n^4/d^2)$. 



\item 
$H \cup H'$ contains two double edges: In this case, $H \cup H'$ must be a wedge with two double edges. These terms contribute at most $O(n^3)$. 
\end{enumerate}
To conclude, we get that the variance under soft RGG is $O(n^5/d^2)$. Moreover, $\Expect_Q[W(A)] = 0$ and $\Var_Q(W(A)) = O(n^3)$. 

For the signed wedge count to succeed in distinguishing $P$ from $Q$, it suffices to have
$$
\Expect_P[W(A)] - \Expect_Q[W(A)] \gg \sqrt{\Var_P(W(A)) \lor \Var_Q(W(A))} ,
$$
which holds if 
$$
\frac{n^3}{d} \gg \max\left\{ \sqrt{\frac{n^5}{d^2}} , \sqrt{n^3} \right\} .
$$
This reduces to $d \ll n^{1.5}$. 
\end{proof}

\section*{Acknowledgment}
C.~Mao is supported in part by an NSF CAREER award DMS-2338062. J.~Xu is supported in part by an NSF CAREER award CCF-2144593.
This research is based in part upon work supported by the National Science Foundation under Grant No.~DMS-1928930, while Y.~Wu and J.~Xu were in residence at the Simons Laufer Mathematical Sciences Institute in Berkeley, California, during the Spring 2025 semester.
Part of this research was performed while 
C.~Mao and J.~Xu were visiting the Institute for Mathematical and Statistical Innovation (IMSI), which is supported by the National Science Foundation under Grant No. DMS-2425650.

\bibliographystyle{alpha}
\bibliography{ref}
	
\end{document}